\documentclass[11pt]{amsart}
\usepackage[leqno]{amsmath}
\usepackage{amssymb}
\usepackage[foot]{amsaddr}

\usepackage{enumerate}
\usepackage{color}
\usepackage{subfigure}
\usepackage{graphicx}
\usepackage[all]{xy}
\usepackage[colorlinks=true,linkcolor=NavyBlue,urlcolor=RoyalBlue,citecolor=PineGreen]{hyperref}
\usepackage[usenames,dvipsnames]{xcolor}
\usepackage{aliascnt}
\usepackage{cleveref}

\usepackage{microtype}
\usepackage{centernot}
\usepackage{stmaryrd}

\usepackage{bm}

\numberwithin{equation}{section}
\numberwithin{figure}{section}

\usepackage{setspace}

\newtheorem{theorem}{Theorem}[section]

\newaliascnt{remark}{theorem}
\newtheorem{remark}[remark]{Remark}
\aliascntresetthe{remark}
\newaliascnt{lemma}{theorem}
\newtheorem{lemma}[lemma]{Lemma}
\aliascntresetthe{lemma}
\newaliascnt{proposition}{theorem}
\newtheorem{proposition}[proposition]{Proposition}
\aliascntresetthe{proposition}
\newaliascnt{corollary}{theorem}

\aliascntresetthe{corollary}
\newaliascnt{problem}{theorem}

\aliascntresetthe{problem}
\newaliascnt{definition}{theorem}
\newtheorem{definition}[definition]{Definition}
\aliascntresetthe{definition}
\newaliascnt{example}{theorem}

\aliascntresetthe{example}
\newaliascnt{assumption}{theorem}

\aliascntresetthe{assumption}

\crefname{theorem}{Theorem}{Theorems}
\crefname{lemma}{Lemma}{Lemmas}
\crefname{proposition}{Proposition}{Propositions}
\crefname{corollary}{Corollary}{Corollaries}
\crefname{definition}{Definition}{Definitions}
\crefname{remark}{Remark}{Remarks}
\crefname{footnote}{footnote}{footnotes}
\Crefname{footnote}{Footnote}{Footnotes}
\Crefname{theorem}{Theorem}{Theorems}
\Crefname{lemma}{Lemma}{Lemmas}
\Crefname{proposition}{Proposition}{Propositions}
\Crefname{corollary}{Corollary}{Corollaries}
\Crefname{definition}{Definition}{Definitions}
\Crefname{remark}{Remark}{Remarks}

\newcounter{proofstep}
\newcounter{stepuid}

\crefname{proofstep}{Step}{Steps}
\Crefname{proofstep}{Step}{Steps}

\newcommand{\steplabel}[2]{
  \stepcounter{stepuid}
  \setcounter{proofstep}{\numexpr#1-1\relax}
  \refstepcounter{proofstep}\label{#2}}

\newcommand{\C}{\mathbf{C}}

\newcommand{\E}{\mathbf{E}}
\newcommand{\F}{\mathbf{F}}

\newcommand{\N}{\mathbf{N}}
\newcommand{\Z}{\mathbf{Z}}
\newcommand{\p}{\mathbf{P}}

\newcommand{\s}{\mathbf{S}}
\newcommand{\T}{\mathbf{T}}

\newcommand{\CD}{\mathcal {D}}

\newcommand{\CF}{\mathcal {F}}

\newcommand{\CL}{\mathcal {L}}

\newcommand{\CS}{\mathcal {S}}
\newcommand{\CT}{\mathcal {T}}

\newcommand{\CG}{\mathcal {G}}

\newcommand{\Dis}{\mathrm{Dis}}

\newcommand{\Bin}{\mathrm{Bin}}
\newcommand{\Ber}{\mathrm{Ber}}
\newcommand{\Gam}{\mathrm{Gamma}}

\newcommand{\one}{{\bf 1}}

\newcommand{\wt}{\widetilde}

\newcommand{\ol}{\overline}

\newcommand{\giv}{\,|\,}

\newcommand{\eps}{\epsilon}

\begin{document}

\title{On the $\ell^2$ distortion of random triangulations}

\author{Jason Miller$^1$}
\address{University of Cambridge}
\email{$^1$\url{jpm205@cam.ac.uk}}
\author{Julian Ransford$^2$}
\email{$^2$\url{jr2012@cam.ac.uk}}
\author{Fredy Yip$^3$}
\email{$^3$\url{fy276@cam.ac.uk}}

\begin{abstract}
For each $n \in \N$, let $\CT_n$ be a uniformly random (rooted, Type~I) triangulation of the sphere with $n$ vertices, viewed as a metric space equipped with its graph distance. We show that for every $\delta>0$, with probability tending to $1$ as $n \to \infty$, every embedding of $\CT_n$ into a separable Hilbert space has distortion at least $(\log n)^{1/4-\delta}$.
\end{abstract}

\date{}
\maketitle

\setcounter{tocdepth}{1}
\tableofcontents

\parindent 0 pt
\setlength{\parskip}{0.20cm plus1mm minus1mm}
 
\section{Introduction}
\label{sec:intro}

\subsection{Overview}
\label{subsec:overview}

Suppose that $(X,d)$ is a finite metric space. The \emph{$\ell^2$ distortion} $\Dis(X)$ of $X$ is the infimum of the set of numbers $L \ge 1$ for which there exists a function $f:X \to \ell^2(\N)$ such that
\[ d(x,y) \le \| f(x) - f(y) \|_2 \le L d(x,y),\]
for all $x,y \in X$ (here and below we abbreviate $\ell^2 = \ell^2(\N)$). Since $X$ is finite, a compactness argument shows that this infimum is attained. Consequently, $X$ can be isometrically embedded into $\ell^2$ if and only if $\Dis(X)=1$, and so the distortion is a measure of how close $X$ is to being isometrically embeddable into $\ell^2$.

It was shown by Bourgain \cite{bourgain1985lip} that if $n = |X| \ge 2$, then $\Dis(X) = O(\log n)$ (that is, there is a universal constant $C>0$ such that for every $n$-point metric space $X$ with $n \ge 2$, $\Dis(X) \le C \log n$). Rao \cite{rao1999planar} later showed that if $X$ is the vertex set of a connected planar graph with $n \ge 2$ vertices, equipped with its shortest-path metric, then $\Dis(X) = O(\sqrt{\log n})$.  Specific examples of planar graphs for which the asymptotics of $\Dis(X)$ are known include:
\begin{itemize}
\item If $B_n$ is the complete binary tree with $n$ vertices (so that $n+1$ is a power of $2$), then $\Dis(B_n) = \Theta(\sqrt{\log \log n})$ as shown by Bourgain \cite{bourgain1986trees}.  The matching upper bound was extended to arbitrary trees by Matou\v{s}ek \cite{matousek1999trees}, who showed that every $n$-vertex tree embeds into $\ell^2$ with distortion $O(\sqrt{\log \log n})$ (the lower bound of course does not extend, e.g., a path embeds isometrically).
\item If $D_k$ is the level-$k$ diamond graph (see \Cref{fig:diamond_laakso}), which has $n_k$ vertices with $\log n_k = \Theta(k)$, then $\Dis(D_k) = \Theta(\sqrt{k}) = \Theta(\sqrt{\log n_k})$, showing that Rao's upper bound is sharp \cite{nr2003diamond}.
\item If $L_k$ is the level-$k$ Laakso graph (see \Cref{fig:diamond_laakso}), which has $n_k$ vertices with $\log n_k = \Theta(k)$, then $\Dis(L_k) = \Theta(\sqrt{k}) = \Theta(\sqrt{\log n_k})$, showing that there is a sequence of graphs whose doubling constants are uniformly bounded for which Rao's bound is attained up to a constant factor \cite{gkl2003laakso}.
\end{itemize}
Both the diamond and Laakso graphs are constructed via an iterative procedure, illustrated in \Cref{fig:diamond_laakso}, and are ``fractal-like''.

\begin{figure}
\includegraphics[scale=0.8]{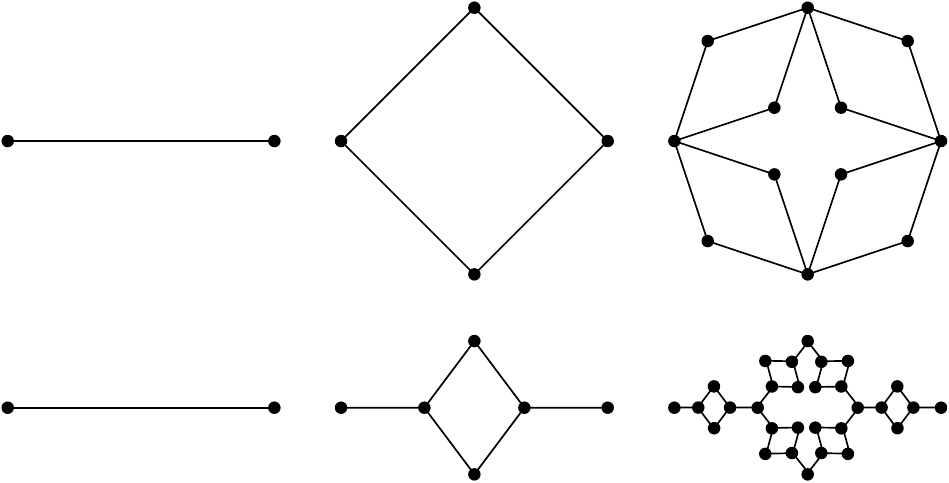}
\caption{\label{fig:diamond_laakso} {\bf Top:} The diamond graphs $D_0$, $D_1$, $D_2$.  {\bf Bottom:} The Laakso graphs $L_0$, $L_1$, $L_2$.  (The illustrations are not exactly to scale.) The graphs $D_0$ and $L_0$ each consist of a single edge, and given $D_{k-1}$ (resp.\ $L_{k-1}$), the graph $D_k$ (resp.\ $L_k$) is obtained by replacing each edge of $D_{k-1}$ (resp.\ $L_{k-1}$) with a copy of $D_1$ (resp.\ $L_1$), the two endpoints of the edge being identified with the leftmost and rightmost vertices of this copy (as drawn in the figure).}
\end{figure}

The purpose of this work is to study how the $\ell^2$ distortion of a \emph{typical} planar map (defined below) behaves as its size $n$ grows.  Are typical planar maps tree-like, so that the typical value of the distortion is $O(\sqrt{\log \log n})$, or are they more like the diamond or Laakso graphs and have typical distortion $\Theta(\sqrt{\log n})$?  We will use uniformly random triangulations of the sphere as our model of a typical planar map (we emphasize that this is not the same as the uniformly random planar \emph{graph} with $n$ vertices, which is another natural notion of a typical planar graph). We will establish that the $\ell^2$ distortion grows much faster than in the case of the binary tree. However, our lower bound of $(\log n)^{1/4-o(1)}$ is of strictly smaller order than Rao's upper bound of $O(\sqrt{\log n})$, and we leave open whether the true typical growth rate is strictly slower than Rao's bound, or matches it.

A triangulation is a type of \emph{planar map}, which is a finite connected multigraph (loops and multiple edges being allowed) together with an embedding into the sphere in which edges intersect only at common endpoints.  Two planar maps are considered to be equivalent if there exists an orientation-preserving homeomorphism of the sphere which takes one to the other (and which, for rooted maps as defined in \Cref{subsec:main_result}, takes the root edge of one to the root edge of the other, preserving its orientation). The \emph{faces} of a planar map are the connected components of the complement of the union of its edges and vertices. A planar map is a \emph{triangulation} if every face has degree three, meaning that it is incident to three edge-sides, an edge being counted twice if the same face lies on both of its sides.  Since there are only finitely many rooted triangulations with a fixed number of vertices, one can naturally consider the uniform probability measure on these triangulations.  This will be the model to which our main result applies.  We expect that our main result also holds for random $p$-angulations for any fixed $p \ge 3$.

The study of random planar maps is rooted in the work of Tutte \cite{tutte1962census1,tutte1962census2,tutte1962census3,tutte1963census4} and Mullin \cite{m1967treebijection}, which focused on the enumeration of planar maps.  A major breakthrough was made independently by Le Gall \cite{lg2013bm} and Miermont \cite{miermont2013bm}, who proved that uniformly random $p$-angulations with $n$ faces (with $p=3$ or $p$ even in \cite{lg2013bm} and with $p=4$ in \cite{miermont2013bm}), with distances rescaled by $c_p n^{-1/4}$ for a suitable constant $c_p>0$, converge in distribution in the Gromov--Hausdorff topology to a random compact metric space called the \emph{Brownian sphere} (also known as the Brownian map).  More recently, a number of other random planar map models have been studied with bijective techniques (see, e.g., \cite{kmsw2019bipolar, lsw2024schnyder, she2016qginventory}). We refer to \cite{miermont2014aspects} for a survey on the subject, and to \cite{adj1997physics} for a review from the physics literature.

Motivated by the goal of linking statistical physics models defined on random planar maps with those defined on planar lattices, a number of recent works have established the convergence of random planar maps and related discrete random surfaces to Liouville quantum gravity, when these are embedded into the plane via conformal-type embeddings (see, e.g., \cite{gms2020pvbm, gms2021matedcrt,hs2023cardy}).  This article approaches the subject from a different angle and is, to the best of our knowledge, the first to consider bi-Lipschitz embeddings of random planar maps into Hilbert space.  One motivation for this is that $\ell^2$ embeddings are related to $\ell^1$ embeddings (since every finite subset of $\ell^2$ embeds isometrically into $\ell^1$, the $\ell^1$ distortion of a finite metric space is always at most its $\ell^2$ distortion) which, in turn, are closely connected to various algorithmic applications in computer science \cite{llr1995algorithms}.  For instance, the planar case of a conjecture of Gupta, Newman, Rabinovich, and Sinclair \cite{gnrs2004cuts} (the so-called GNRS conjecture) asserts that there exists a constant $C > 0$ such that the shortest-path metric of every planar graph with non-negative edge weights can be embedded into $\ell^1$ with distortion at most~$C$ (i.e., $C$ does not depend on the graph).

\subsection{Main result}
\label{subsec:main_result}

We now turn to stating our main result.  Throughout, all planar maps are rooted, meaning that they are equipped with a distinguished oriented edge, called the root edge. The initial vertex of this edge is called the root vertex (or simply the root). We shall only consider Type~I triangulations, meaning that loops and multiple edges are allowed.  Recall that Type~II triangulations are those with no loops (but possibly with multiple edges) and that Type~III triangulations are those which are simple. Our main result is the following.

\begin{theorem}
\label{thm:main}
Let $\CT_n$ be a uniformly random (rooted Type~I) triangulation of the sphere with $n$ vertices, viewed as a metric space equipped with its graph distance. Then for any $\delta>0$,
\[\p(\Dis(\CT_n) \ge (\log n)^{1/4-\delta}) \to 1 \quad\text{as}\quad n \to \infty.\]
\end{theorem}

\subsection{Outline}
\label{subsec:outline}

We now provide a detailed outline of the proof of \Cref{thm:main}.  The main step is to show the existence of Laakso-like structures (recall \Cref{fig:diamond_laakso}) in the infinite triangulations which arise as (iterated) local limits of $\CT_n$, suitably re-rooted.  To make this precise, we first recall the constructions of the UIPT \cite{as2003uipt} and the LHPT \cite{cl2019fpp}, and we briefly describe the UIHPT \cite{ac2015uihpt}.

\subsubsection{The UIPT} Suppose that $\CT_n$ is a uniformly random triangulation of the sphere with $n$ vertices which, we recall, comes equipped with a root.  For an integer $r \ge 1$, the ball $B_r(\CT_n)$ of radius $r$ is the planar map consisting of all of the faces which are incident to at least one vertex at distance at most $r-1$ from the root, together with all of the edges and vertices which are incident to such a face.  (Note that the vertex set of $B_r(\CT_n)$ is then exactly the set of vertices at distance at most $r$ from the root.  It is important here to include the edges of the faces in question rather than only those edges which are incident to a vertex at distance at most $r-1$, since a face of $B_r(\CT_n)$ can have an edge both of whose endpoints are at distance exactly $r$.) Now let $o_n$ be a vertex selected uniformly at random among the vertices of $\CT_n$ distinct from the root, independently of everything else.  The \emph{hull} of radius $r$, denoted by $B_r^{\bullet}(\CT_n)$, is the union of $B_r(\CT_n)$ and all connected components of the complement of $B_r(\CT_n)$ in the sphere that do not contain $o_n$ (if $o_n \in B_r(\CT_n)$, then the hull is defined to be the entire planar map).  We write $\partial B_r^{\bullet}(\cdot)$ for the boundary of the hull.  It is a simple cycle (unless the hull is the entire planar map), and we refer to it as such below. Then there exists a random infinite triangulation $\CT_{\infty}$ of the plane such that for any $r \ge 1$ and any finite planar map $\mathfrak{M}$,
\[\lim_{n \to \infty} \p(B_r^{\bullet}(\CT_n)=\mathfrak{M})=\p(B_r^{\bullet}(\CT_{\infty})=\mathfrak{M}).\]
See \cite[Section 2.4]{cl2019fpp} (see just below for the definition of the hull of an infinite triangulation).  The triangulation $\CT_{\infty}$ is called the \emph{Uniform Infinite Planar Triangulation} (UIPT). It was first constructed by Angel and Schramm in \cite{as2003uipt} for Type~II and Type~III triangulations.  See \cite[Proposition 6.2]{stephenson2018uipt} for the case of Type~I triangulations. In many situations, the UIPT is easier to work with than $\CT_n$. Indeed, the UIPT can be explored face by face in a Markovian manner without needing to condition on the total volume, using a procedure known as the \emph{peeling process}. See \cite{curien2023peeling} for a detailed survey on the peeling process. For our purposes, the crucial property of the UIPT is that the hull annuli $B_s^{\bullet}(\CT_{\infty})\setminus B_r^{\bullet}(\CT_{\infty})$ for $s>r$ can be encoded by a cyclic sequence of trees with an explicit branching structure, together with independent Boltzmann triangulations filling the regions between the branches of these trees.  This encoding is called the \emph{skeleton decomposition} \cite{krikun2005uipt,cl2019fpp}.

\subsubsection{The LHPT and UIHPT} Beyond the UIPT, there are two further infinite triangulations which describe the UIPT near the boundary of a large hull.  The first of them, the LHPT, will play an important role in this paper. Analogously to the finite case, for an infinite triangulation, the hull of radius $r$ is the union of the ball of radius $r$ and all finite connected components of the complement of the ball (it can be shown that the UIPT is almost surely one-ended, i.e., for every $r \ge 1$, the complement of the ball of radius $r$ has a unique infinite component). Let $\wt{B}_r^{\bullet}(\CT_{\infty})$ be the hull of radius $r$ in the UIPT, re-rooted at a uniformly random edge on the boundary of the hull, oriented so that the complement is on the left of the edge. Then there exists a random infinite triangulation of the half-plane $\CL$ such that
\[\wt{B}_r^{\bullet}(\CT_{\infty}) \to \CL,\]
as $r \to \infty$, in distribution, in the sense of local limits of rooted planar maps. See \cite[Sections 3.1 and 3.2]{cl2019fpp} for more details. The triangulation $\CL$ is called the \emph{Lower Half-Plane Triangulation} (LHPT). It is a triangulation of the half-plane, in the sense that all faces are triangles except for one (the boundary face), which has infinite degree.  By contrast, the Uniform Infinite Half-Plane Triangulation (UIHPT) \cite{ac2015uihpt} arises as the $r \to \infty$ local limit of the complement of $\wt{B}_r^{\bullet}(\CT_{\infty})$ in the UIPT, re-rooted at a uniformly random edge on the boundary of the hull.  By the spatial Markov property of the UIPT (conditionally on $B_r^{\bullet}(\CT_\infty)$, its complement is a UIPT of the polygon bounded by $\partial B_r^{\bullet}(\CT_\infty)$), and since the length of $\partial B_r^{\bullet}(\CT_\infty)$ tends to infinity in probability as $r \to \infty$, the entire local limit as $r \to \infty$ of the UIPT near a uniformly random edge of the boundary of $B_r^{\bullet}(\CT_{\infty})$ can be described by gluing \emph{independent} copies of the LHPT and the UIHPT along their common infinite boundary.

The LHPT has an explicit description in terms of a two-sided infinite sequence of i.i.d.\ critical Galton--Watson trees and of independent Boltzmann triangulations, which we now explain. For a non-negative integer $k$, let
\begin{equation}
\label{eq:theta_distribution}
\theta(k)=\frac{3 \binom{2k}{k}}{2 \cdot 4^k (k+1)(k+2)}.
\end{equation}
The function $\theta$ is a probability distribution on the non-negative integers, and it has mean~$1$, so that the associated Galton--Watson trees are critical. Since $\theta$ is also non-degenerate, they are therefore almost surely finite.  One can check from~\eqref{eq:theta_distribution} that $\theta(k) \sim \frac{3}{2\sqrt{\pi}}k^{-5/2}$ as $k \to \infty$.  This is the offspring distribution which appears in the skeleton decomposition of the UIPT and the LHPT; see \cite[Section 2.2]{cl2019fpp}.

Now let $\{\CG_j\}_{j \in \Z}$ be a collection of i.i.d.\ Galton--Watson trees with offspring distribution $\theta$. We embed these trees in the right half-plane as follows. The vertex set of the skeleton will be the points of the form $(i,j+1/2)$, where $i,j \in \Z$ and $i \ge 0$ (these half-integer points are \emph{not} vertices of the triangulation which we construct below. The vertices of the skeleton at generation $i$ are in bijection with the vertices $(i,j)$, $j \in \Z$, of the line $\{i\} \times \Z$ via $(i,j+1/2) \leftrightarrow (i,j)$, and we pass freely between the two). The root of the tree $\CG_j$ is placed at $(0,j+1/2)$. More generally, for each $i \ge 0$, the vertices at generation $i$ of all of the trees are placed bijectively on the points of the line $\{i\} \times (\Z+\frac{1}{2})$ so that, listed from bottom to top, they appear in the order induced by the ordering $\dots,\CG_{-1},\CG_0,\CG_1,\dots$ of the trees and by the planar order within each tree, and so that those of the trees $\CG_j$ with $j \ge 0$ occupy exactly the half-line $\{i\}\times (\Z_{\ge 0} +\frac{1}{2})$ (hence those of the trees $\CG_j$ with $j<0$ occupy exactly the negative half-line $\{i\}\times (\Z_{<0} +\frac{1}{2})$).  This is well defined because almost surely, for each $i \ge 0$, infinitely many of the trees $(\CG_j)_{j \ge 0}$ and infinitely many of the trees $(\CG_j)_{j<0}$ survive to generation $i$, so that each of these two half-lines is entirely filled (see \Cref{fig:LHPT}). These trees will be referred to as the \emph{skeleton} of the LHPT.

Given the skeleton, we now build the triangulation as follows. First, for each vertex $(i,j)$ with $i,j \in \Z$ and $i \ge 0$, we place a triangle whose vertices are $(i,j)$, $(i,j+1)$ and $(i+1,k)$ where $k$ is the minimal integer such that the vertex $(i+1,k+1/2)$ in the skeleton is a child of a vertex $(i,j'+1/2)$ for some $j'>j$ (such a $k$ exists almost surely by the previous paragraph; when $(i,j+1/2)$ has at least one child, $(i+1,k-1/2)$ is its last child).  These are the \emph{rightward triangles} of \Cref{fig:LHPT}. Then each vertex $(i,j)$ will have an empty \emph{slot} to its right, which is a polygon with a certain number of sides $p \ge 2$ (namely, $p$ is two plus the number of children of $(i,j+1/2)$.  In particular, when $(i,j+1/2)$ has no children, the slot is a $2$-gon bounded by two parallel edges from $(i,j)$ to $(i+1,k)$). All of these slots are then filled with independent Boltzmann triangulations of the corresponding polygons, which are also independent of the skeleton (see \cite[Section 2.1]{cl2019fpp} for a definition of Boltzmann triangulations). As in \cite{cl2019fpp}, a $2$-gon slot may be filled with the \emph{edge-triangulation}, which consists of a single edge, in which case the two edges bounding the slot are glued into a single edge.  In this way, the slots at time $i$ are in bijection with the vertices of the skeleton at generation $i$, the slot to the right of $(i,j)$ corresponding to $(i,j+1/2)$, and we refer to a slot by the skeleton vertex which indexes it. The vertex set of the LHPT therefore consists of the points $(i,j)$ with $i \ge 0$ and $j \in \Z$ together with the inner vertices of the Boltzmann triangulations which fill the slots, and its boundary is the bi-infinite path $\{0\} \times \Z$ and its root edge is the edge from $(0,0)$ to $(0,1)$, which has the complement of $\CL$ on its left, as in the limit defining $\CL$.  Throughout the paper we refer to the first coordinate $i$ of a point $(i,j)$ as its \emph{time} and to the second coordinate $j$ as its \emph{height}.  We shall use repeatedly that, almost surely, any path from a vertex of $\{i\} \times \Z$ to a vertex of $\{i'\} \times \Z$ has length at least $|i-i'|$. This follows from the definition of $\CL$ as a limit. Again, see \Cref{fig:LHPT} for an illustration of the LHPT.

\begin{figure}
\includegraphics[scale=0.8]{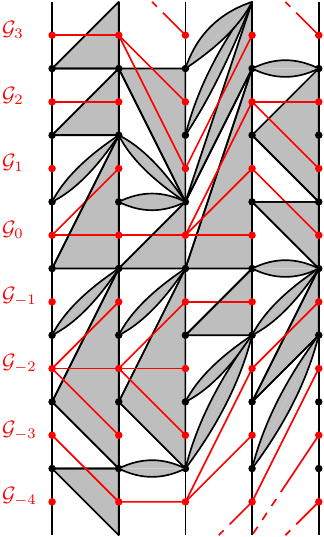}
\caption{\label{fig:LHPT} The LHPT. In red are the trees of the skeleton. In white are the rightward triangles. The remaining gray slots are filled in by independent Boltzmann triangulations.}
\end{figure}

We remark that the usual convention is to embed the LHPT in the \emph{lower} half-plane, as the name would suggest. However, we have chosen to depart from this convention and embed it in the right half-plane so that the trees in the skeleton, viewed as Galton--Watson processes indexed by time, evolve in the standard direction of time.

All distances in the LHPT, as well as the times and heights of its vertices, are necessarily integer-valued, so it will be tacitly understood throughout that every real-valued parameter representing a distance, a height or a time needs to be rounded up or down to an integer (it will never matter which of the two is chosen). We almost always omit floors and ceilings to ease notation.

\subsubsection{Geodesic structure of the LHPT}\label{subsubsec:geodesic_structure} We write $d_{\mathrm{gr}}$ for the graph distance on the vertex set of a planar map, adding a superscript to indicate the map when it is not clear from the context.  In particular, $d_{\mathrm{gr}}^{\CL}$ denotes the graph distance in the LHPT.  All of the distances considered below are graph distances in this sense.  By construction, each vertex $(i,j)$ has at least one edge connecting it to a vertex on the line $\{i+1\} \times \Z$ (namely, the sides from $(i,j)$ of the rightward triangles at $(i,j-1)$ and at $(i,j)$. These two edges may be parallel, in which case they bound a $2$-gon slot and may be glued into a single edge, as explained above). The path obtained by moving right repeatedly while always taking the bottom-most edge is called the \emph{bottom-most geodesic} started from $(i,j)$. It is indeed an infinite geodesic: it hits the line $\{i'\} \times \Z$ after exactly $i'-i$ steps for every $i' \ge i$, while any path from a vertex on $\{i\} \times \Z$ to a vertex on $\{i'\} \times \Z$ has length at least $i'-i$ as observed above.  If $b$ is any point lying on this path, then the restriction of this path from $(i,j)$ to $b$ is a geodesic. We will call this restriction a \emph{geodesic segment} from $(i,j)$ to $b$, and its length is the distance from $(i,j)$ to $b$ (we will usually omit the word ``geodesic'' and simply call this a segment). Note that from the way that we chose to embed the trees $\CG_j$, the bottom-most geodesic started from $(0,0)$ is always the path $((0,0),(1,0),(2,0), \dots)$.

For two points $(i,j)$ and $(i,j+k)$ with $k \ge 1$, their bottom-most geodesics will coalesce after $\ell$ steps (i.e., at time $i+\ell$) if and only if the forest consisting of the $k$ subtrees of the skeleton rooted at $(i,j+1/2), \dots, (i,j+k-1/2)$ (which, by the branching property, are i.i.d.\ Galton--Watson trees with offspring distribution $\theta$) goes extinct at generation~$\ell$. Since the offspring distribution $\theta$ is critical and non-degenerate, almost surely, each tree of the skeleton is finite, and so for any two points, their bottom-most geodesics will eventually coalesce.

\begin{definition}[Bubble]\label{def:bubble}
Let $(i,j)$ be a vertex of $\CL$ and let $k \ge 1$ be an integer.  The \emph{bubble} started at $(i,j)$ with \emph{initial height} $k$ is the sub-triangulation of $\CL$ lying in between the bottom-most geodesics started from $(i,j)$ and from $(i,j+k)$, up to the time at which these two geodesics coalesce.  We call the bottom-most geodesic started from $(i,j)$ the \emph{bottom geodesic} of the bubble, the bottom-most geodesic started from $(i,j+k)$ its \emph{top geodesic}, and $i$ its \emph{starting time}.  For $t \ge 0$, the \emph{height of the bubble at time $t$} is the difference between the heights of its top and its bottom geodesic at time $i+t$. The \emph{lifetime} of the bubble is the first time at which its height vanishes, that is, the time taken by the two geodesics to coalesce.
\end{definition}

The bubble started at $(i,j)$ with initial height $k$ is encoded by the forest of the $k$ subtrees of the skeleton rooted at $(i,j+1/2), \dots, (i,j+k-1/2)$ considered above, together with the Boltzmann triangulations filling the slots indexed by the vertices of this forest.  Indeed, the slots lying between the two geodesics at time $i+t$ are exactly the individuals at generation $t$ of that forest, so that the height of the bubble at time $t$ is the population at generation $t$ of the forest.  As a process in $t$, the height of a bubble with initial height $k$ is therefore a Galton--Watson process with offspring distribution $\theta$ started from $k$, and the lifetime of the bubble is the extinction time of that process.  We pass freely between these two descriptions of a bubble, as a region of the LHPT and as a forest of $k$ i.i.d.\ Galton--Watson trees.

\subsubsection{Stochastic thin Laakso substructures constructed from LHPT geodesics} We are now ready to give an overview of how the proof will go. We will construct thin Laakso-like structures (in the spirit of the thin Laakso substructures of \cite{bgn2015doubling,bss2021thinlaakso}) in which the role of the diamonds in the Laakso construction is played by what we call \emph{good} bubbles. A bubble of initial height $k$ typically has lifetime of order $\sqrt{k}$ (i.e., the bottom-most geodesics started from $(i,j)$ and $(i,j+k)$ will usually coalesce after moving about $\sqrt{k}$ steps to the right). Occasionally, however, we get an exceptionally long-lived bubble which behaves in many ways like the ``diamond'' that gets inserted in the middle of an edge when iterating the Laakso graph construction. Since the top geodesic of a bubble is itself a geodesic segment, and the behavior of the LHPT on top of it (and before and after it) is independent of what happens inside the bubble (this is a form of the spatial Markov property of the LHPT, which is made precise in \Cref{step:distrect_3} of the proof of \Cref{prop:distortion_rectangle_lhpt}), we can repeat the search for good bubbles at a smaller scale around and on top of a good bubble. We can then iterate this process, in the same way that Laakso graphs are built recursively. This procedure of searching for good bubbles is what we call the \emph{exploration process}, and much of this paper is devoted to showing that the structures which it produces resemble Laakso graphs closely enough to yield the desired lower bound on the $\ell^2$ distortion. See \Cref{fig:exploration process}.

The metric structure of the triangulation inside a bubble is mostly determined by the part of the skeleton that it contains, that is, by the subtrees of the skeleton whose roots index the slots along the vertical segment where the bubble starts.  Recall that these are Galton--Watson trees with offspring distribution $\theta$ given in~\eqref{eq:theta_distribution}. The distribution $\theta$ has several nice properties, the most important of which for us is that the iterates of its generating function can be computed explicitly. This allows us to find an exact, simple formula for the distribution of the extinction time of a Galton--Watson process with this offspring distribution, and this in turn will let us obtain some very precise information on the behavior of such processes.

The origin of the exponent $1/4$ in \Cref{thm:main} is somewhat involved. We will provide an intuitive explanation in \Cref{sec:distortion LHPT}, right after the proof of \Cref{prop:distortion_rectangle_lhpt}. Essentially, it is related to the fact that the boundary of the hull of radius $r$ in the UIPT typically has about $r^2$ vertices (the same is true in $\CT_n$ provided $r$ is small compared to $n^{1/4}$).

\subsubsection{Organization} The rest of the paper is organized as follows. In \Cref{sec:lhpt_skeleton}, we collect various estimates on Galton--Watson processes with offspring distribution $\theta$, which encode many of the metric properties of the LHPT.  In \Cref{sec:exploration_process}, we describe the exploration process in more detail, we give a precise definition of ``good'' bubbles, and we estimate the probability of a given bubble being good, as well as the probability of observing at least one good bubble.  We also prove the necessary estimates to be able to iterate the exploration around and below a good bubble.  In \Cref{sec:distortion LHPT}, we show how the exploration process can be used to derive an $\ell^2$ distortion lower bound for the LHPT. Finally, in \Cref{sec:LHPT to sphere}, we show how to transfer this distortion bound to a lower bound for finite triangulations of the sphere to complete the proof of \Cref{thm:main}.

\subsection*{Acknowledgements}

J.M.\ was supported by the ERC Consolidator Grant ARPF (Horizon Europe UKRI G120614).
J.R.\ was supported by an NSERC PDF scholarship.

\section{Estimates for the skeleton of the LHPT}
\label{sec:lhpt_skeleton}

As already discussed above, the metric properties inside a bubble are largely dictated by the behavior of the skeleton. This is particularly convenient, as the skeleton has a simple description in terms of Galton--Watson trees with a specific critical offspring distribution $\theta$, which we recall is given in~\eqref{eq:theta_distribution}. The distribution $\theta$ has many nice properties; for instance, it is in the domain of attraction of a $3/2$-stable law (recall the tail asymptotics recorded just after~\eqref{eq:theta_distribution}). For our purposes, however, the most important property will be that the iterates of its generating function can be computed explicitly.

\begin{lemma}[Formula for the generating function]
\label{lem:generating_function}
For $x \in [0,1)$, let $f(x)=\sum_{k=0}^{\infty} \theta(k)x^k$ be the generating function of $\theta$, let $f_1(x)=f(x)$, and for $n>1$, let $f_n(x)=f_{n-1}(f(x))$ be the $n$-th iterate of $f$ with itself. Then for all $n \ge 1$,
\begin{equation}\label{eq:nth_iterate}
f_n(x)=1-\left(n+\frac{1}{\sqrt{1-x}}\right)^{-2}.
\end{equation}
\end{lemma}

\begin{proof}
\noindent\steplabel{1}{step:genfn_1}\emph{Step 1. Expanding the right-hand side of~\eqref{eq:nth_iterate} for $n=1$.} Throughout the proof we take $x \in (0,1)$, so that we may divide by $x$; the case $x=0$ of~\eqref{eq:nth_iterate} then follows by continuity (or directly, since $f_n(0)=1-(n+1)^{-2}$ and $\theta(0)=3/4$). We first expand:
\begin{equation}\label{eq:expand_generating_function}
\begin{aligned}
1-\left(1+\frac{1}{\sqrt{1-x}}\right)^{-2}&=1-\frac{1-x}{(1+\sqrt{1-x})^2}=\frac{(1+\sqrt{1-x})^2-(1-x)}{(1+\sqrt{1-x})^2}\\
&=\frac{x^2-(1-x)(1-\sqrt{1-x})^2}{x^2}=\frac{3x-2+2\sqrt{1-x}-2x\sqrt{1-x}}{x^2}.
\end{aligned}
\end{equation}

\noindent\steplabel{2}{step:genfn_2}\emph{Step 2. Computing $f$.} On the other hand, it is easy to check that
\[\theta(k)=\frac{\binom{2k}{k}}{4^k (k+1)}-\frac{\binom{2(k+1)}{k+1}}{4^{k+1} (k+2)};\]
indeed, $\binom{2k+2}{k+1}=\frac{2(2k+1)}{k+1}\binom{2k}{k}$, so that the right-hand side equals
\[\frac{\binom{2k}{k}}{2 \cdot 4^k (k+1)(k+2)}\big(2(k+2)-(2k+1)\big)=\frac{3\binom{2k}{k}}{2 \cdot 4^k(k+1)(k+2)},\]
which is $\theta(k)$ as defined in~\eqref{eq:theta_distribution}.
Since
\[c(x)=\frac{1-\sqrt{1-4x}}{2x}\]
is the generating function of the Catalan numbers, we have
\begin{align*}
f(x)&=\sum_{k=0}^{\infty} \theta(k)x^k=c\left(\frac{x}{4}\right)-\frac{1}{x}\left(c\left(\frac{x}{4}\right)-1\right)=\frac{2-2\sqrt{1-x}}{x}-\frac{2-2\sqrt{1-x}}{x^2}+\frac{1}{x}\\
&=\frac{3x-2+2\sqrt{1-x}-2x\sqrt{1-x}}{x^2}.
\end{align*}
This is exactly~\eqref{eq:expand_generating_function}. This proves~\eqref{eq:nth_iterate} for $n=1$.

\noindent\steplabel{3}{step:genfn_3}\emph{Step 3. The induction.} The case $n>1$ follows by induction.  Indeed, the case $n=1$ can be rewritten as $(1-f(x))^{-1/2}=1+(1-x)^{-1/2}$, and hence, assuming~\eqref{eq:nth_iterate} for $n-1$, we get
\[f_n(x)=f_{n-1}(f(x))=1-\big(n-1+(1-f(x))^{-1/2}\big)^{-2}=1-\big(n+(1-x)^{-1/2}\big)^{-2},\]
as desired.  (The inductive hypothesis is applicable at the point $f(x)$ because $f(x) \in [\theta(0),1)=[3/4,1)$ for $x \in [0,1)$, $f$ being the generating function of a non-degenerate probability distribution on $\Z_{\ge0}$.)
\end{proof}

Let $X=(X_n)_{n \ge 0}$ be a Galton--Watson process with offspring distribution $\theta$ and initial population $X_0=1$. Then the generating function of $X_n$ is $f_n$. In particular,
\[\p_1(X_n>0)=1-f_n(0)=\frac{1}{(n+1)^2},\]
where the subscript $1$ in $\p_1$ is to emphasize that the initial population size is $1$.  More generally, for a positive integer $x$ we use $\p_x$ to denote the law of $X$ when the initial population size is $x$, and $\E_x$ for the corresponding expectation (as elsewhere in the paper, when $x$ is given by an expression which is not an integer, it is to be rounded to one). Then we have
\begin{equation}
\label{eq:bubble_lifetime}
\p_x(X_n>0)=1-(f_n(0))^x=1-\left(1-\frac{1}{(n+1)^2}\right)^x.
\end{equation}

We will frequently use the following inequalities:
\begin{equation}\label{eq:bernoulli}
1-my \le (1-y)^m \le 1-my+\frac{m^2 y^2}{2}.
\end{equation}
These bounds hold for any $0 \le y \le 1$ and any non-negative integer $m$. These bounds follow, for instance, from Taylor's theorem with remainder, or by induction on $m$.

We will also use the Chernoff--Hoeffding inequality:
\begin{equation}\label{eq:chernoff_hoeffding}
\p(Z \ge n p) \le e^{-\CD(p \| q)n}
\end{equation}
for $Z$ a $\Bin(n,q)$ random variable and $0<q<p<1$. Here $\CD$ denotes the Kullback--Leibler divergence which is given by
\begin{equation}\label{eq:kl_divergence}
\CD(p \| q)=p\log\left(\frac{p}{q}\right)+(1-p)\log\left(\frac{1-p}{1-q}\right).
\end{equation}
We shall repeatedly bound it from below using
\begin{equation}\label{eq:kl_lower_bound}
\CD(p \| q) \ge \frac{(p-q)^2}{2p} \qquad\text{for } p \ge q.
\end{equation}
(To see this, note that both sides vanish when $q=p$, and that for $0<q\le p$ we have
\[\partial_q \CD(p\|q)=\frac{q-p}{q(1-q)} \le \frac{q-p}{p}=\partial_q\frac{(p-q)^2}{2p},\]
since $q(1-q) \le q \le p$ and $q-p\le0$; integrating from $q$ to $p$ gives the claim.)

As a result of \eqref{eq:bubble_lifetime}, we can obtain estimates for the maximum and minimum of $X$.

\begin{lemma}[Bound for the maximum]\label{lem:tail_bound_maximum}
For all integers $x \ge y \ge 1$, we have
\begin{equation}\label{eq:upper_tail_bound_maximum}
\p_y\!\left(\sup_{j \ge 0} X_j \ge x\right) \le \frac{y}{x}.
\end{equation}
\end{lemma}

\begin{proof}
Since $\theta$ has mean 1, the process $X$ is a non-negative martingale. Thus by Doob's inequality,
\[\p_y\!\left(\max_{0 \le j \le n} X_j \ge x\right) \le \frac{\E_y(X_n)}{x}=\frac{y}{x}\]
for all $n\ge 0$. Applying the continuity of probability to the increasing sequence of events $\{\max_{0 \le j \le n} X_j \ge x\}$ as $n \to \infty$ gives~\eqref{eq:upper_tail_bound_maximum}.
\end{proof}

\begin{lemma}[Bound for the minimum]\label{lem:tail_bound_minimum}
Let $x$ and $n$ be positive integers, and let $a \in (0,x)$. Then
\[\p_x\!\left(\min_{0 \le j \le n}X_j \ge a \right) \ge \left(1-\frac{1}{(n+1)^2}\right)^a-\left(1-\frac{1}{(2n+1)^2}\right)^x.\]
\end{lemma}

\begin{proof}
\noindent\steplabel{1}{step:minimum_1}\emph{Step 1. A stopping time.} Let
\[\tau=\min \{k \ge 0: X_k<a\}.\]
Then $\tau$ is a stopping time, and $\{\tau>n\}=\{\min_{0 \le j \le n} X_j \ge a\}$, so that it suffices to bound $\p_x(\tau>n)$ from below.  Note also that $\tau \ge 1$, since $X_0=x>a$.

\noindent\steplabel{2}{step:minimum_2}\emph{Step 2. Decomposition via the strong Markov property.} Since $0$ is an absorbing state for the branching process $X$, on the event $\{\tau \le n\}$ we have that the event $\{X_{2n}>0\}$ is contained in $\{X_{\tau+n}>0\}$, because $\tau+n \le 2n$ there. By the strong Markov property at the bounded stopping time $\min(\tau,n)$, we have
\begin{align*}
\p_x(X_{2n}>0)&\le \p_x(X_{2n}>0, \tau \le n)+\p_x(\tau>n)\\
&\le \p_x(X_{\tau+n}>0, \tau \le n )+\p_x(\tau>n)\\
&= \E_x\!\left(\p_{X_{\tau}}(X_n>0)\one_{\{\tau \le n\}}\right)+\p_x(\tau>n).
\end{align*}
Since $X_{\tau} < a$ on the event $\{\tau \le n\}$, and $t \mapsto 1-(f_n(0))^{t}$ is increasing in $t$, we have $\p_{X_{\tau}}(X_n>0) \le \p_a(X_n>0) := 1-(f_n(0))^a$ almost surely on $\{\tau \le n\}$. Therefore,
\begin{align*}
\p_x(X_{2n}>0) &\le \p_a(X_n>0)\p_x(\tau \le n)+\p_x(\tau>n)\\
&\le \p_a(X_n>0)+\p_x(\tau>n).
\end{align*}

\noindent\steplabel{3}{step:minimum_3}\emph{Step 3. Conclusion.} Rearranging, we obtain
\begin{align*}
\p_x(\tau>n) &\ge \p_x(X_{2n}>0)-\p_a(X_n>0)\\
&= \left(1-\frac{1}{(n+1)^2}\right)^a- \left(1-\frac{1}{(2n+1)^2}\right)^x,
\end{align*}
where we used~\eqref{eq:bubble_lifetime} for both terms; this is the assertion of the lemma, by \Cref{step:minimum_1}.
\end{proof}

We can also get precise asymptotics for the expected extinction time.

\begin{lemma}[Expected extinction time]\label{lem:expected_length_of_bubble}
Let $T=\min \{k \ge 1: X_k=0\}$ be the extinction time of~$X$. Then
\begin{equation}\label{eq:expected_length_of_bubble}
\E_n(T)=\sqrt{\pi n}+O(1) \quad\text{as}\quad n \to \infty.
\end{equation}
Also
\[\frac{3 \sqrt{n}}{4} \le \E_n(T) \le 2\sqrt{n} \quad\text{for all}\quad n \ge 1.\]
\end{lemma}

\begin{proof}
\noindent\steplabel{1}{step:explen_1}\emph{Step 1. Reduction to a maximum of continuous random variables.} Under $\p_1$, the law of $T$ is given by
\[\p_1(T>k)=\p_1(X_{k}>0)=\frac{1}{(k+1)^2} \quad\text{for all integers } k \ge 0,\]
and so, since the $n$ initial individuals generate independent trees, under $\p_n$ the extinction time $T$ has the same law as $M_n=\max(T^{(1)},\dots, T^{(n)})$ where the $T^{(i)}$ are i.i.d.\ copies of $T$ under $\p_1$. Let $Y_1, Y_2, \dots$ be i.i.d.\ with
\[\p(Y_1>x)=\frac{1}{(x+1)^2} \quad\text{for}\quad  x \ge 0,\]
that is, they have density $2/(x+1)^3$ for $x \ge 0$. Then the $\lceil Y_i \rceil$ have the same distribution as the $T^{(i)}$. Then setting $T^{(i)}=\lceil Y_i \rceil$ and $M_n'=\max(Y_1, \dots, Y_n)$, we clearly have
\[Y_i \le T^{(i)} \le Y_i+1\]
and
\[M_n' \le M_n \le M_n'+1.\]
Thus it is enough to prove~\eqref{eq:expected_length_of_bubble} for $M_n'$ instead of $M_n$.

\noindent\steplabel{2}{step:explen_2}\emph{Step 2. A recursion for $\E(M_n')$.} Now,
\begin{align*}
\E(M_n')&=\int_0^{\infty} \p(M_n'>x)dx=\int_0^{\infty} 1-\left(1-\frac{1}{(x+1)^2}\right)^n dx\\
&=\int_0^{\infty} 1-\left(1-\frac{1}{(x+1)^2}\right)^{n-1} dx+\int_0^{\infty} \frac{1}{(x+1)^2} \left(1-\frac{1}{(x+1)^2}\right)^{n-1} dx\\
&=\E(M_{n-1}')+\int_0^{\infty} \frac{1}{(x+1)^2} \left(1-\frac{1}{(x+1)^2}\right)^{n-1} dx.
\end{align*}
For this last integral, we perform the substitution $u=1-1/(x+1)^2$. The resulting integrand is proportional to the density of the Beta$(n,1/2)$ distribution, and so we find
\[\int_0^{\infty} \frac{1}{(x+1)^2} \left(1-\frac{1}{(x+1)^2}\right)^{n-1} dx=\frac{1}{2}\int_0^1 \frac{u^{n-1}}{\sqrt{1-u}} du=\frac{\sqrt{\pi}}{2}\frac{\Gamma(n)}{\Gamma(n+1/2)}.\]

\noindent\steplabel{3}{step:explen_3}\emph{Step 3. Evaluation of the sum.} Therefore, by telescoping (with the convention $\E(M_0')=0$),
\[\E(M_n')=\sum_{k=1}^n \frac{\sqrt{\pi}}{2}\frac{\Gamma(k)}{\Gamma(k+1/2)}=\frac{\sqrt{\pi}\,\Gamma(n+1)}{\Gamma(n+1/2)}-1=\frac{4^n}{\binom{2n}{n}}-1.\]
The second equality can be checked easily by induction: it holds for $n=1$, both sides being equal to $1$, and if it holds for $n$ then, using $\Gamma(n+3/2)=(n+\frac12)\Gamma(n+\frac12)$,
\[\frac{\sqrt{\pi}\,\Gamma(n+1)}{\Gamma(n+1/2)}+\frac{\sqrt{\pi}}{2}\frac{\Gamma(n+1)}{\Gamma(n+3/2)}=\sqrt{\pi}\,\Gamma(n+1)\,\frac{(n+\frac12)+\frac12}{\Gamma(n+3/2)}=\frac{\sqrt{\pi}\,\Gamma(n+2)}{\Gamma(n+3/2)},\]
which is the corresponding expression for $n+1$.  The third equality follows from the identity $\Gamma(n+1/2)=\sqrt{\pi}\,(2n)!/(4^n n!)$.

\noindent\steplabel{4}{step:explen_4}\emph{Step 4. Asymptotics and explicit bounds.} Stirling's formula yields~\eqref{eq:expected_length_of_bubble} immediately, and for the bounds, we can use the well-known inequalities
\[\frac{1}{2\sqrt{n}} \le \frac{1}{4^n} \binom{2n}{n} \le \frac{1}{\sqrt{\pi n}}\]
to conclude that
\[\E_n(T) \ge \E(M_n') \ge \sqrt{\pi n}-1 \ge (\sqrt{\pi}-1)\sqrt{n} \ge \frac{3\sqrt{n}}{4}\]
and
\[\E_n(T) \le \E(M_n')+1 \le 2\sqrt{n}.\qedhere\]
\end{proof}

We next prove two lemmas which control the probability that $X$ drops too quickly from its initial height.

\begin{lemma}
\label{lem:dropping_below_initial}
There exists a constant $c_{\ref{lem:dropping_below_initial}} > 0$ so that for all $\zeta \in (0,c_{\ref{lem:dropping_below_initial}})$ we have that
\[ \p_{m^2}\!\left(\min_{0 \le j \le 3m} X_j \ge \zeta m^2 \right) \ge c_{\ref{lem:dropping_below_initial}} - \zeta \quad\text{for all}\quad m \in \N.\]
\end{lemma}
\begin{proof}
\noindent\steplabel{1}{step:dropinit_1}\emph{Step 1. Application of \Cref{lem:tail_bound_minimum}.} Fix $\zeta \in (0,c_{\ref{lem:dropping_below_initial}})$.  By \Cref{lem:tail_bound_minimum}, applied with $n=3m$ and with the real number $a = \zeta m^2$,
\[\p_{m^2}\!\left(\min_{0 \le j \le 3m} X_j \ge \zeta m^2 \right) \ge \left(1-\frac{1}{(3m+1)^2}\right)^{\zeta m^2}-\left(1-\frac{1}{(6m+1)^2}\right)^{m^2}.\]

\noindent\steplabel{2}{step:dropinit_2}\emph{Step 2. The first term.}  We use that $\log(1-y) \ge -y/(1-y)$ for $y \in (0,1)$, which gives
\[\left(1-\frac{1}{(3m+1)^2}\right)^{\zeta m^2} \ge \exp\left(-\frac{\zeta m^2}{(3m+1)^2-1}\right) = \exp\left(-\frac{\zeta m}{9m+6}\right) \ge 1-\frac{\zeta m}{9m+6} \ge 1-\zeta.\]

\noindent\steplabel{3}{step:dropinit_3}\emph{Step 3. The second term and conclusion.}  For the second term, using~\eqref{eq:bernoulli} we get
\[\left(1-\frac{1}{(6m+1)^2}\right)^{m^2} \le 1-\vartheta+\frac{\vartheta^2}{2}, \qquad \text{where}\quad \vartheta=\frac{m^2}{(6m+1)^2}.\]
For every $m \in \N$ we have $\vartheta \in [\frac{1}{49},\frac{1}{36}]$, and $t \mapsto t-\frac{t^2}{2}$ is increasing on $[0,1]$, so the right-hand side is at most $1-c$ for a universal constant $c \in (0,1)$.  Combining the two bounds shows that the probability in the statement is at least $c-\zeta$, which is the assertion of the lemma with $c_{\ref{lem:dropping_below_initial}}=c$.
\end{proof}

\begin{lemma}
\label{lem:dropping_below_strong}
There exists a constant $c_{\ref{lem:dropping_below_strong}} > 0$ so that for each $\gamma \in (0,1)$ we have that
\[\p_{m^2}\!\left(\min_{0 \le j \le \gamma m} X_j \ge c_{\ref{lem:dropping_below_strong}} m^2 \right) \ge 1-e^{-c_{\ref{lem:dropping_below_strong}}  \gamma^{-2}} \quad \text{for all}\quad m \in \N. \]
\end{lemma}
\begin{proof}
\noindent\steplabel{1}{step:dropstrong_1}\emph{Step 1. Splitting into $\gamma^{-2}$ independent groups.} Set $\zeta = c_{\ref{lem:dropping_below_initial}}/2$ and $p = c_{\ref{lem:dropping_below_initial}} - \zeta = c_{\ref{lem:dropping_below_initial}}/2$, so that all of the constants below are universal.  We can break a Galton--Watson process starting at height $m^2$ into a sum of $\gamma^{-2}$ i.i.d.\ Galton--Watson processes starting at height $m^2 \gamma^2$. \Cref{lem:dropping_below_initial} then implies that each of these has probability at least $p$ that its height never drops below $m^2 \gamma^2 \zeta$ before time $3 m \gamma$, and in particular before time $m \gamma$.  If a proportion of at least $p/2$ of these processes achieve this, then the total height will never have dropped below $m^2 p \zeta/2$ (there are then at least $\gamma^{-2}p/2$ of them, each contributing at least $m^2\gamma^2\zeta$).

\noindent\steplabel{2}{step:dropstrong_2}\emph{Step 2. A binomial lower bound.} Thus,
\[\p_{m^2}\!\left(\min_{0 \le j \le  \gamma m} X_j \ge m^2 p \zeta/2 \right) \ge \p\!\left(Z \ge \gamma^{-2} p/2 \right),\]
where $Z \sim \Bin(\gamma^{-2}, p)$ (the number of sub-processes which stay above $m^2\gamma^2\zeta$ stochastically dominates $Z$, since they are independent and each succeeds with probability at least $p$).

\noindent\steplabel{3}{step:dropstrong_3}\emph{Step 3. Conclusion.} By the Chernoff--Hoeffding inequality applied to the lower tail of $Z$, whose mean is $\gamma^{-2}p$,
\[\p\!\left(Z \ge \gamma^{-2} p/2 \right) \ge 1-\exp\left(-\CD\left( p/2 \| p \right) \gamma^{-2} \right)\ge 1-\exp(-c_{\ref{lem:dropping_below_strong}} \gamma^{-2}),\]
where we may take $c_{\ref{lem:dropping_below_strong}} = \min\!\big(\CD(p/2\,\|\,p),\, p\zeta/2\big)$, which is a universal constant.  Combining these estimates completes the proof of the lemma.
\end{proof}

We finish this section with two lemmas giving bounds on the distance between points on the same vertical line in the LHPT. Morally speaking, the graph distance between the points $(n, h)$ and $(n, h+r)$ should be about $\sqrt{r}$ with high probability. Here we use $\CL_{[a,b]}$ to denote the triangulation obtained by restricting the LHPT to the strip $[a,b] \times \Z$, i.e., the sub-triangulation formed by the vertices $(i,j)$ with $a \le i \le b$, the slots between two consecutive lines $\{i\}\times \Z$ and $\{i+1\} \times \Z$ with $a \le i <b$, and the Boltzmann triangulations filling them.  We allow $b=\infty$, in which case we write $\CL_{[a,\infty)}$.

\begin{lemma}[Upper bound for the distance between points on the same vertical line]\label{lem:upper_bound_distances}
Let $x \ge 1$ be an integer, and let $A_x$ be the event that there is a path of length less than $(2r+1)x$ joining $(n,0)$ to $(n, r^2)$ that is entirely contained in the bubble of initial height $r^2$ started from $(n,0)$ and, moreover, in the strip $\CL_{[n,n+r]}$. Then for any integers $r \ge 1$ and $n \ge 0$,
\[\p(A_x) \ge 1-\exp\left(-\frac{(x-1)^2}{2x}\right).\]
\end{lemma}
\begin{proof}
\noindent\steplabel{1}{step:updist_1}\emph{Step 1. Reduction to $n=0$ and to $x \le r^2$.} The event $A_x$ is measurable with respect to $\CL_{[n,\infty)}$, which has the same law as $\CL$ (with $(n,0)$ playing the role of the origin, since the bottom-most geodesic started from $(0,0)$ passes through $(n,0)$).  Hence $\p(A_x)$ does not depend on $n$, and so we can assume without loss of generality that $n=0$.

We may also assume that $x \le r^2$: otherwise at most $r^2<x$ trees can survive, so that the event considered in the next step has probability one.

\noindent\steplabel{2}{step:updist_2}\emph{Step 2. Few trees survive for $r$ generations.} There are $r^2$ trees in the skeleton between the points $(0,0)$ and $(0,r^2)$, and each one has probability $1/(r+1)^2 \le 1/r^2$ of surviving for at least $r$ generations, i.e., of having at least one individual at generation $r$.  Consequently, the number of trees which survive for at least $r$ generations is stochastically dominated by a $\Bin(r^2,1/r^2)$ random variable $Z$. Thus
\[\p(\text{fewer than $x$ trees survive at least $r$ generations}) \ge 1 - \p(Z \ge x).\]
By the Chernoff--Hoeffding inequality,
\[\p(Z \ge x) \le \exp\left(-\CD\left(\frac{x}{r^2} \left\| \frac{1}{r^2} \right.\right)r^2\right) \le \exp\left(-\frac{\left(\frac{x}{r^2}-\frac{1}{r^2}\right)^2}{\frac{2x}{r^2}} \cdot r^2\right)=\exp\left(-\frac{(x-1)^2}{2x}\right).\]
Here $\CD$ is as in~\eqref{eq:kl_divergence}, and for the second inequality we used~\eqref{eq:kl_lower_bound}.

\noindent\steplabel{3}{step:updist_3}\emph{Step 3. Construction of the path.} So now consider the event that fewer than $x$ trees survive, which implies the number of surviving trees $m$ satisfies $m \le x-1$, and hence $m+1 \le x$. Let $(0, i_1), \dots, (0, i_m)$ denote the points on the boundary between $(0,0)$ and $(0, r^2)$ such that the tree of the skeleton rooted at the slot $(0,i_k+\frac12)$ immediately above the point survives for at least $r$ generations, listed in increasing order; we set $i_0 = -1$ and $i_{m+1} = r^2$. We can then build a path of length at most $(2r+1)x - 1 < (2r+1)x$ between $(0,0)$ and $(0, r^2)$ as follows. For each of the $m+1$ intervals $0 \le j \le m$, the slots lying strictly between the points $(0,i_j+1)$ and $(0,i_{j+1})$ are precisely those at heights $i_j+\frac32,\dots,i_{j+1}-\frac12$, none of which is one of the surviving slots $(0,i_k+\frac12)$; hence the trees rooted at them are all extinct by generation $r$, and so the bottom-most geodesics starting from $(0, i_j+1)$ and $(0, i_{j+1})$ must coalesce before reaching the line $\{r\} \times \Z$, so the path that goes right along the bottom-most geodesic from $(0, i_j+1)$, then back left along the bottom-most geodesic from $(0, i_{j+1})$ has length at most $2r$. We do this $m+1$ times to get from $(0,0)$ to $(0, r^2)$, going right and returning left along all the bottom-most geodesics, and then to get from $(0, i_{j})$ to $(0, i_{j}+1)$ (of which there are $m$ transitions), we simply take a step up along the corresponding edge of the boundary of $\CL$. This produces a path of length at most 
\[2r(m+1) + m = (2r+1)(m+1) - 1 \le (2r+1)x - 1 < (2r+1)x\]
between $(0,0)$ and $(0, r^2)$, and it is entirely contained in the bubble between these points.  Moreover, by construction, it is clearly contained in $\CL_{[0,r]}$.
\end{proof}

The next lemma gives a lower bound, and is a generalization of \cite[Proposition 15]{cl2019fpp} to the case where the points are not required to lie on the boundary of $\CL$.

\begin{lemma}\label{lem:lower_bound_distances}
For every $\kappa>0$ there is an integer $K \ge 1$, depending only on $\kappa$, such that the following holds for all integers $r \ge 1$ and $t_0 \ge 3r$.  Let $h_0 \in \Z$, let $\Gamma$ be the bottom-most geodesic started from $(0,h_0)$, and let $h_{\Gamma}(t)$ denote the height of $\Gamma$ at time $t$.  Then, with conditional probability at least $1-\kappa$ given $\CL_{[0,t_0-3r]}$, every path of $\CL$ of length less than $r$ which starts at the vertex of $\Gamma$ at some time $i$ with $|i-t_0| \le 2r$ visits only vertices $(t,h)$ with
\[h<h_{\Gamma}(t)+Kr^2.\]
\end{lemma}

Reversing the paths, the lemma says equivalently that with the same conditional probability no vertex $(t,h)$ with $h \ge h_{\Gamma}(t)+Kr^2$ is within distance $r$ of the portion of $\Gamma$ lying within $2r$ of time $t_0$.  Taking $h_0=0$, so that $\Gamma$ is the line $\Z_{\ge 0}\times\{0\}$ and $h_{\Gamma}\equiv 0$, gives back a statement similar to \cite[Proposition 15]{cl2019fpp}, but with the far endpoint free to lie at any time rather than on a prescribed vertical line.

\begin{proof}
\noindent\steplabel{1}{step:lowdist_1}\emph{Step 1. Reduction to $h_0=0$ and $t_0=3r$.} Set $s=t_0-3r \geq 0$.  A path of length less than $r$ which starts at a vertex of $\Gamma$ at a time $i$ with $|i-t_0| \le 2r$ is contained in the strip $\CL_{[s,t_0+3r]}$.  The event in the statement is therefore measurable with respect to $\CL_{[s,\infty)}$.  Now $h_{\Gamma}(s)$ is measurable with respect to $\CL_{[0,s]}$, whereas $\CL_{[s,\infty)}$ is independent of $\CL_{[0,s]}$ and has the same law as $\CL$, by the Markov property of the LHPT (the trees of the skeleton at generation $s$ form again an i.i.d.\ family of Galton--Watson trees with offspring distribution $\theta$, and the slots are filled with independent Boltzmann triangulations).  Since the law of $\CL$ is moreover invariant under vertical translations, we may translate vertically so that $(s,h_{\Gamma}(s))$ becomes the origin; the restriction of $\Gamma$ to $[s,\infty)$ then becomes the bottom-most geodesic started from the origin, which is the line $\Z_{\ge 0}\times\{0\}$, and heights become heights above $\Gamma$.  We may therefore assume that $h_0=0$ and $t_0=3r$, so that the starting points are the vertices $(i,0)$ with $i \in [r,5r]$ and every path under consideration is contained in $\CL_{[0,6r]}$.

\noindent\steplabel{2}{step:lowdist_2}\emph{Step 2. Three events of high probability.} By \cite[Proposition 15]{cl2019fpp}, applied with $13r$ in place of $r$, there is an integer $K' \ge 1$ such that for every $r \ge 1$,
\[\p\!\left(\inf_{j \ge K'r^2}d_{\mathrm{gr}}^{\CL}((0,0),(0,j)) \ge 13r\right) \ge 1-\frac{\kappa}{3}.\]
Additionally,
\[\p_{K'r^2}(X_{6r}>0)=1-\left(1-\frac{1}{(6r+1)^2}\right)^{K'r^2}.\]
This expression converges uniformly over $r \in [1,\infty)$ to 1 as $K' \to \infty$ (indeed $K'r^2/(6r+1)^2 \ge K'/49$ for $r \ge 1$, so that the expression is at least $1-e^{-K'/49}$), and so by choosing $K'$ larger if necessary, we also have that for all $r \ge 1$,
\[\p_{K'r^2}(X_{6r}>0) \ge 1-\frac{\kappa}{3}.\]
Note that enlarging $K'$ preserves the previous estimate as well, since it only shrinks the set of $j$ over which the infimum is taken.  Finally, \Cref{lem:tail_bound_maximum}, applied with $y=K'r^2$ and $x=Kr^2$, gives
\[\p_{K'r^2}\!\left(\sup_{j \ge 0} X_j \ge Kr^2\right) \le \frac{K'}{K},\]
so that upon taking $K \ge 3K'/\kappa$ we have, for all $r \ge 1$,
\[\p_{K'r^2}\!\left(\sup_{j \ge 0} X_j<Kr^2\right) \ge 1-\frac{\kappa}{3}.\]
Now fix $r \ge 1$, and let $A$ be the event where the following three things happen:
\begin{enumerate}[(i)]
\item\label{it:sep_distance} We have $\inf_{j \ge K'r^2} d_{\mathrm{gr}}^{\CL}((0,0),(0,j)) \ge 13r$.
\item\label{it:sep_survive} The trees in the skeleton started between $(0,0)$ and $(0, K'r^2)$ do not all die before generation $6r$.
\item\label{it:sep_population} The total population size of these trees never reaches $Kr^2$.
\end{enumerate}
By the preceding remarks and a union bound, we have that $\p(A) \ge 1-\kappa$.  We will use repeatedly that, by the way the skeleton is embedded, the bottom-most geodesic $\gamma$ started from $(0,K'r^2)$ passes through the point $(m,X_m)$ at time $m$, where $X_m$ denotes the total population at generation $m$ of the $K'r^2$ trees rooted between $(0,0)$ and $(0,K'r^2)$.  In particular, \eqref{it:sep_survive} and~\eqref{it:sep_population} say that $\gamma$ stays strictly above the line $\Z_{\ge0} \times \{0\}$ up to time $6r$, and that it stays strictly below height $Kr^2$ at all times.

\noindent\steplabel{3}{step:lowdist_3}\emph{Step 3. The separation argument.} Suppose that $A$ occurs, and that there is a path of length less than $r$ joining some point $(i,0)$ with $i \in [r,5r]$ to some point $(t,h)$ with $h \ge Kr^2$.  By \Cref{step:lowdist_1} the path is contained in $\CL_{[0,6r]}$; in particular $0 \le t \le 6r$.  By~\eqref{it:sep_survive} the starting point $(i,0)$ lies strictly below $\gamma$, and by~\eqref{it:sep_population} the endpoint $(t,h)$ lies strictly above it.  Now $\gamma$ crosses the strip $\CL_{[0,6r]}$ from its left side to its right side and therefore separates it into two parts, and a path in a planar map cannot cross $\gamma$ without visiting one of its vertices; hence the path must cross $\gamma$. Let $(m, \ell)$ be the crossing point; note that $0 \le m \le 6r$, and that this point has distance exactly $m$ to $(0, K'r^2)$ since it is on its bottom-most geodesic. Then the path
\[(0,0) \to (i,0) \to (m, \ell) \to (0, K'r^2)\]
has length at most $12r$, since $i \le 5r$, the subpath from $(i,0)$ to $(m,\ell)$ has length less than $r$, and the subpath from $(m,\ell)$ to $(0,K'r^2)$ has length $m \le 6r$. Here the first and last arrows are the bottom-most geodesics started from $(0,0)$ and $(0, K'r^2)$ respectively, and the middle arrow is the path of length less than $r$ joining $(i, 0)$ to $(t,h)$ but stopped upon hitting $(m, \ell)$. This contradicts~\eqref{it:sep_distance} in the definition of $A$, and so no such path exists. This concludes the proof.
\end{proof}

\section{Good bubbles}
\label{sec:exploration_process}

As discussed in \Cref{subsec:outline}, the overall strategy for proving \Cref{thm:main} is to establish the existence of certain Laakso-like structures in the LHPT.  We will do so by considering ``bubbles'' whose bottom and top boundaries consist of geodesics in the LHPT.  These bubbles will play the role of the diamond structures which appear in the Laakso graph.  To define these bubbles, we (roughly speaking) fix a small scaling parameter $\eps>0$ and a large parameter $B > 0$.  We then explore the bottom-most geodesic started from height $B \eps^4 n^2$ and we let it run until it coalesces with the geodesic segment of length $n$ from which the exploration starts; we call the latter the \emph{starting segment}, and we always take it to be the segment from $(0,0)$ to $(n,0)$.  In the terminology of \Cref{def:bubble}, we are then looking at the bubble of initial height $B\eps^4n^2$ started at $(0,0)$, whose bottom geodesic contains the starting segment and whose top geodesic is the geodesic being explored. See the left side of \Cref{fig:exploration process}.  Eventually, we will take $n$ to depend on $\eps$ in such a way that $n(\eps) \to \infty$ faster than any polynomial in $\eps^{-1}$, so that the smaller $\eps$ is, the more room we have in the LHPT to discover these bubbles. The estimates of the present section, however, require only the polynomial lower bound $n \ge \eps^{-4}$; see \Cref{fn:parameters}.

\begin{figure}
\includegraphics[scale=0.8]{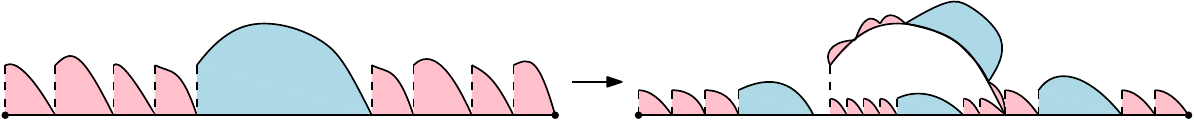}
\caption{The exploration process. In pink are ``bad'' bubbles, and in blue are ``good'' bubbles. When a good bubble is present, we can repeat the exploration at a smaller scale around the good bubble, as indicated in the right figure.}
\label{fig:exploration process}
\end{figure}

By \Cref{lem:expected_length_of_bubble}, a bubble will typically have lifetime approximately $\sqrt{B}\eps^2n$ (i.e., the square root of the initial height).  If we think of exploring the bubbles from left to right in the LHPT at equally spaced times $j\eps^2n$, then every so often, a bubble will survive for much longer than this. If a bubble does this and satisfies various other technical conditions that we list below, then we call it \emph{good}, and we can recursively search for bubbles along the parts of the starting segment before and after the bubble, below the bubble, and also on top of the bubble; see the right of \Cref{fig:exploration process}.  This scheme should be thought of as the stochastic analog of the way that the Laakso graph is built.  The length $n$ of the starting segment is now different, and so the initial height of these bubbles is also different, and the required lifetime needed for them to be good is different as well, and is given in terms of the new length. Note that we only consider good bubbles in the middle third of the segment; this is to ensure that the parts of the starting segment before and after the bubble are long enough so that we can perform the exploration on them and obtain a non-trivial structure.

Suppose that we have fixed a $\eps^{-2}/3 \le j \le 2\eps^{-2}/3$ and we are working on the event that a good bubble starts at time $j\eps^2n$ (we emphasize that we are not necessarily assuming that this is the first such good bubble).  Then the parts of the LHPT on top of, to the left of, and to the right of the good bubble are independent of the good bubble, but this is not the case for the part of the LHPT below the bubble.  For this reason, we divide our good bubbles into their ``bottom'' and ``top'' parts.  The bottom part of the bubble (discovered by the \emph{bottom exploration} described in \Cref{subsec:bottom_top_explorations}) will have maximal height much smaller than $\eps^2n^2$ and serves to decorrelate the behavior of the LHPT on the bottom of the bubble from the top of the bubble. In \Cref{subsec:bottom_exploration} and \Cref{subsec:prob_bubble_good}, we will describe and prove estimates for the top exploration, which discovers the top of the bubble, and use this to obtain bounds for the probability of a bubble being good.  The bottom and top explorations are designed to be independent of each other which will aid in the aforementioned decorrelation.  Finally, in \Cref{subsec:no_good_bubbles} we prove a tail bound for the probability that there are not any good bubbles along a geodesic of length $n$ in the LHPT.

\subsection{Definition of the bottom and top explorations}
\label{subsec:bottom_top_explorations}

Let us begin by giving the definition of the bottom exploration.  Suppose that we have fixed a small parameter $u>0$, then large parameters $R,B>0$, and then a small parameter $\eps>0$.\footnote{\label{fn:parameters}Five parameters $u, R, D, B, \eps$ (and then $n = n(\eps)$) appear in this section, and the constraints on them go in both directions, so we record here the order in which they are to be chosen. First we fix $u \in (0,1/10)$ small, and then $k_0 = k_0(u)$ large, as in \Cref{step:botexp_1} of the proof of \Cref{prop:bottom_exploration_good} and \Cref{step:topsurv_4} of the proof of \Cref{lem:top_exploration_survives_long}. The constants $c_{\ref{prop:no_good_bubbles}}$ and $c_{\ref{lem:conditional_prob_given_failure}}$ from \Cref{prop:no_good_bubbles} and \Cref{lem:conditional_prob_given_failure} only depend on $u$. Next we fix $R$ large; then $\kappa>0$ small, which determines $K=K(\kappa)$ from \Cref{lem:lower_bound_distances}; and then $D>0$ small, depending on $R$ and $K$.  These three are chosen in \Cref{step:goodbubble_5,step:goodbubble_9} of the proof of the lower bound of \Cref{prop:good_bubble_probability}.  We then fix $B$ large, subject to $B \ge C_0 R^6$ and $B \ge C_0 k_0^2\,2^{k_0(1+2u)}$ for a sufficiently large universal constant $C_0$: the first is used in the proof of \Cref{prop:bottom_exploration_good}, so that the threshold $R^3\eps^4n^2$ is a small multiple of the initial height $\sqrt{B}\eps^4n^2$ of a slice, and the second at the end of the proof of \Cref{lem:top_exploration_survives_long}.  Finally we take $\eps>0$ small, depending on all of the above, and $n \ge \eps^{-4}$.  It matters that this lower bound on $n$ be \emph{polynomial} in $\eps^{-1}$: in \Cref{sec:distortion LHPT} the results below are applied at every scale down to $\eps^{-4}$ at a single fixed value of $\eps$.  (There the largest scale is taken to be $n=e^{\eps^{-2}}$.  That is needed only in order to have many scales available, not for any estimate below.)  Every lower bound on $n$ used here is of the required form: each asks only that $n$ exceed a fixed power of $\eps^{-1}$ times a constant depending on $u$ and $B$.  All constants in this section (including those implicit in the $O(\cdot)$ notation) may depend on $u$, $k_0$, $R$, $D$ and $B$, but never on $\eps$ or $n$; likewise ``for all $\eps>0$ sufficiently small'' means: for all $\eps$ below a threshold depending on those five parameters.}  In words, the height of the bottom exploration evolves as a Galton--Watson process starting from $\sqrt{B}\eps^4 n^2$ with upward jumps of size $\sqrt{B}\eps^4 n^2$ at a progressively denser set of times. These times are dense enough that the height process has a positive chance of never going below $R^3 \eps^4 n^2$ before time $R \eps n$, and sparse enough that with positive probability it remains $O(\eps^2 n^2)$ up to that time.  The exploration is defined precisely by starting off with an initial height of $\sqrt{B}\eps^4 n^2$ and then discovering the trees with height between $0$ and $\sqrt{B}\eps^4 n^2$.  Suppose that $1 \le k \le k_{\max} := \lfloor \log_2 (\eps^{-2}/\sqrt[4]{B}) \rfloor$ and we have defined the exploration up to time $j \sqrt[4]{B} 2^{-u k} \eps^2 n$ where $j \in \N$ satisfies
\[\sqrt[4]{B}2^{k/2} \eps^2 n \le j \sqrt[4]{B}2^{-u k} \eps^2 n \le \sqrt[4]{B}2^{(k+1)/2} \eps^2 n,\]
and let $H_{j,k}$ be the height of the bottom exploration at this time.  Then we continue the bottom exploration by adding to it the trees with height between $H_{j,k}$ and $H_{j,k} + \sqrt{B}\eps^4 n^2$.  The collection of $\sqrt{B}\eps^4n^2$ trees added at such a step, as well as the initial collection, is a bubble in the sense of \Cref{def:bubble}: the former is the bubble with initial height $\sqrt{B}\eps^4n^2$ started at $(t,H_{j,k})$, where $t$ denotes the time of the step, and the latter is the bubble with initial height $\sqrt{B}\eps^4n^2$ started at $(0,0)$.  We call these particular bubbles the \emph{slices} of the bottom exploration.  In particular the height of a slice evolves as a Galton--Watson process started from $\sqrt{B}\eps^4n^2$, and by the branching property of the skeleton the slices are i.i.d.  The height of the bottom exploration at any given time is the sum of the heights at that time of the slices added so far.  Afterwards, for each $j \in \N$, we let $H_j$ be the height of the exploration at time $T_{\max}+j\,\sqrt[4]{B}\,2^{-uk_{\max}}\eps^{2}n$ (where $T_{\max}:=\sqrt[4]{B}2^{(k_{\max}+1)/2}\eps^2 n$ is the right endpoint of the last of the time intervals above, which is of order $\eps n$) and continue it by adding the trees with height between $H_j$ and $H_j + \sqrt{B}\eps^4 n^2$. In other words, after time $T_{\max}$ we stop increasing the density at which the upward jumps are added to the height process. The spacing $\sqrt[4]{B}2^{-uk_{\max}}\eps^2 n$ after $T_{\max}$ agrees with the spacing just before it, and at every stage of the exploration the spacing is at least $\sqrt[4]{B}\eps^{2+2u}n$; the latter is the only property of it used below.

\begin{definition}
\label{def:bottom_good}
For $u, R, B, \eps> 0$ fixed as above, we call the bottom exploration \emph{good} if it satisfies the following conditions.
\begin{enumerate}[(i)]
\item\label{it:bottom_height_bound} Its height never drops below $R^3\eps^4 n^2$ and never goes above $\eps^2 n^2$, up to time $R \eps n$.
\item\label{it:bottom_distance_initial} The distance between $(0,0)$ and $(0, \sqrt{B}\eps^4 n^2)$ inside the bottom exploration is at most $\eps^{2-u} n$.
\end{enumerate}
\end{definition}

We will prove the following two statements in \Cref{subsec:bottom_exploration}.

\begin{proposition}
\label{prop:bottom_exploration_good}
Suppose that $u, R, B, \eps > 0$ are fixed as above, with $B \ge C_0R^6$ as in \Cref{fn:parameters}.  There exists $p_{\ref{prop:bottom_exploration_good}} \in (0,1)$ depending only on $u$ (but not on $R$, $B$, $\eps$ or $n$) so that, for all $\eps>0$ sufficiently small (depending on $u$, $R$ and $B$), the probability that the bottom exploration is good in the sense of \Cref{def:bottom_good} is at least~$p_{\ref{prop:bottom_exploration_good}}$.
\end{proposition}

We also need the following lemma which shows that at any \emph{fixed} time $\ell \ge \eps n$, the bottom exploration is unlikely to be much higher than $\eps^4n^2$, and its top and bottom at time $\ell$ are unlikely to be far apart.

\begin{lemma}
\label{lem:bottom_exploration_typical_height}
Suppose that $u, R, B, \eps > 0$ are fixed as above.  Fix $\ell \ge \eps n$ and let $h$ be the height of the bottom exploration at time $\ell$.  Then the following two conditions hold simultaneously with probability at least $1-O(\eps^u)$, where the implicit constant depends only on $u$ and $B$.
\begin{enumerate}[(i)]
\item\label{it:typical_height} We have $h \le \eps^{4-6u} n^2$.
\item\label{it:typical_path} For any height $0 \le h' \le h$ which is measurable with respect to $\CL_{[0,\ell]}$, there is a path of length at most $\eps^{2-4u} n$ connecting $(\ell,0)$ to $(\ell,h')$ which is entirely contained in the bottom exploration to the right of time $\ell$.
\end{enumerate}
\end{lemma}

Next we define the top exploration.  We start exploring the bottom-most geodesic in the LHPT to the right from the point $(0, B \eps^4 n^2)$ and then stop once this geodesic hits the bottom exploration.  The \emph{top exploration} is the part of the LHPT between the bottom exploration and this geodesic.  As in the case of the bottom exploration, the height process for the top exploration can be described explicitly.  For $t \ge 0$, let $Y_t$ be the height of this geodesic at time $t$ minus the height of the bottom exploration at time $t$; we call $(Y_t)_{t \ge 0}$ the \emph{height process of the top exploration}.  Then $Y_0 = (B-\sqrt{B})\eps^4 n^2$, and $(Y_t)$ evolves as a Galton--Watson process with offspring distribution $\theta$ and with a downward jump of size $\sqrt{B}\eps^4 n^2$ at each of the times at which a slice is added to the bottom exploration; that is, for each $1 \le k \le k_{\max}$ and each time of the form
\[\sqrt[4]{B}2^{k/2} \eps^2 n \le j \sqrt[4]{B}2^{-u k} \eps^2 n \le \sqrt[4]{B}2^{(k+1)/2} \eps^2 n,\]
as well as at each of the times $T_{\max}+j\,\sqrt[4]{B}\,2^{-uk_{\max}}\eps^{2}n$ for $j \in \N$. The process is stopped upon hitting a value $\le 0$; this happens precisely when the geodesic reaches a height below that of the bottom exploration for the first time. We call this time $\ell$ the \emph{merging time} of the two explorations, and the point $(\ell,h)$ on the top geodesic at this time is the \emph{merging vertex}.

\begin{definition}
\label{def:top_good}
For each $u, R, B, \eps>0$, we say that the top exploration is \emph{good} if it satisfies the following conditions.
\begin{enumerate}[(i)]
\item\label{it:initial_distance} The distance between $(0,B \eps^4 n^2)$ and $(0,\sqrt{B}\eps^4 n^2)$ in the top exploration is at most $\eps^{2-u} n$.
\item \label{it:top_survival} The top exploration merges with the bottom exploration between time $\eps n$ and $R \eps n$.
\item\label{it:max_height_top} The height process of the top exploration never goes above $R\eps^2 n^2/4$.
\item\label{it:middle_height} The height process of the top exploration at time $\eps n/2$ is at least $\eps^2 n^2/R$.
\end{enumerate}
\end{definition}

For the next definition, we additionally fix another small parameter $D>0$.

\begin{definition}
\label{def:bubble_good}
Suppose that we have parameters $u, R, D, B, \eps>0$ fixed as above, let $\ell$ be the merging time of the two explorations, and let $(\ell,h)$ be the merging vertex.  We say that the bubble is \emph{good} if the following conditions hold.
\begin{enumerate}[(i)]
\item\label{it:bottom_good} The bottom exploration is good in the sense of \Cref{def:bottom_good}.
\item\label{it:top_good} The top exploration is good in the sense of \Cref{def:top_good}.
\item\label{it:merge_distance} There is a path from $(\ell,h)$ to $(\ell,0)$ of length at most $\eps^{2-4 u} n$ which is entirely contained in the bottom exploration and to the right of time $\ell$.
\item \label{it:middle_graph_distance} Let $\Gamma$ be the top geodesic of the bubble and let $c_1$ be the vertex on $\Gamma$ at time $\epsilon n/2$ (see \Cref{fig:high distance point}).  The distance between $c_1$ and the bottom geodesic of the bubble is at least $\sqrt{D}\eps n$.
\end{enumerate}
\end{definition}

See \Cref{fig:good bubble} for an illustration of the anatomy of a good bubble.

\begin{figure}
\includegraphics[scale=0.8]{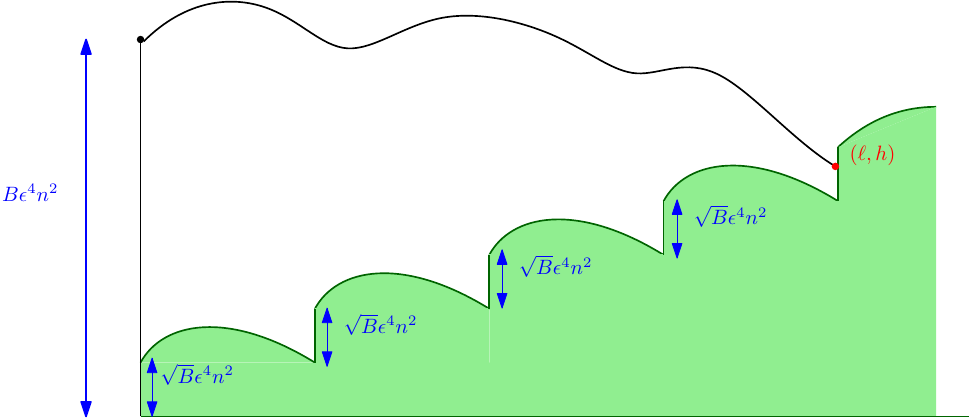}
\caption{The anatomy of a good bubble. The green region is the bottom exploration. The white region (i.e., the space between the top geodesic and the green region) is the top exploration. The merging time is $\ell$ and the merging vertex (in red) is $(\ell,h)$. Note that by definition, the merging vertex is always on the top geodesic of the bubble.}
\label{fig:good bubble}
\end{figure}

We note that condition~\eqref{it:bottom_good} is measurable with respect to the bottom exploration and that~\eqref{it:top_good} is measurable with respect to the top exploration.  Conditions~\eqref{it:merge_distance} and~\eqref{it:middle_graph_distance} depend on both the bottom and top explorations. Note that~\eqref{it:max_height_top} of \Cref{def:top_good} bounds the height process of the top exploration rather than the height of the bubble; on~\eqref{it:bottom_good} the bottom exploration has height at most $\eps^2n^2 \le R\eps^2n^2/4$, so the two together bound the height of the bubble by $R\eps^2n^2/2$, which is the form used in \Cref{sec:distortion LHPT}.

\begin{figure}
\includegraphics[scale=0.8]{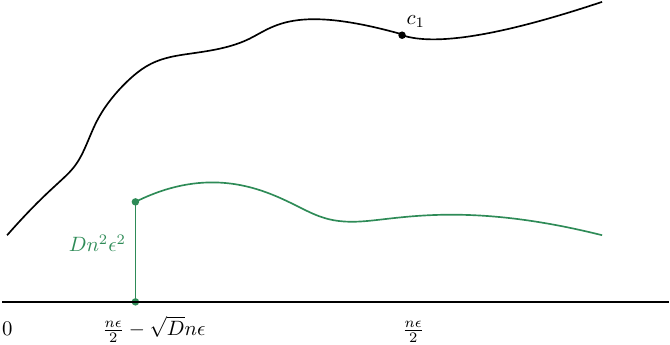}
\caption{Illustration of~\eqref{it:middle_graph_distance} from \Cref{def:bubble_good}.  The black curve is the top geodesic of the bubble and $c_1$ is its vertex at time $\eps n/2$; the condition requires that no path of length less than $\sqrt{D}\eps n$ started at $c_1$ hits the bottom geodesic of the bubble, and in particular that $d_{\mathrm{gr}}^{\CL}(c_1,(\eps n/2,0)) \ge \sqrt{D}\eps n$.  The green curve is the bottom-most geodesic started from $(\eps n/2-\sqrt{D}\eps n,\,D\eps^2n^2)$, illustrating the argument used to prove \Cref{lem:lower_bound_distances}: on~\eqref{it:middle_height} of \Cref{def:top_good}, any path of length less than $\sqrt{D}\eps n$ from $c_1$ to $(\eps n/2,0)$ would have to cross such a geodesic.}
\label{fig:high distance point}
\end{figure}

\begin{proposition}
\label{prop:good_bubble_probability}
Let $p$ be the probability that the bubble is good in the sense of \Cref{def:bubble_good}.  Then the parameters $u$, $R$, $D$ and $B$ appearing in \Cref{def:bottom_good,def:top_good,def:bubble_good} can be chosen (in the order given in \Cref{fn:parameters}) so that there is a constant $c_{\ref{prop:good_bubble_probability}} > 0$ such that for all $\eps > 0$ sufficiently small,
\[c_{\ref{prop:good_bubble_probability}} B\eps^2 \le p \le B\eps^2.\]
\end{proposition}

We prove \Cref{prop:good_bubble_probability} in \Cref{subsec:prob_bubble_good}.

The reason for setting things up in this manner is the following lemma.

\begin{lemma}[Conditional law of the exploration at the bottom of a good bubble]
\label{lem:conditional_prob_given_good}
Let $G$ be the event that the bubble starting at time $0$ is good in the sense of \Cref{def:bubble_good}.  For $\eps n \le \ell \le R\eps n$, let $G_\ell \subseteq G$ be the event that, in addition, the merging time of the two explorations equals $\ell$ and let $G_\ell^1$ be the event that the top exploration is good in the sense of \Cref{def:top_good} and that the merging time equals $\ell$.  Then for any positive probability event $A$ for the part of the LHPT between times $0$ and $R\eps n$ and with height between $0$ and $R^3\eps^4 n^2$ (i.e., any positive probability event which is measurable with respect to the subtriangulation of the LHPT formed by the vertices $(i,j)$ of $[0,R\eps n] \times [0,R^3\eps^4n^2]$, the slots they index and the Boltzmann triangulations filling them), we have
\begin{equation}
\label{eq:good_ell_resolved}
\p(A \cap G_{\ell}) \le \p(A)\,\p(G_{\ell}^1) \quad\text{for each }\ell
\end{equation}
and
\begin{equation}
\label{eq:good_top_decomposed}
\sum_{\eps n \le \ell \le R\eps n} \p(G_{\ell}^1) \le B\eps^2.
\end{equation}
\end{lemma}

Informally, the first bound says that an event $A$ concerning the LHPT near the bottom of the bubble is independent of the bubble being good, up to replacing $G_{\ell}$ by the slightly larger event $G_{\ell}^1$, which involves the top exploration alone; the second bounds the total mass of the enlarged events by $B\eps^2$, which by \Cref{prop:good_bubble_probability} is the order of $\p(G)$ itself, so that the enlargement is harmless.  We resolve the merging time because the events to which the lemma is applied in \Cref{sec:distortion LHPT} themselves depend on $\ell$, so that the conditioning has to be performed at a fixed value of~$\ell$.
\begin{proof}
\noindent\steplabel{1}{step:condgood_1}\emph{Step 1. Two events and their independence.} Let $G^2$ be the event that the bottom exploration is good in the sense of \Cref{def:bottom_good}.  Then, for each $\ell$,
\[G_{\ell} \subseteq G_{\ell}^1 \cap G^2,\]
since $G_\ell$ additionally requires conditions~\eqref{it:merge_distance} and~\eqref{it:middle_graph_distance}.  Moreover $G_{\ell}^1$ is measurable with respect to the top exploration (its conditions, together with the merging time, are functions of the height process $(Y_t)$ alone, except for~\eqref{it:initial_distance}, which involves the triangulation filling the top exploration), whereas $G^2$ is measurable with respect to the bottom exploration; hence $G_{\ell}^1$ and $G^2$ are independent.

\noindent\steplabel{2}{step:condgood_2}\emph{Step 2. Proof of~\eqref{eq:good_ell_resolved}.}  We note that the event $A \cap G^2$ is measurable with respect to the bottom exploration (by~\eqref{it:bottom_height_bound} the height of the bottom exploration never drops below $R^3\eps^4n^2$ up to time $R\eps n$ on $G^2$, so that the restriction of the LHPT to $[0,R\eps n] \times [0,R^3\eps^4n^2]$ lies inside the bottom exploration).  Since $G_{\ell}^1$ is independent of the bottom exploration, we thus have that
\[ \p(G_{\ell}\cap A) \le \p(G_{\ell}^1 \cap G^2 \cap A) = \p(G_{\ell}^1) \p(G^2 \cap A) \le \p(G_{\ell}^1) \p(A),\]
which is~\eqref{eq:good_ell_resolved}.

\noindent\steplabel{3}{step:condgood_3}\emph{Step 3. Proof of~\eqref{eq:good_top_decomposed}.}  The events $G^1_\ell$ are pairwise disjoint and their union over $\eps n \le \ell \le R\eps n$ is the event $G^1$ that the top exploration is good (recall from~\eqref{it:top_survival} that a good top exploration merges at a time in $[\eps n, R\eps n]$).  On $G^1$ we have $Y_t>0$ for every $t<\eps n$; since $(Y_t)$ is stochastically dominated by a Galton--Watson process started from $B\eps^4n^2$ (its initial height is $(B-\sqrt{B})\eps^4n^2 \le B\eps^4n^2$, and its downward jumps only decrease it),
\[\sum_{\ell} \p(G^1_\ell)=\p(G^1) \le \p_{B\eps^4n^2}(X_{\eps n-1}>0)=1-\left(1-\frac{1}{(\eps n)^2}\right)^{B\eps^4n^2} \le B\eps^2.\]
\end{proof}

\subsection{Bounds for the bottom exploration}
\label{subsec:bottom_exploration}

We now give the proof of \Cref{prop:bottom_exploration_good}, using \Cref{lem:dropping_below_initial,lem:dropping_below_strong} to control the probability that the bottom exploration drops in height too quickly.
\begin{proof}[Proof of \Cref{prop:bottom_exploration_good}]
Suppose that we have fixed parameters $u, R, B, \eps > 0$ as in the statement of the proposition.

\noindent\steplabel{1}{step:botexp_1}\emph{Step 1. The bottom exploration does not go too low before time $T_{\max}$.} Let us first show that the bottom exploration has a positive chance of never going too low.  For $1 \le k \le k_{\max}$, consider the time interval between $\sqrt[4]{B}2^{k/2} \eps^2 n$ and $\sqrt[4]{B}2^{(k+1)/2} \eps^2 n$.  Then by the definition of the bottom exploration, we are adding in upward jumps of size $\sqrt{B}\eps^4 n^2$ every $\sqrt[4]{B}2^{-u k} \eps^2 n$ units of time during this interval.  In particular the time between consecutive slices is at most $(\sqrt{2}+1)\sqrt[4]{B}\eps^2 n \le 3\sqrt[4]{B}\eps^2 n$ for the initial slice (added at time $0$, the next one at the first multiple of $\sqrt[4]{B}2^{-u}\eps^2 n$ which is at least $\sqrt[4]{B}2^{1/2}\eps^2 n$), and at most twice the spacing $\sqrt[4]{B}2^{-uk}\eps^2 n$ for a slice added in the $k$-th interval, the factor of two accounting for the change of spacing at the endpoints. Let $k_0$ be a large positive integer. We break this into two cases depending on whether $k$ is larger or smaller than $k_0$.

\noindent\emph{Case 1. The early stages $k<k_0$.} Since the height of the bottom exploration at any time is at least the current height of the most recently added slice, it suffices to bound from below the probability that every slice individually stays above $R^3\eps^4n^2$ until the next slice is added.  By \Cref{lem:dropping_below_initial}, applied with $m=\sqrt[4]{B}\eps^2n$ (so that $m^2=\sqrt{B}\eps^4n^2$ is the initial height of a slice, and $3m$ bounds the time until the next slice is added, by the gap bounds noted above) and with $\zeta = R^3/\sqrt{B}$, each slice has probability at least $p_0$ of not dropping below $R^3\eps^4 n^2$ before the next slice is added, where $p_0 = c_{\ref{lem:dropping_below_initial}}-R^3/\sqrt{B}$ is some positive constant (provided that we assume $B$ is much larger than $R$; precisely, $B \ge C_0R^6$ gives $R^3/\sqrt{B} \le C_0^{-1/2} \le c_{\ref{lem:dropping_below_initial}}/2$ for $C_0$ large, whence the universal lower bound $p_0 \ge c_{\ref{lem:dropping_below_initial}}/2$). For each $k$, there are
\[2^{k(u+1/2)+1/2}-2^{k(u+1/2)}+1 \le 2^{k(u+1/2)+1}\]
slices added in the time interval between $\sqrt[4]{B}2^{k/2} \eps^2 n$ and $\sqrt[4]{B}2^{(k+1)/2} \eps^2 n$. So in total, up to time $\sqrt[4]{B}2^{k_0/2} \eps^2 n$, there are at most $k_0 2^{k_0(u+1/2)+1}$ slices added, including the initial slice. The slices are all independent, and so with probability at least
\[p_0^{k_0 2^{k_0(u+1/2)+1}}>0,\]
the bottom exploration never drops below $R^3\eps^4 n^2$ before time $\sqrt[4]{B}2^{k_0/2} \eps^2 n$.

\noindent\emph{Case 2. The later stages $k \ge k_0$.} \Cref{lem:dropping_below_strong} with $m = \sqrt[4]{B}\eps^2 n$ and $\gamma = 2^{1-u k}$ implies that the probability that a slice in the LHPT with initial height $m^2 = \sqrt{B}\eps^4 n^2$ drops below height $R^3\eps^4 n^2$ before the next slice is added (which by the gap bounds above happens within $2 \cdot \sqrt[4]{B}2^{-u k} \eps^2 n = \gamma m$ units of time) is at most $\exp(-c_{\ref{lem:dropping_below_strong}} 2^{2u k-2})$, the requirement $B \ge C_0R^6$ guaranteeing that the floor $c_{\ref{lem:dropping_below_strong}}\sqrt{B}\eps^4n^2$ provided by \Cref{lem:dropping_below_strong} is at least $R^3\eps^4n^2$. Taking a union bound over $k_0 \le k \le k_{\max}$ and $j \in \N$ such that
\[\sqrt[4]{B}2^{k/2} \eps^2 n \le \sqrt[4]{B}2^{-uk} j \eps^2 n \le \sqrt[4]{B}2^{(k+1)/2} \eps^2 n,\]
 we see that the probability that the bottom exploration does not drop below $R^3\eps^4 n^2$ at any time between $\sqrt[4]{B}2^{k_0/2}\eps^2 n$ and $T_{\max}$ is at least
\[1-\sum_{k=k_0}^{k_{\max}} \exp(-c_{\ref{lem:dropping_below_strong}} 2^{2u k-2})\cdot 2^{k(u+1/2)+1} \ge 1-\sum_{k=k_0}^{\infty} \exp(-c_{\ref{lem:dropping_below_strong}} 2^{2u k-2})\cdot 2^{k(u+1/2)+1}.\]
For a given $u>0$, we can pick $k_0$ sufficiently large (in particular $k_0 \ge 2/u$, so that $\gamma = 2^{1-uk}<1$ throughout Case~2) so that this last expression is positive.

Finally, we combine these two cases using the independence of slices. The probability that the bottom exploration never drops below $R^3\eps^4 n^2$ before time $T_{\max}$ is at least
\[p_0^{k_0 2^{k_0(u+1/2)+1}} \left(1-\sum_{k=k_0}^{\infty} \exp(-c_{\ref{lem:dropping_below_strong}} 2^{2u k-2})\cdot 2^{k(u+1/2)+1}\right),\]
which as already discussed is positive if $k_0$ (as a function of $u$) is large enough. Since $p_0 \ge c_{\ref{lem:dropping_below_initial}}/2$ and $c_{\ref{lem:dropping_below_strong}}$ are universal constants, this bound depends on $u$ (through $k_0$) but not on $R$, $B$, $\eps$, or $n$.

\noindent\steplabel{2}{step:botexp_2}\emph{Step 2. The bottom exploration does not go too low between times $T_{\max}$ and $R\eps n$.} It remains to control the exploration on $[T_{\max}, R\eps n]$, where the slices are added at the constant spacing $\sqrt[4]{B}2^{-uk_{\max}}\eps^2 n$.  Case~2 applies verbatim with $\gamma = 2^{1-uk_{\max}} \le 4B^{u/4}\eps^{2u}$: each of the $O(R\eps^{-1-2u})$ slices on this interval, and the last one added before $T_{\max}$, stays above $c_{\ref{lem:dropping_below_strong}}\sqrt{B}\eps^4n^2 \ge R^3\eps^4n^2$ until the next is added, except with probability at most $\exp(-c\,\eps^{-4u})$ for a constant $c>0$ depending on $u$ and $B$.  A union bound therefore bounds the probability of dropping below $R^3\eps^4n^2$ somewhere in $[T_{\max}, R\eps n]$ by $O(R\eps^{-1-2u}\exp(-c\eps^{-4u})) \to 0$ as $\eps \to 0$.  Combining with the two cases above, the probability $q$ that the height never drops below $R^3\eps^4n^2$ up to time $R\eps n$ is bounded below by a positive constant depending only on $u$, for all $\eps$ small.  (That $k_0$ depends on $u$ alone is consistent with $B \ge C_0k_0^2 2^{k_0(1+2u)}$, since $B$ is chosen after $k_0$; see \Cref{fn:parameters}.)

\noindent\steplabel{3}{step:botexp_3}\emph{Step 3. The bottom exploration does not go too high.} Let us now show that the bottom exploration is unlikely to get too high.  First of all, the number of slices that we add to the bottom exploration over its course (that is, up to time $R\eps n$) is at most a constant (depending on $R$ and $B$) times $\sum_{k=1}^{k_{\max}} 2^{u k + k/2}$, which is of order $\eps^{-1- 2u}$: the sum counts the slices added before $T_{\max}$, and those added between $T_{\max}$ and $R\eps n$ number at most $(R\eps n)/(\sqrt[4]{B}\eps^{2+2u}n) = O(R\eps^{-1-2u})$, of the same order.   By \Cref{lem:tail_bound_maximum}, applied with $y = \sqrt{B}\eps^4n^2$ and $x = \eps^{3-3u}n^2$ (its hypothesis $x \ge y$ amounts to $\eps^{-1-3u} \ge \sqrt{B}$, valid for $\eps$ small), each slice has probability at most $\sqrt{B}\eps^{1+3u}$ of reaching a maximal height of at least $\eps^{3-3u}n^2$. So by a union bound, the probability that some slice hits $\eps^{3-3u}n^2$ is $O(\eps^u)$. 

Also, \Cref{lem:tail_bound_maximum} (this time with $x=\eps^{3+3u}n^2$, for which $x \ge y$ amounts to $\eps^{-1+3u} \ge \sqrt{B}$, again valid for $\eps$ small since $u<1/10$) implies that with probability at most $\sqrt{B}\eps^{1-3u}$, a slice has a maximal height of at least $\eps^{3+3u}n^2$.  Since the slices which make up the bottom exploration are i.i.d., the number of slices which have maximal height at least $\eps^{3+3u}n^2$ is stochastically dominated by $Z$, where $Z \sim \Bin(C\eps^{-1-2u},\sqrt{B}\eps^{1-3u})$ ($C$ is some positive constant). By Markov's inequality, it follows that the probability that there are more than $\eps^{-6u}$ slices that have a maximal height of at least $\eps^{3+3u}n^2$ is $O(\eps^u)$.

To summarize,
\begin{itemize}
\item There are $O(\eps^{-1-2u})$ slices added in total.
\item With probability $1-O(\eps^u)$, none of the slices has a maximal height greater than $\eps^{3-3u}n^2$.
\item With probability $1-O(\eps^u)$, there are fewer than $\eps^{-6u}$ slices that have a maximal height greater than $\eps^{3+3u}n^2$.
\end{itemize}

Thus with probability $1-O(\eps^u)$, the total maximal height of all these slices combined (i.e., if their maxima were all attained at the same time) is at most
\[C\eps^{-1-2u} \cdot \eps^{3+3u} n^2+\eps^{-6u} \cdot \eps^{3-3u} n^2=C\eps^{2+u}n^2+\eps^{3-9u}n^2,\]
which is of order $\eps^{2+u} n^2$ if $u>0$ is small enough (namely $u<1/10$, so that $3-9u \ge 2+u$).  Since the height of the bottom exploration is at all times at most the sum of the maximal heights of its slices, it follows that with probability $1-O(\eps^u)$ the height never exceeds $\eps^2 n^2$ up to time $R\eps n$, for all $\eps$ small.

\noindent\steplabel{4}{step:botexp_4}\emph{Step 4. The distance condition~\eqref{it:bottom_distance_initial}.} We verify condition~\eqref{it:bottom_distance_initial} of \Cref{def:bottom_good}, namely that $(0,0)$ and $(0,\sqrt{B}\eps^4n^2)$ are at distance at most $\eps^{2-u}n$ inside the bottom exploration.  Apply \Cref{lem:upper_bound_distances} with $r = \sqrt[4]{B}\eps^2 n$, so that $r^2$ is the height difference between these two points, and with $x = \lceil \eps^{-u}/(4\sqrt[4]{B})\rceil$: with probability at least $1-\exp(-c_B\eps^{-u})$, where $c_B>0$ depends on $B$, this produces a path between them of length less than $(2r+1)x \le \eps^{2-u}n$ (for $\eps$ small), contained in the bubble between the bottom-most geodesics started from them.

\noindent\steplabel{5}{step:botexp_5}\emph{Step 5. Conclusion.} Combining the four preceding steps, the probability that the bottom exploration is good is at least
\[q - O(\eps^u) - \exp(-c_B\eps^{-u}) \ge \frac{q}{2} =: p_{\ref{prop:bottom_exploration_good}}\]
for all $\eps$ sufficiently small, where $q$ is the probability from \Cref{step:botexp_1,step:botexp_2}, bounded below by a positive constant depending only on $u$.  This proves the proposition.
\end{proof}

We also now give the proof of \Cref{lem:bottom_exploration_typical_height}.

\begin{proof}[Proof of \Cref{lem:bottom_exploration_typical_height}]
Fix a value of $\ell$ as above.

\noindent\steplabel{1}{step:typheight_1}\emph{Step 1. The slices added long before time $\ell$ have died out.} For each integer $j \ge 0$, consider the time interval from $\ell-\sqrt[4]{B}2^{(j+1)/2}\eps^2 n$ to $\ell-\sqrt[4]{B}2^{j/2}\eps^2 n$. Since the highest density of slices being added to the bottom exploration is one every $\sqrt[4]{B}\eps^{2+2u}n$ units of time, this interval contains at most $2^{j/2} \eps^{-2u}$ slices for $\eps$ small. For each of these slices, the probability that it survives until time $\ell$ is at most
\[1-\left(1-\frac{1}{(\sqrt[4]{B}2^{j/2}\eps^2n+1)^{2}}\right)^{\sqrt{B}\eps^4 n^2}\le \frac{1}{2^j}.\]
So by a union bound, the probability that any survive in this interval is at most $2^{-j/2} \eps^{-2u}$. Then taking a union bound over $j \ge 2\log_2 \eps^{-3u}$, the probability that any slice between time 0 and $\ell-\sqrt[4]{B}\eps^{2-3u}n$ survives until time $\ell$ is at most
\[\sum_{j=2\log_2 \eps^{-3u}}^{\infty} \frac{\eps^{-2u}}{2^{j/2}}=O(\eps^u).\]

\noindent\steplabel{2}{step:typheight_2}\emph{Step 2. The slices added shortly before time $\ell$ are not too high.} Between times $\ell-\sqrt[4]{B}\eps^{2-3u}n$ and $\ell$, there are at most $\eps^{-5u}+1$ slices added to the bottom exploration. Since the height processes of the slices are martingales, their expected height at any time is the same as their initial height, namely $\sqrt{B}\eps^4 n^2$. Thus the expected total height of the slices in this interval at time $\ell$ is at most $\sqrt{B}\eps^{4-5u}n^2$. By Markov's inequality, it follows that with probability $1-O(\eps^u)$, those slices will not have a total height of more than $\eps^{4-6u}n^2$.  With \Cref{step:typheight_1} this gives assertion~\eqref{it:typical_height}: the height $h$ of the bottom exploration at time $\ell$ satisfies $h \le \eps^{4-6u}n^2$ with probability $1-O(\eps^u)$.

\noindent\steplabel{3}{step:typheight_3}\emph{Step 3. The connecting path.} The point $(\ell,h)$ is the vertex at the top of the bottom exploration at time $\ell$, and $(\ell,h')$ is in between it and $(\ell,0)$, on the same vertical line. We now join these last two points.  Apply \Cref{lem:upper_bound_distances} with $\sqrt{h'}$ in place of $r$ and $x=\eps^{-u}/3$.  Since $r=\sqrt{h'}$ is random, we apply it conditionally on the part of the LHPT to the left of time $\ell$.  We obtain a path from $(\ell,0)$ to $(\ell,r^2)=(\ell,h')$ contained in the bubble between these two points, hence in the bottom exploration and to the right of time $\ell$, of length less than $(2r+1)(\eps^{-u}/3+1)$.  On the event $h \le \eps^{4-6u}n^2$ we have $r \le \eps^{2-3u}n+1$, so that this length is at most
\[\frac{2r+1}{3}\eps^{-u}+(2r+1) \le \frac{2}{3}\eps^{2-4u}n+2\eps^{2-3u}n+\eps^{-u}+3<\eps^{2-4u}n\]
for all $\eps$ sufficiently small, since $\eps^{2-3u}=\eps^u\cdot\eps^{2-4u}$, $n \ge \eps^{-4}$ (see \Cref{fn:parameters}) and $u<1/10$.  The conditional probability that no such path is produced is at most $\exp(-c\eps^{-u})$ for a constant $c>0$, uniformly over the conditioning, since $r$ enters this bound only through the choice of $x$.  Combining with \Cref{step:typheight_1,step:typheight_2} proves the lemma.
\end{proof}

\subsection{Probability of a bubble being good}
\label{subsec:prob_bubble_good}

We will break the proof of \Cref{prop:good_bubble_probability} into the upper bound and lower bound.  As we will see just below, the proof of the upper bound is straightforward but the proof of the lower bound will require more work.

\begin{proof}[Proof of \Cref{prop:good_bubble_probability}, upper bound]  If the bubble is good, then by~\eqref{it:top_survival} the merging time is at least $\eps n$, so that $Y_j>0$ for all $j < \eps n$.  Moreover, the height process of the top exploration is stochastically dominated by a standard Galton--Watson process starting from height $B\eps^4 n^2$ (its initial height is $(B-\sqrt{B})\eps^4n^2 \le B\eps^4n^2$, and its downward jumps only decrease it).  Therefore,
\[p \le \p_{B\eps^4n^2}(X_{\eps n-1}>0)=1-\left(1-\frac{1}{(\eps n)^2}\right)^{B\eps^4n^2} \le \frac{B\eps^4n^2}{(\eps n)^2} = B\eps^2.\]
\end{proof}

For the lower bound, we will first need the following lemma.

\begin{lemma}
\label{lem:top_exploration_survives_long}
There exists a universal constant $c_{\ref{lem:top_exploration_survives_long}}>0$ so that, provided the parameters are chosen as in \Cref{fn:parameters} and $\eps$ is sufficiently small, the top exploration has probability at least $c_{\ref{lem:top_exploration_survives_long}}B\eps^2$ of merging with the bottom exploration after time $\eps n$.
\end{lemma}

\begin{proof}
\noindent\steplabel{1}{step:topsurv_1}\emph{Step 1. The bubble has probability at least $cB\eps^2$ of staying high up to time $2\eps n$.} By \Cref{lem:tail_bound_minimum}, the probability that the bubble of initial height $B\eps^4 n^2$ survives until time $2\eps n$ and never drops below $\sqrt{B}\eps^4 n^2$ before then is
\begin{align*}
\p_{B\eps^4 n^2}\!\left(\min_{0 \le j \le 2\eps n} X_j \ge \sqrt{B}\eps^4 n^2\right) &\ge \left(1-\frac{1}{(2\eps n+1)^2}\right)^{\sqrt{B}\eps^4 n^2}-\left(1-\frac{1}{(4\eps n+1)^2}\right)^{B\eps^4 n^2}\\
&\ge \left(1-\frac{\sqrt{B}\eps^4 n^2}{(2\eps n+1)^2}\right)-\left(1-\frac{B\eps^4 n^2}{(4\eps n+1)^2}+\frac{B^2 \eps^8 n^4}{2(4\eps n+1)^4}\right)\\
&\ge cB\eps^2
\end{align*}
for some positive constant $c$ (here $X_t$ is the height at time $t$ of the bubble of initial height $B\eps^4n^2$ started at $(0,0)$, which by \Cref{def:bubble} evolves as a Galton--Watson process started from $B\eps^4n^2$).  Indeed, using $(4\eps n+1)^{-2} \ge (4\eps n)^{-2}(1-(2\eps n)^{-1})$, $(2\eps n+1)^{-2} \le (2\eps n)^{-2}$ and $(4\eps n+1)^{-4} \le (4\eps n)^{-4}$, the right-hand side of the second line is at least
\[B\eps^2\left(\frac{1}{16}-\frac{1}{4\sqrt{B}}-\frac{B\eps^2}{512}-\frac{1}{32\eps n}\right),\]
the term in the parentheses is at least a universal constant for $B$ large, $B\eps^2 \le 1$ and $\eps n$ is large; the last of these holds for all small $\eps$ since $n \ge \eps^{-4}$.  The constant $c$ may therefore be taken to be universal.

\noindent\steplabel{2}{step:topsurv_2}\emph{Step 2. Reduction to the survival of a slice of the bottom exploration.} Now, suppose that this event happens, but the top exploration merges with the bottom before time $\eps n$. Then it means that some slice that is added in the bottom exploration before time $\eps n$ must have survived until time $2 \eps n$: at the merging time $t<\eps n$ the top exploration is empty, so every tree lying below the bottom-most geodesic started from $(0,B\eps^4n^2)$ at time $t$ belongs to a slice of the bottom exploration added at some time $\le t$; since $X_{2\eps n} \ge \sqrt{B}\eps^4n^2>0$, one of these trees has a descendant at time $2\eps n$. If we can show that the probability the latter happens is less than $c'B\eps^2$ for a constant $c'<c$, then this will prove the lemma.

\noindent\steplabel{3}{step:topsurv_3}\emph{Step 3. The probability that a slice of the bottom exploration survives until time $2\eps n$.} Each slice started before time $\eps n$ has probability at most
\[1-\left(1-\frac{1}{(\eps n+1)^2}\right)^{\sqrt{B}\eps^4 n^2} \le \sqrt{B}\eps^2\]
of surviving until time $2\eps n$. Consider the time interval between $\sqrt[4]{B}2^{k/2} \eps^2 n$ and $\sqrt[4]{B}2^{(k+1)/2} \eps^2 n$; since
\[T_{\text{max}}=\sqrt[4]{B}2^{(k_{\text{max}}+1)/2}\eps^2 n>B^{1/8}\eps n> \eps n,\]
every slice added before time $\eps n$ is added in one of those intervals for some $1 \le k \le k_{\text{max}}$. As already noted in the proof of \Cref{prop:bottom_exploration_good}, there are at most $2^{k(u+1/2)+1}$ slices added in this interval. Let $k_0$ be a large positive integer. We once again break things into two cases, according to whether $k$ is larger or smaller than $k_0$.

\noindent\emph{Case 1. The early stages $k<k_0$.} By a union bound, the probability that one of the slices between $\sqrt[4]{B}2^{k/2} \eps^2 n$ and $\sqrt[4]{B}2^{(k+1)/2} \eps^2 n$ survives until $2\eps n$ is at most $\sqrt{B}2^{k(u+1/2)+1}\eps^2$. Then taking a union bound over $k<k_0$ gives that with probability at most $\sqrt{B}k_0 2^{k_0(u+1/2)+1}\eps^2$, one of the slices added before time $\sqrt[4]{B}2^{k_0/2}\eps^2 n$ (including the initial slice, whose survival probability satisfies the same bound) survives until time $2\eps n$.

\noindent\emph{Case 2. The later stages $k \ge k_0$.} The event that the bubble survives until time $\sqrt[4]{B}2^{k/2}\eps^2 n$ is measurable with respect to the skeleton up to that time, while the survival of a slice added at a later time concerns the fresh subtrees rooted at that later time; by the branching property, the two events are therefore independent, and their probabilities may be multiplied. The first of these has probability
\[1-\left(1-\frac{1}{(\sqrt[4]{B}2^{k/2}\eps^2 n+1)^2}\right)^{B\eps^4 n^2} \le \frac{\sqrt{B}}{2^k}.\] 
As already explained above, there are at most $2^{k(u+1/2)+1}$ slices in this interval, and each one has probability no larger than $\sqrt{B}\eps^2$ of surviving until $2\eps n$. So by a union bound, the probability that the bubble survives until time $\sqrt[4]{B}2^{k/2}\eps^2 n$ and one of the slices in the interval from $\sqrt[4]{B}2^{k/2} \eps^2 n$ to $\sqrt[4]{B}2^{(k+1)/2} \eps^2 n$ survives until $2\eps n$ is at most
\[\frac{\sqrt{B}}{2^k} \cdot 2^{k(u+1/2)+1}\cdot \sqrt{B}\eps^2=B\eps^2 \cdot 2^{k(u-1/2)+1}.\]
Taking a union bound over $k \ge k_0$ yields a bound of
\[B\eps^2 \sum_{k=k_0}^{k_{\max}} 2^{k(u-1/2)+1} \le CB\eps^2 2^{k_0(u-1/2)+1}\]
for a universal constant $C$, since $u<1/2$ makes the sum a convergent geometric series with ratio $2^{u-1/2}$ bounded away from~$1$.

\noindent\steplabel{4}{step:topsurv_4}\emph{Step 4. Conclusion.} Combining these cases, we therefore have that the probability that the bubble never drops below $\sqrt{B}\eps^4 n^2$ before time $2\eps n$ and the top merges with the bottom before time $\eps n$ is at most
\[\left(\sqrt{B}k_0 2^{k_0(u+1/2)+1}+CB 2^{k_0(u-1/2)+1}\right)\eps^2.\]
By picking $k_0$ very large, and then making sure that $B$ is much larger than $2^{k_0}$, this is less than $c'B\eps^2$ for a constant $c'$ that can be made as small as we want. Indeed, the second term is at most $c'B\eps^2/2$ once $2^{k_0(1/2-u)}$ is large compared with $1/c'$, which fixes $k_0$; the first is then at most $c'B\eps^2/2$ once $\sqrt{B}$ is large compared with $k_02^{k_0(u+1/2)}/c'$, that is, once $B \ge C_0k_0^2 2^{k_0(1+2u)}$, the constraint recorded in \Cref{fn:parameters}.  Taking $c'=c/2$, the probability that the top exploration merges after time $\eps n$ is therefore at least $(c-c')B\eps^2=cB\eps^2/2$, which proves the lemma with $c_{\ref{lem:top_exploration_survives_long}} = c/2$.
\end{proof}

We now prove the lower bound of \Cref{prop:good_bubble_probability}.
\begin{proof}[Proof of \Cref{prop:good_bubble_probability}, lower bound]

We must show that there is a constant $c_{\ref{prop:good_bubble_probability}}>0$ such that with probability at least $c_{\ref{prop:good_bubble_probability}}B \eps^2$, all the conditions in \Cref{def:bottom_good,def:top_good,def:bubble_good} hold.

\noindent\steplabel{1}{step:goodbubble_1}\emph{Step 1. Discarding condition~\eqref{it:max_height_top} of \Cref{def:top_good} and the upper bound on the merging time in~\eqref{it:top_survival} of \Cref{def:top_good}.} We begin with a few simplifications. First, by \Cref{lem:tail_bound_maximum}, together with the fact that the height process of the top exploration is stochastically dominated by a Galton--Watson process started from $B\eps^4n^2$, the probability that the height process of the top exploration ever exceeds $R\eps^2n^2/4$, i.e.\ that~\eqref{it:max_height_top} fails, is at most (the hypothesis $x \ge y$ of that lemma holding because $R/4 \ge B\eps^2$ for $\eps$ small)
\[\p_{B\eps^4 n^2}\!\left(\max_{j \ge 0} X_j \ge \frac{R\eps^2 n^2}{4}\right) \le \frac{4B \eps^2}{R}.\]
This is negligible compared to $c_{\ref{prop:good_bubble_probability}}B \eps^2$ if $R$ is large. Next, the height process of the top exploration is stochastically dominated by a Galton--Watson process with offspring distribution $\theta$, so the probability that it merges with the bottom after time $R\eps n$ is at most
\[\p_{B\eps^4 n^2}(X_{R\eps n}>0)=1-\left(1-\frac{1}{(R\eps n+1)^2}\right)^{B\eps^4 n^2}\le \frac{B\eps^2}{R^2}.\]
Once again, this is negligible compared to $c_{\ref{prop:good_bubble_probability}}B \eps^2$.

\noindent\steplabel{2}{step:goodbubble_2}\emph{Step 2. Discarding condition~\eqref{it:initial_distance}.} We next take care of condition~\eqref{it:initial_distance} of \Cref{def:top_good}, namely that the distance between $(0,B\eps^4n^2)$ and $(0,\sqrt{B}\eps^4n^2)$ \emph{in the top exploration} is at most $\eps^{2-u}n$. The first slice to be removed from the top exploration and added to the bottom exploration is removed at the time $t_1$, the smallest multiple of $\sqrt[4]{B}2^{-u}\eps^2n$ which is at least $\sqrt[4]{B}2^{1/2}\eps^2 n$, so that $t_1 \ge \sqrt{2}\sqrt[4]{B}\eps^2 n$. Up to time $t_1$, but not after, the top exploration coincides with the region of the LHPT between the bottom-most geodesics started from $(0,\sqrt{B}\eps^4n^2)$ and $(0,B\eps^4n^2)$.  The path produced by \Cref{lem:upper_bound_distances} at scale $r$ has horizontal extent at most $r$, and $r=\sqrt{B-\sqrt{B}}\,\eps^2 n$ would exceed $t_1$ for large $B$; we therefore apply it at a smaller scale, splitting $[\sqrt{B}\eps^4n^2, B\eps^4n^2]$ into $M$ equal parts.  Set
\[M=\lceil B \rceil, \qquad \sigma=\sqrt{\frac{B-\sqrt{B}}{M}}\,\eps^2 n, \qquad h_i=\sqrt{B}\eps^4 n^2+(i-1)\sigma^2 \quad (1 \le i \le M+1),\]
so that $h_1=\sqrt{B}\eps^4n^2$, $h_{M+1}=B\eps^4 n^2$ and
\[\sigma \le \sqrt{\frac{B-\sqrt{B}}{B}}\,\eps^2 n \le \eps^2 n<\sqrt{2}\sqrt[4]{B}\eps^2 n \le t_1.\]
Apply \Cref{lem:upper_bound_distances} to each of the $M$ pairs $(0,h_i)$, $(0,h_{i+1})$, with $\sigma$ in place of $r$ and $\lceil \eps^{-u/2}\rceil$ in place of $x$, and use invariance of the law of the LHPT under vertical translations.  By a union bound, with probability at least $1-M\exp(-c\eps^{-u/2}) \ge 1-\exp(-c'\eps^{-u/2})$ we obtain, simultaneously for all $i$, a path from $(0,h_i)$ to $(0,h_{i+1})$ of length less than $(2\sigma+1)x$ and horizontal extent at most $\sigma$, contained in the bubble between the bottom-most geodesics started from these two points.  Each such bubble lies between the bottom-most geodesics started from $(0,\sqrt{B}\eps^4n^2)$ and $(0,B\eps^4n^2)$, and each path lies to the left of time $\sigma<t_1$, so each is contained in the top exploration.  Concatenating them gives a path from $(0,\sqrt{B}\eps^4n^2)$ to $(0,B\eps^4n^2)$ inside the top exploration, of length less than
\[M(2\sigma+1)x \le 2\sqrt{M(B-\sqrt{B})}\,\eps^2 n\,x+Mx \le CB\eps^{2-u/2}n+CB\eps^{-u/2} \le \eps^{2-u}n\]
for all $\eps$ sufficiently small, where $C$ is universal; here $\sqrt{M(B-\sqrt{B})}$ and $M$ are both at most $CB$ and $x \le 2\eps^{-u/2}$, and the last inequality holds because $\eps^{u/2}$ is small compared with $1/B$ and $n \ge \eps^{-4}$ (see \Cref{fn:parameters}).  As $\exp(-c'\eps^{-u/2})$ is of much smaller order than $B\eps^2$, we may ignore~\eqref{it:initial_distance} as well.

\noindent\steplabel{3}{step:goodbubble_3}\emph{Step 3. Condition~\eqref{it:merge_distance}, and reduction to four events.} Finally, let $E_1$ be the event that the bottom exploration is good \emph{and} the distance condition in~\eqref{it:merge_distance} of \Cref{def:bubble_good} holds.  The merging time $\ell$ is measurable with respect to the top exploration, which is independent of the bottom exploration, and by \Cref{lem:bottom_exploration_typical_height} (applied conditionally on the top exploration, with a bound uniform over $\ell \ge \eps n$) that distance condition fails with conditional probability $O(\eps^u)$.  Hence, conditionally on any top-exploration event on which $\ell \ge \eps n$, the event $E_1$ has probability at least $p_{\ref{prop:bottom_exploration_good}} - O(\eps^u) \ge p_{\ref{prop:bottom_exploration_good}}/2$ for $\eps$ small.  (Both~\eqref{it:merge_distance} and the goodness of the bottom exploration concern the bottom exploration, so their probabilities cannot be multiplied; this conditional bound is how~\eqref{it:merge_distance} is incorporated below.)

In summary, it thus remains to show that the probability of the intersection of the following four events is at least $cB\eps^2$ for some constant $c>0$:
\begin{itemize}
\item $E_1$: The bottom exploration is good \emph{and} the distance condition in~\eqref{it:merge_distance} holds;
\item $E_2$: The top exploration merges with the bottom exploration after time $\eps n$;
\item $E_3$: The height process of the top exploration at time $\eps n/2$ is at least $\eps^2n^2/R$ (i.e., condition~\eqref{it:middle_height} holds);
\item $E_4$: Condition~\eqref{it:middle_graph_distance} holds, i.e.\ the distance between $c_1$ and the bottom geodesic of the bubble is at least $\sqrt{D}\eps n$.
\end{itemize}

\noindent\steplabel{4}{step:goodbubble_4}\emph{Step 4. Splitting the bubble at time $\eps n/4$.} Recall that $(Y_j)_{j \ge 0}$ denotes the height process for the top exploration, described before \Cref{def:top_good}. We of course have that $E_2$ is the event that $Y_{\eps n}>0$.

We have (note that $E_2 \subseteq \{Y_{\eps n/4}>0\}$)
\[\p(E_1 \cap E_2 \cap E_3 \cap E_4)=\p(Y_{\eps n/4}>0) \p(E_1 \cap E_2 \cap E_3 \cap E_4 \giv Y_{\eps n/4}>0).\]
By \Cref{lem:top_exploration_survives_long},
\[\p(Y_{\eps n/4}>0) \ge \p(Y_{\eps n}>0) \ge c_{\ref{lem:top_exploration_survives_long}} B \eps^2.\]
So now we must show that $\p(E_1 \cap E_2 \cap E_3 \cap E_4 \giv Y_{\eps n/4}>0) \ge c$ for some positive constant $c$. Recall that $E_3$ is the event $Y_{\eps n/2}\ge \eps^2 n^2/R$. Then
\begin{align*}
\p(E_1 \cap E_2 \cap E_3 \cap E_4 \giv Y_{\eps n/4}>0)&\ge \p(E_1 \cap E_2 \cap E_3 \giv Y_{\eps n/4}>0)-\p(E_3 \cap E_4^c \giv Y_{\eps n/4}>0)\\
&=\Pi_1-\Pi_2.
\end{align*}

\noindent\steplabel{5}{step:goodbubble_5}\emph{Step 5. Upper bound for $\Pi_2$.} Let $\kappa>0$, let $K=K(\kappa)$ be the integer provided by \Cref{lem:lower_bound_distances}, and set $r=\sqrt{D}\eps n$; see \Cref{fig:high distance point}.  We assume that
\[D \le \min\left(\frac{1}{144},\ \frac{1}{KR}\right),\]
which is possible since in the order of \Cref{fn:parameters} $R$ is fixed before $D$, while $\kappa$ (hence $K$) is chosen in \Cref{step:goodbubble_9} in terms of $u$ alone.  With these choices $3r \le \eps n/4$, so that all of the paths considered below lie at times greater than $\eps n/4$, and $\eps n/2 \ge 3r$, as \Cref{lem:lower_bound_distances} requires.

On $E_3$ the vertex $c_1$ has height at least $Y_{\eps n/2} \ge \eps^2n^2/R \ge Kr^2$, since its height is the sum of the heights of the two explorations at time $\eps n/2$ and since $Kr^2=KD\eps^2n^2 \le \eps^2n^2/R$.  If $E_4$ fails, then some path of length less than $r$ started at $c_1$ visits a vertex $(t,h)$ with $h \le 0$.

A vertex of height $h<0$ lies strictly below the bottom-most geodesic $\Z_{\ge0}\times\{0\}$ started from the origin, whereas $c_1$ lies strictly above it, and a path cannot cross that geodesic without visiting one of its vertices; so the path visits a vertex $(t,0)$, necessarily with $|t-\eps n/2|<r$.  On $E_3$ this possibility is therefore contained in the event $N_{\downarrow}$ that some path of length less than $r$ started at a vertex $(\eps n/2,j)$ with $j \ge Kr^2$ visits a vertex of $\Z_{\ge0}\times\{0\}$.  Reversing the paths and applying \Cref{lem:lower_bound_distances} with $h_0=0$, this $r$ and $t_0=\eps n/2$ gives $\p(N_{\downarrow}) \le \kappa$.  Moreover $N_{\downarrow}$ is measurable with respect to $\CL_{[\eps n/4,\infty)}$: the paths involved lie at times greater than $\eps n/4$, and the height labeling to the right of that time is determined by $\CL_{[\eps n/4,\infty)}$, the bottom-most geodesic from the origin passing through $(\eps n/4,0)$.  Since $\{Y_{\eps n/4}>0\}$ is measurable with respect to $\CL_{[0,\eps n/4]}$, which is independent of $\CL_{[\eps n/4,\infty)}$, the conditioning may be removed, and this possibility contributes at most $\kappa$ to $\Pi_2$. We obtain that $\Pi_2 \le \kappa$, which is as small as we like once $\kappa$, and then $D$, is small.

\noindent\steplabel{6}{step:goodbubble_6}\emph{Step 6. Lower bound for $\Pi_1$.} By the above upper bound for $\Pi_2$, it now suffices to show that $\Pi_1$ is bounded below by a positive constant. The events $E_2$, $E_3$ and $\{Y_{\eps n/4}>0\}$ are measurable with respect to the top exploration, and on their intersection $\ell \ge \eps n$; so by the conditional bound of \Cref{step:goodbubble_3},
\[\Pi_1=\p\!\big(E_1 \giv E_2 \cap E_3 \cap \{Y_{\eps n/4}>0\}\big)\,\p(E_2 \cap E_3 \giv Y_{\eps n/4}>0) \ge \frac{p_{\ref{prop:bottom_exploration_good}}}{2} \cdot \p(E_2 \cap E_3 \giv Y_{\eps n/4}>0).\] So now we can focus on lower bounding $\p(E_2 \cap E_3 \giv Y_{\eps n/4}>0)$. We have
\[\p(E_2 \cap E_3 \giv Y_{\eps n/4}>0)=\p(E_2 \giv Y_{\eps n/4}>0)-\p(E_2 \cap E_3^c \giv Y_{\eps n/4}>0)=\Pi_3-\Pi_4.\]

\noindent\steplabel{7}{step:goodbubble_7}\emph{Step 7. Lower bound for $\Pi_3$.} As already discussed above, the height process for the top exploration is stochastically dominated by a Galton--Watson process with offspring distribution $\theta$. Thus
\[\p(Y_{\eps n/4}>0) \le \p_{B\eps^4 n^2}(X_{\eps n/4}>0)=1-\left(1-\frac{1}{(\eps n/4+1)^2}\right)^{B\eps^4 n^2} \le 16B\eps^2.\]
By \Cref{lem:top_exploration_survives_long}, $\p(Y_{\eps n}>0) \ge c_{\ref{lem:top_exploration_survives_long}} B\eps^2$, and so
\[\Pi_3=\frac{\p(Y_{\eps n}>0)}{\p(Y_{\eps n/4}>0)} \ge \frac{c_{\ref{lem:top_exploration_survives_long}}}{16}.\]

\noindent\steplabel{8}{step:goodbubble_8}\emph{Step 8. Upper bound for $\Pi_4$.} By the Markov property of the $Y$ process, 
\begin{align*}
\Pi_4&=\p(E_2 \cap E_3^c \giv Y_{\eps n/4}>0)=\sum_{1 \le j < \eps^2 n^2/R} \p(E_2 \cap \{Y_{\eps n/2}=j\} \giv Y_{\eps n/4}>0)\\
&=\sum_{1 \le j < \eps^2 n^2/R} \p(E_2 \giv Y_{\eps n/2}=j) \p(Y_{\eps n/2}=j \giv Y_{\eps n/4}>0).
\end{align*}
Once again, using the fact that $Y$ is stochastically dominated by a Galton--Watson process and monotonicity in the initial height,
\[\p(E_2 \giv Y_{\eps n/2}=j) \le \p_j(X_{\eps n/2}>0) \le \p_{\eps^2 n^2/R}(X_{\eps n/2}>0)\]
for all $1 \le j < \eps^2 n^2/R$. Therefore
\begin{align*}
\Pi_4 &\le \p_{\eps^2 n^2/R}(X_{\eps n/2}>0)\p\!\left(1 \le Y_{\eps n/2}<\frac{\eps^2 n^2}{R} \giv Y_{\eps n/4}>0\right)\\
&\le \p_{\eps^2 n^2/R}(X_{\eps n/2}>0)\\
&=1-\left(1-\frac{1}{(\eps n/2+1)^2}\right)^{\eps^2 n^2/R} \le \frac{4}{R}.
\end{align*}
This can be made arbitrarily small if $R$ is large, which completes the bounds on $\Pi_1$ through $\Pi_4$.

\noindent\steplabel{9}{step:goodbubble_9}\emph{Step 9. Conclusion.} Write $p_0' = p_{\ref{prop:bottom_exploration_good}}/2$ and $c_0' = c_{\ref{lem:top_exploration_survives_long}}$, so that $p_0'$ depends only on $u$ and $c_0'$ is universal.  Choose $R$ large enough that $4/R \le c_0'/32$, then $\kappa = p_0'\,c_0'/128$ in \Cref{step:goodbubble_5}, and then $D$ small enough for the three bounds there, so that $\Pi_2 \le \kappa \leq p_0'\,c_0'/64$.  Then $\Pi_3-\Pi_4 \ge c_0'/16 - c_0'/32 = c_0'/32$, so that $\Pi_1 \ge p_0'(\Pi_3-\Pi_4) \ge p_0'c_0'/32$ by \Cref{step:goodbubble_6}, and hence, for all $\eps$ sufficiently small,
\[\p(E_1 \cap E_2 \cap E_3 \cap E_4 \giv Y_{\eps n/4}>0) \ge \Pi_1 - \Pi_2 \ge \frac{p_0'\,c_0'}{32}-\frac{p_0'\,c_0'}{64}=\frac{p_0'\,c_0'}{64}.\]
Combining with the two displays of \Cref{step:goodbubble_4},
\[\p(E_1 \cap E_2 \cap E_3 \cap E_4) \ge c_0'B\eps^2 \cdot \frac{p_0'\,c_0'}{64} =: c_1 B \eps^2.\]
The event $E_1 \cap E_2 \cap E_3 \cap E_4$ encodes conditions~\eqref{it:bottom_good}, \eqref{it:merge_distance} and~\eqref{it:middle_graph_distance} of \Cref{def:bubble_good}, together with~\eqref{it:middle_height} and the first half of~\eqref{it:top_survival} of \Cref{def:top_good}.  It remains to take into account the conditions set aside in \Cref{step:goodbubble_1,step:goodbubble_2}, by subtracting the failure probabilities $4B\eps^2/R$ for~\eqref{it:max_height_top}, $B\eps^2/R^2$ for the second half of~\eqref{it:top_survival}, and $\exp(-c'\eps^{-u/2})$ for~\eqref{it:initial_distance}:
\[p \ge c_1B\eps^2 - \frac{4B\eps^2}{R}-\frac{B\eps^2}{R^2}-\exp(-c'\eps^{-u/2}) \ge \frac{c_1}{2}B\eps^2\]
provided $R$ was also taken large enough that $4/R+1/R^2 \le c_1/4$, and $\eps$ is small enough.  The two requirements on $R$ are compatible, since $c_1 = p_0'(c_0')^2/64$ depends only on $u$.
\end{proof}

\subsection{Probability of at least one good bubble and failures}
\label{subsec:no_good_bubbles}

Suppose that we have fixed parameters $u$, $R$, $D > 0$, $B > 1$ and $\eps > 0$ as in \Cref{def:bubble_good} for a bubble to be good.  Let $E$ be the event that there exists an integer $\eps^{-2}/3 \le j \le 2\eps^{-2}/3$ so that the bubble started at time $j \eps^2 n$ is good (defined exactly as in \Cref{subsec:bottom_top_explorations}, but with the two explorations started from $\{t\} \times [0, B\eps^4n^2]$ in place of $\{0\} \times [0,B\eps^4n^2]$; by the Markov property of the LHPT its law does not depend on $t$). Although the probability of a given bubble being good is quite small, the experiment gets repeated a number of times of order $\eps^{-2}$. By adjusting the parameter $B$ accordingly, we can make the probability of the event $E$ that we observe at least one good bubble as close to $1$ as we want. Here is the precise statement.

\begin{proposition}
\label{prop:no_good_bubbles}
There is a constant $c_{\ref{prop:no_good_bubbles}} > 0$ which depends only on $u$ (and not on $B, R$ or $D$) such that for all $\eps$ sufficiently small,
\[\p(E^c)\le c_{\ref{prop:no_good_bubbles}}  \left(\eps+\frac{1}{\sqrt{B}}\right).\]
In particular, for any $q<1$, we can choose $B$ sufficiently large so that for all sufficiently small $\eps$, the probability of seeing at least one good bubble is at least $q$.
\end{proposition}
\begin{proof}
\noindent\steplabel{1}{step:nogoodbub_1}\emph{Step 1. A renewal sequence of bubbles.} We begin by defining a sequence of random times $(S_k)_{k \ge 0}$, which will all be of the form $j\eps^2n$ for some $j \ge \eps^{-2}/3$, as follows. First we let $S_0=n/3$. Now assume that $S_{k-1}$ has been defined for some $k \ge 1$. We then define $S_k$ to be the smallest time of the form $j\eps^2n$ which is larger than or equal to the time when the bubble started at $S_{k-1}$ dies, i.e.\ larger than or equal to $S_{k-1}$ plus the lifetime of that bubble.  Note that the bubbles defined at the times $S_k$ are independent of each other (by the strong Markov property: $S_k$ is determined by the LHPT to the left of $S_k$, and by the branching property, the bubble started at $S_k$ is a function of fresh randomness).  In particular, the increments $(S_k-S_{k-1})_{k \ge 1}$ are i.i.d., and have the same law as
\[\left \lceil \frac{T}{\eps^2n} \right \rceil \eps^2n\]
where $T$ denotes the lifetime of a bubble of initial height $B\eps^4n^2$: by \Cref{def:bubble} its height is a Galton--Watson process started from $B\eps^4n^2$, so that $T$ is the extinction time of $X$ under $\p_{B\eps^4n^2}$ and its law is given by~\eqref{eq:bubble_lifetime}. We write $G$ for the event that this bubble is good, so that $p:=\p(G)$ is the probability appearing in \Cref{prop:good_bubble_probability}.

Next, let $N$ be the smallest integer $k$ such that the bubble started at $S_k$ is good, and let $\tau=S_{N+1}$. Thus $\tau$ is the first time of the form $j\eps^2n$ by which a good bubble started at one of the $S_k$ has been observed in its entirety.

\noindent\steplabel{2}{step:nogoodbub_2}\emph{Step 2. Reduction to an estimate for $\E(\tau-n/3)$.} If there are no good bubbles on $[n/3, 2n/3]$, then we must have $\tau \ge 2n/3$ (indeed, none of the bubbles started at the times $S_k \in [n/3,2n/3]$ is then good, so $S_N > 2n/3$). So by Markov's inequality,
\begin{equation}\label{eq:markov_no_good_bubbles}
\p(E^c) \le \p\!\left(\tau \ge \frac{2n}{3}\right) \le \frac{3\E(\tau-n/3)}{n}.
\end{equation}

The law of $\tau$ is the same as that of
\[\left \lceil \frac{W}{\eps^2n} \right \rceil \eps^2n+\sum_{k=1}^{N} \left \lceil \frac{V_k}{\eps^2n} \right \rceil \eps^2n+\frac{n}{3},\]
where $N$ is a geometric random variable on $\{0,1, \dots\}$ with success probability $p=\p(G)$, the $V_k$'s are i.i.d.\ with the law of $T$ conditioned on $G^c$, and $W$ has the law of $T$ conditioned on $G$, with all these variables independent.

\noindent\steplabel{3}{step:nogoodbub_3}\emph{Step 3. Estimates for $\E(N)$, $\E(V_1)$ and $\E(W)$.} By \Cref{prop:good_bubble_probability},
\[\E(N)=\frac{1}{p}-1\le \frac{1}{c_{\ref{prop:good_bubble_probability}} B\eps^2}.\]
Moreover,
\begin{align*}
\E(V_1)
&=\frac{\E(T \one_{G^c})}{1-p}
 \le \frac{\E_{B\eps^4n^2}(T)}{1-B\eps^2} \quad\text{(by \Cref{prop:good_bubble_probability})}\\
&\le c\sqrt{B}\eps^2n \quad\text{(by \Cref{lem:expected_length_of_bubble})}
\end{align*}
for all $\eps > 0$ sufficiently small and a constant $c > 0$. Finally,
\begin{align*}
\E_{B\eps^4n^2}(T \cdot \one_{\{T \ge \eps n\}})&\le \sum_{k<\eps n} \p_{B\eps^4n^2}(T \ge \eps n)+\sum_{k \ge \eps n} \p_{B\eps^4n^2}(T \ge k)\\
&\le B\eps^3n+\sum_{k \ge \eps n} \frac{B\eps^4n^2}{k^2} \quad\text{(by~\eqref{eq:bubble_lifetime}, \eqref{eq:bernoulli})}\\
&\le 3B\eps^3n,
\end{align*}
so that by \Cref{prop:good_bubble_probability} again,
\[\E(W)=\frac{\E(T \one_{G})}{p} \le \frac{\E_{B\eps^4n^2}(T \cdot \one_{\{T \ge \eps n\}})}{p} \le c \eps n,\]
possibly increasing the value of $c > 0$; here we used that $G \subseteq \{T \ge \eps n\}$: on $G$ the two explorations merge at a time $\ell \ge \eps n$ by~\eqref{it:top_survival} of \Cref{def:top_good}, and so the top exploration has positive height and the bubble has not died by time $\eps n$.

\noindent\steplabel{4}{step:nogoodbub_4}\emph{Step 4. Conclusion.} Putting all this together, we thus find
\begin{equation}\label{eq:prob_no_good_bubbles}
\begin{aligned}
\E(\tau-n/3)&\le \E\!\left(W+\sum_{k=1}^{N} V_k +(N+1)\eps^2n\right)=\E(W)+\E(N)\E(V_1)+\E(N+1)\eps^2n\\
&\le c n\left(\eps+\frac{1}{\sqrt{B}}+\frac{1}{B}+\eps^2\right),
\end{aligned}
\end{equation}
possibly increasing $c > 0$ further at the end.
The proposition follows from this and~\eqref{eq:markov_no_good_bubbles}, upon using $1/B \le 1/\sqrt{B}$ and $\eps^2 \le \eps$ to absorb the last two terms into the first two, and taking $c_{\ref{prop:no_good_bubbles}}=6c$.
\end{proof}

When we establish the lower bound on the distortion in the next section, we think of exploring the LHPT from left to right looking for good bubbles at times of the form $j \eps^2 n$.  We refer to the situation where there are no good bubbles as ``failures''.  In this case, we will split the starting segment of length $n$ into six equal \emph{sub-segments} of length $n/6$ each, and redo the exploration on each of them independently (with the initial height and required maximal height rescaled appropriately).  We will also need to add an extra chance of there being a failure even if there are good bubbles present.  This is because we need to estimate the law of the exploration conditional on a failure having occurred at the previous level, and the knowledge that there are no good bubbles gives us too much information. By \Cref{prop:no_good_bubbles}, the probability of seeing no good bubbles is at most $c_{\ref{prop:no_good_bubbles}} (\eps+1/\sqrt{B})$. So every time we perform the exploration on a segment, we sample an independent $\Ber(c_{\ref{prop:no_good_bubbles}}/\sqrt{B})$, and if this Bernoulli is 1, then we declare that a failure has occurred (note that for $B$ large, $c_{\ref{prop:no_good_bubbles}}/\sqrt{B}$ is less than 1). This way, the added probability of a failure has the same order of magnitude as that of seeing no good bubbles, so when we condition on a failure having occurred, it is not possible to determine if it was because there were no good bubbles.

\begin{lemma}[Conditional law of the exploration during a failure]
\label{lem:conditional_prob_given_failure}
There exists a constant $c_{\ref{lem:conditional_prob_given_failure}} > 0$ which depends only on $u$ (and not on $B, R$ or $D$) so that the following is true.  Let $F$ be the event that there was a failure and let $A$ be any positive probability event for the LHPT restricted to one of the six sub-segments of the starting segment.  Then
\[\p(A  \giv F ) \le c_{\ref{lem:conditional_prob_given_failure}} \p(A).\]
\end{lemma}
\begin{proof}
\noindent\steplabel{1}{step:condfail_1}\emph{Step 1. Bayes' rule and a lower bound for $\p(F)$.} By Bayes' rule,
\[\p(A \giv F) = \frac{\p(F \giv A)}{\p(F)} \p(A).\]
As explained above, we have that
\[\p(F) \ge \frac{c_1}{\sqrt{B}}\]
(due to the added chance of a failure) where we can take $c_1 = c_{\ref{prop:no_good_bubbles}}$.

\noindent\steplabel{2}{step:condfail_2}\emph{Step 2. An event $E_J^c$ which contains $F$ up to the extra Bernoulli, and is independent of $A$.} Now given $A$, the extra information that we get is at most that there are no good bubbles on the particular sub-segment of the starting segment that we are working with. We formalize this as follows.  Write $J=[a,b]$ for the time window of the sub-segment to which $A$ refers, and set $J^+=[a-2R\eps n, b)$.  Whether the bubble started at a time $t$ is good is determined by the restriction of the LHPT to the time window $[t, t+2R\eps n]$.  All of the conditions in \Cref{def:bubble_good} concern the two explorations up to time $R\eps n$ after $t$ (the merging time is at most $R\eps n$ after $t$ when the bubble is good, by~\eqref{it:top_survival} of \Cref{def:top_good}, and the height bounds of~\eqref{it:bottom_height_bound} of \Cref{def:bottom_good} exactly up to that time) together with paths of length at most $\eps^{2-4u}n$ or $\sqrt{D}\eps n$ started within this range, which move at most that much in the time direction and so stay inside the stated window.  In particular this window is disjoint from $J$ when $t \notin J^+$.  Let $E_J$ be the event that some good bubble is started at a time $t \in [n/3,2n/3]\setminus J^+$, the analog of the event $E$ of \Cref{prop:no_good_bubbles} for this smaller collection of starting times.  Then $F \subseteq E_J^c \cup \{\text{the extra Bernoulli equals } 1\}$, so that
\[\p(F \giv A) \le \p(E_J^c \giv A)+\frac{c_{\ref{prop:no_good_bubbles}}}{\sqrt{B}}.\]
By the preceding paragraph $\one_{E_J^c}$ is the product of a function of $\CL_{[0,a]}$ and a function of $\CL_{[b,\infty)}$, whereas $\one_A$ is a function of $\CL_{[a,b]}$.  By the branching property of the skeleton, $\CL_{[a,\infty)}$ is independent of $\CL_{[0,a]}$ and $\CL_{[b,\infty)}$ is independent of $\CL_{[0,b]}$; conditioning successively, these two statements give the \emph{joint} independence of $\CL_{[0,a]}$, $\CL_{[a,b]}$ and $\CL_{[b,\infty)}$.  Consequently $E_J^c$ is independent of $A$, and $\p(E_J^c \giv A) = \p(E_J^c)$.

\noindent\steplabel{3}{step:condfail_3}\emph{Step 3. An upper bound for $\p(E_J^c)$.} To bound $\p(E_J^c)$ we run the chain $(S_k)$ from the proof of \Cref{prop:no_good_bubbles}, modified so as to skip over $J^+$: whenever the next starting time would fall in $J^+$, the next bubble is instead started at the first time of the form $j\eps^2n$ which is at least $b$.  Each $S_k$ is still a stopping time for the exploration of the LHPT from left to right, so by the strong Markov property the bubbles of the modified chain are still independent, with the increments distributed as before (the single skip contributes no increment).  On $E_J^c$ none of these bubbles started in $[n/3,2n/3]$ is good, so the first good one is started after time $2n/3$; as the only part of $[n/3,2n/3]$ which the chain does not traverse by means of bubble lifetimes is the skipped interval, of length at most $|J^+|+\eps^2n=n/6+2R\eps n+\eps^2n$, the accumulated increments must sum to at least
\[\frac{n}{3}-\frac{n}{6}-2R\eps n-\eps^2n \ge \frac{n}{7}\]
for $\eps$ sufficiently small.  Hence, in place of~\eqref{eq:markov_no_good_bubbles}, Markov's inequality gives
\[
\p(E_J^c)\le \p\!\left(\tau - \frac{n}{3} \ge \frac{n}{7}\right) \le \frac{7\E(\tau-n/3)}{n},
\]
with $\tau$ having the exact same distribution. Therefore the rest of the calculations in \Cref{prop:no_good_bubbles} are exactly the same, and we get the same bound up to a constant factor.

\noindent\steplabel{4}{step:condfail_4}\emph{Step 4. Conclusion.} So altogether, using~\eqref{eq:prob_no_good_bubbles} and $1/B+\eps^2 \le 1/\sqrt{B}+\eps$, for constants $c_2,c_3 > 0$ we have
\[\p(A \giv F ) \le \frac{2 c_2\left(\eps+ 1/\sqrt{B}\right)+ c_1/\sqrt{B}}{c_1/\sqrt{B}} \p(A)=\left(\frac{2c_2}{c_1}\left(\eps \sqrt{B}+1\right)+1\right)\p(A)\le c_3\p(A)\]
if $\eps$ is sufficiently small.
\end{proof}

\section{The $\ell^2$ distortion for the LHPT}
\label{sec:distortion LHPT}

We now use our work on the exploration process to obtain a distortion lower bound for certain subsets of the LHPT. Before we do so, we first need to introduce some terminology. Throughout this section we abbreviate the graph distance in the LHPT by $d = d_{\mathrm{gr}}^{\CL}$. For a geodesic segment from $a$ to $b$, we define the \emph{parallelogram} above this segment to be the set of points in $\Z^2$ whose time coordinate lies between those of $a$ and $b$, and whose height above the segment is at least $0$ and no more than $R\eps^2d(a,b)^2$. We denote this set by $P_{a,b}$.

Suppose $c$ and $c'$ are points on the \textit{top} geodesic of the same good bubble from the exploration of the segment from $a$ to $b$ (with $c'$ to the right of $c$). By
part~\eqref{it:top_survival} of \Cref{def:top_good}, we have $d(c,c') \le R\eps\, d(a,b)$, and by part~\eqref{it:max_height_top} of \Cref{def:top_good}  combined with part~\eqref{it:bottom_height_bound} of \Cref{def:bottom_good} (see the discussion following \Cref{def:bubble_good}), the heights of all points on the geodesic segment from $c$ to $c'$ (these points all lie on the top geodesic of the bubble) are at most $(R/2) \eps^2 d(a,b)^2$. It follows that the total height of the parallelogram $P_{c,c'}$ above the segment from $a$ to $b$ is at most
\[R\eps^2d(c,c')^2+(R/2)\eps^2 d(a,b)^2 \le (R^2\eps^2+1/2)\cdot R\eps^2 d(a,b)^2.\]
Thus if $\eps$ is small enough (namely once $R^2\eps^2 \le 1/2$), this last quantity is at most $R\eps^2d(a,b)^2$. Since, in addition, the time coordinates of the points of $P_{c,c'}$ lie between those of $c$ and $c'$, and hence between those of $a$ and $b$, and since the points of $P_{c,c'}$ lie weakly above the segment from $c$ to $c'$ and therefore weakly above the segment from $a$ to $b$, we conclude that $P_{c,c'} \subseteq P_{a,b}$.

On the other hand, suppose that $c$ and $c'$ are on the \textit{bottom} geodesic of the same good bubble from the exploration of the segment from $a$ to $b$, between the time at which the bubble starts and its merging time. Thus both $c$ and $c'$ are at height 0, and by the above calculations, the height of $P_{c,c'}$ is at most $R^3 \eps^4 d(a,b)^2$: indeed, $d(c,c')$ is at most the merging time of the bubble, which is at most $R\eps\, d(a,b)$ by part~\eqref{it:top_survival} of \Cref{def:top_good}, so that the height $R\eps^2 d(c,c')^2$ of $P_{c,c'}$ is at most $R^3\eps^4 d(a,b)^2$. In particular, $P_{c,c'}$ is entirely contained in the bottom exploration of this good bubble (by part~\eqref{it:bottom_height_bound} of \Cref{def:bottom_good}, the height of the bottom exploration stays above $R^3\eps^4 d(a,b)^2$ up to time $R\eps\,d(a,b)$, and in particular up to the merging time), and \Cref{lem:conditional_prob_given_good} applies to any event  which is measurable with respect to the restriction of the LHPT to $P_{c,c'}$. This will be crucial later.

Now, recall that for a metric space $X$, the $\ell^2$ distortion of $X$ is the infimum over all $L \ge 1$ for which there exists a function $f: X \to \ell^2$ such that for all $x,y \in X$,
\begin{equation}\label{eq:distortion_def}
d(x,y) \le \|f(x)-f(y)\|_2 \le L d(x,y).
\end{equation}
For technical reasons that we explain below, we will restrict our attention to a smaller class of functions.
\begin{definition}\label{def:L-eps_admissible}
Let $a \in \Z_{\ge 0} \times \Z$, and $b$ be a point on the bottom-most geodesic started from $a$, and let $L \ge 1$ and $\eps>0$. We will say that a function $f: P_{a,b} \to \ell^2$ is $(L,\eps)$-\emph{admissible} if the following conditions hold. See \Cref{fig:exploration around good bubble}.
\begin{enumerate}[(i)]
\item\label{it:adm_lipschitz} For every $x,y \in P_{a,b}$ which are connected to each other by a path entirely contained in $P_{a,b}$,
\[\|f(x)-f(y)\|_2 \le L\, d_{P_{a,b}}(x,y),\]
where $d_{P_{a,b}}(x,y)$ denotes the minimal length of such a path; and for every $x,y \in P_{a,b}$ such that $y$ lies on the bottom-most geodesic started from $x$ and the geodesic segment from $x$ to $y$ is entirely contained in $P_{a,b}$,
\[d(x,y) \le \|f(x)-f(y)\|_2.\]
\item\label{it:adm_bubble_start} For every bubble which is good relative to a starting segment of length $m$ (that is, a bubble arising from the bottom and top explorations started at a point of a geodesic segment of length $m$, which is good in the sense of \Cref{def:bubble_good} with $m$ in place of $n$) and which is entirely contained in $P_{a,b}$, if $a_1$ is the starting point of the bubble and $a_1'$ is the point right above the $B\eps^4m^2$ trees at the start of the bubble (that is, the starting point of the top geodesic of the bubble, which lies at height $B\eps^4m^2$ above $a_1$), then
\[\|f(a_1)-f(a_1')\|_2 \le 2L\eps^{2-u}m.\]
\item\label{it:adm_bubble_merge} For every bubble which is good relative to a starting segment of length $m$ and which is entirely contained in $P_{a,b}$, if $b_1'$ is the merging vertex, and $b_1$ is the point on the starting segment at the same time, then
\[\|f(b_1)-f(b_1')\|_2 \le L\eps^{2-4u}m.\]
\item\label{it:adm_bubble_middle} For every bubble which is good relative to a starting segment of length $m$ and which is entirely contained in $P_{a,b}$, if $c_1$ is the point from part~\eqref{it:middle_graph_distance} of \Cref{def:bubble_good}, and $c_2$ is the point on the starting segment at the same time, then
\[\|f(c_1)-f(c_2)\|_2 \ge \sqrt{D}\eps m.\]
\end{enumerate}
\end{definition}
If $f$ satisfies~\eqref{eq:distortion_def} for all $x,y \in P_{a,b}$, then in particular it is $(L,\eps)$-admissible. Indeed,
\begin{enumerate}[(i)]
\item For the first inequality in~\eqref{it:adm_lipschitz}, if $x$ and $y$ are joined by a path contained in $P_{a,b}$, then $d(x,y) \le d_{P_{a,b}}(x,y)$, so that by~\eqref{eq:distortion_def},
\[\|f(x)-f(y)\|_2 \le L d(x,y) \le L\,d_{P_{a,b}}(x,y).\]
For the second inequality, if $y$ lies on the bottom-most geodesic started from $x$, then~\eqref{eq:distortion_def} directly gives $\|f(x)-f(y)\|_2 \ge d(x,y)$.
\item If $a_1$ and $a_1'$ are as in~\eqref{it:adm_bubble_start}, then by part~\eqref{it:bottom_distance_initial} of \Cref{def:bottom_good} as well as part~\eqref{it:initial_distance} of \Cref{def:top_good} (the two conditions produce paths of length at most $\eps^{2-u}m$ from $a_1$ and from $a_1'$ to the common point separating the two explorations at the start of the bubble, and we concatenate them),
\[\|f(a_1)-f(a_1')\|_2 \le Ld(a_1, a_1') \le 2L\eps^{2-u}m.\]
\item If $b_1$ and $b_1'$ are as in~\eqref{it:adm_bubble_merge}, then by part~\eqref{it:merge_distance} of \Cref{def:bubble_good},
\[\|f(b_1)-f(b_1')\|_2 \le Ld(b_1,b_1') \le L\eps^{2-4u}m.\]
\item If $c_1$ and $c_2$ are as in~\eqref{it:adm_bubble_middle}, then by part~\eqref{it:middle_graph_distance} of \Cref{def:bubble_good},
\[\|f(c_1)-f(c_2)\|_2 \ge d(c_1,c_2) \ge \sqrt{D}\eps m.\]
\end{enumerate}
This means that if we can show no $(L,\eps)$-admissible function exists on $P_{a,b}$, then no function satisfies~\eqref{eq:distortion_def} for all $x,y \in P_{a,b}$, and therefore the $\ell^2$ distortion of $P_{a,b}$ must be at least $L$.

The reason for making this technical definition is that admissibility is determined by the restriction of the LHPT to $P_{a,b}$, but the condition that~\eqref{eq:distortion_def} holds for all $x$ and $y$ in $P_{a,b}$ is not.  Likewise, all of the quantities appearing in conditions~\eqref{it:adm_bubble_start} and~\eqref{it:adm_bubble_merge} concern only heights, merging times and distances \emph{inside} the explorations of the bubble in question, and are therefore determined by the restriction of the LHPT to any region containing that bubble.  The same for the definition of a good bubble: conditions~\eqref{it:bottom_good}--\eqref{it:merge_distance} of \Cref{def:bubble_good} are of the kind just described, and~\eqref{it:middle_graph_distance}, which does refer to the graph distance in the whole of $\CL$, has been formulated as a confinement statement precisely so that it too is determined by the restriction of $\CL$ to a bounded region (namely, by the discussion following \Cref{def:bubble_good}, to the rectangle of times within $\sqrt{D}\eps m$ of the middle of the bubble and heights in $[0,R\eps^2m^2]$ above its starting segment).  That rectangle is contained in $P_{a,b}$ whenever the bubble is: its time range is contained in that of the bubble, since $\sqrt{D} \le 1/2$ and the bubble has merging time at least $\eps m$, and its height is at most $R\eps^2m^2 \le R\eps^2 d(a,b)^2$.  Consequently admissibility, and with it each of the events $\{\chi^{(j)}>t\}$ appearing below, is determined by the restriction of the LHPT to $P_{a,b}$, which is what the argument of \Cref{prop:distortion_rectangle_lhpt} requires.
For a geodesic segment from $a$ to $b$, $L \ge 1$, and $\eps>0$, let
\[\chi_{L,\eps}(a,b)=\sup\left\{\frac{\|f(a)-f(b)\|_2}{d(a,b)} \text{ s.t.\ } f : P_{a,b} \to \ell^2 \text{ is } (L,\eps) \text{-admissible} \right\}.\]
If no such function exists, we instead set $\chi_{L,\eps}(a,b)=-\infty$ (the supremum of the empty set).  This convention matters.  Indeed, below we will consider events of the form $\{\chi_{L,\eps}(a,b)>t\}$ with $t \le 0$, and with the convention $\chi_{L,\eps}(a,b)=-\infty$ such an event asserts the \emph{existence} of an $(L,\eps)$-admissible function whose associated ratio exceeds $t$, which is the intended meaning (with the value $0$ in place of $-\infty$, the event would instead hold vacuously whenever no admissible function exists, and the tail bound~\eqref{eq:tail_bound_chi} below would fail for $k > L/s$). Note of course that if an $(L,\eps)$-admissible function exists, then $1\le \chi_{L,\eps}(a',b') \le L$ for any geodesic segment from $a'$ to $b'$ inside $P_{a,b}$ satisfying $P_{a',b'} \subseteq P_{a,b}$ (the restriction of an admissible function to $P_{a',b'}$ is admissible, as explained in the proof of \Cref{lem:laakso_diamond} below, and the two bounds then follow from condition~\eqref{it:adm_lipschitz} applied to the pair $a',b'$, whose connecting segment is contained in $P_{a',b'}$). Our strategy is then as follows.  Assuming an $(L,\eps)$-admissible function exists, we can use the exploration process to gradually produce a sequence of pairs of points for which $\chi_{L,\eps}$ becomes smaller and smaller, until we find a pair for which $\chi_{L,\eps}$ is less than 1. This is a contradiction, and thus no $(L,\eps)$-admissible function could have existed in the first place. Therefore by the above discussion, the distortion of $P_{a,b}$ must be at least $L$.

\begin{lemma}[Laakso diamond structures]\label{lem:laakso_diamond}
Suppose that we have a good bubble on a starting segment of length $n$, and that the bubble is entirely contained in the parallelogram $P_{a,b}$ defined below.  Let $a$, $a_1$, $a_1'$, $b_1$, $b_1'$, $c_1$, $c_2$, $b$ be the following points:
\begin{itemize}
\item $a=(0,0)$ and $b=(n,0)$.
\item $a_1$ is the starting point of the good bubble on the starting segment.
\item $b_1'$ is the merging vertex, and $b_1$ is the point on the starting segment at the same time.
\item $a_1'$ is the point right above the $B\eps^4n^2$ trees at the start of the bubble, i.e., the starting point of its top geodesic.
\item $c_1$ is the point from part~\eqref{it:middle_graph_distance} of \Cref{def:bubble_good}, and $c_2$ is the point on the starting segment with the same time coordinate.
\end{itemize}
See \Cref{fig:exploration around good bubble}. Let
\begin{align*}
&\chi^{(1)}=\chi_{L,\eps}(a,a_1) \qquad \chi^{(4)}=\chi_{L,\eps}(c_1,b_1')\\
&\chi^{(2)}=\chi_{L,\eps}(b_1,b) \qquad \chi^{(5)}=\chi_{L,\eps}(a_1,c_2)\\
&\chi^{(3)}=\chi_{L,\eps}(a_1',c_1) \qquad \chi^{(6)}=\chi_{L,\eps}(c_2,b_1).
\end{align*}
Then
\[\chi_{L,\eps}(a,b) \le \max(\chi^{(1)}, \dots, \chi^{(6)})+L\beta-\frac{\eta}{L},\]
where
\[\beta= 2\eps^{2-4u} \quad\text{and}\quad \eta=\frac{D\eps}{16R}.\]
\end{lemma}
\begin{figure}
\includegraphics[scale=0.8]{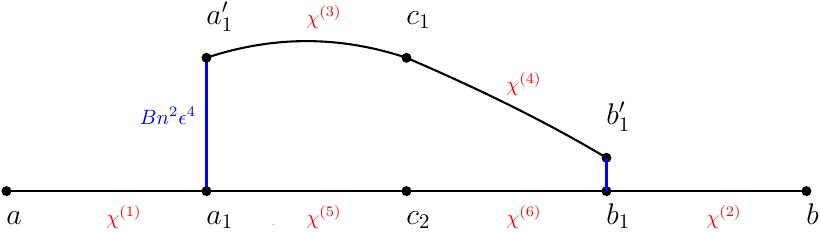}
\caption{
The six segments around a good bubble used in the exploration analysis. Here the bubble starts at $a_1$ and has initial height $B\eps^4n^2$. The top exploration merges with the bottom exploration at $b_1'$. The distance between $c_1$ and $c_2$ is at least $\sqrt{D}\eps n$, by the definition of ``good''.}
\label{fig:exploration around good bubble}
\end{figure}
\begin{proof}
\noindent\steplabel{1}{step:laakso_1}\emph{Step 1. Restrictions of admissible functions.} If no $(L,\eps)$-admissible function $P_{a,b} \to \ell^2$ exists, then $\chi_{L,\eps}(a,b)=-\infty$ and the claim holds trivially.  So suppose that at least one exists. Set $M=\max(\chi^{(1)}, \dots, \chi^{(6)})$ and let $f$ be an $(L,\eps)$-admissible function. By the remarks above, all six of the parallelograms $P_{a,a_1}$, $P_{b_1,b}$, $P_{a_1',c_1}$, $P_{c_1,b_1'}$, $P_{a_1,c_2}$, $P_{c_2,b_1}$ are contained in $P_{a,b}$. Also, each of these segments is either contained in the segment from $a$ to $b$, or is on the top geodesic of a good bubble; in particular, in each case the right endpoint lies on the bottom-most geodesic started from the left endpoint, so that the six parallelograms are of the form to which \Cref{def:L-eps_admissible} applies.  These observations imply that the restriction of $f$ to any of these six parallelograms is also an $(L,\eps)$-admissible function.  Indeed, for condition~\eqref{it:adm_lipschitz}, a path contained in one of the six parallelograms is contained in $P_{a,b}$ (so the minimal path length can only decrease when computed in $P_{a,b}$, which strengthens the required upper bound), and a geodesic segment contained in one of them is contained in $P_{a,b}$; and for conditions~\eqref{it:adm_bubble_start}--\eqref{it:adm_bubble_middle}, a good bubble contained in one of the six parallelograms is contained in $P_{a,b}$, so the corresponding requirements on $f$ are among those already imposed.  In particular, since an admissible function exists on $P_{a,b}$, one exists on each of the six parallelograms, and hence $1 \le \chi^{(j)} \le L$ for all $j$, so that $1 \le M \le L$.

\noindent\steplabel{2}{step:laakso_2}\emph{Step 2. Decomposition into four terms.} By the triangle inequality, We have
\begin{align*}
\|f(a)-f(b)\|_2 &\le \|f(a)-f(a_1)\|_2+\left\|f(a_1)-\frac{f(c_1)+f(c_2)}{2}\right\|_2\\
&\phantom{=}+\left\|\frac{f(c_1)+f(c_2)}{2}-f(b_1)\right\|_2+\|f(b_1)-f(b)\|_2\\
&=:T_1+T_2+T_3+T_4.
\end{align*}
We now estimate each of these four terms below.

\noindent\steplabel{3}{step:laakso_3}\emph{Step 3. The terms $T_1$, $T_2$ and $T_4$.} For the first and fourth terms, we have
\[T_1 \le \chi^{(1)}d(a,a_1) \le Md(a,a_1)\]
and
\[T_4 \le \chi^{(2)}d(b_1,b) \le Md(b_1,b).\]
For the second term, we write 
\[f(a_1)-\frac{f(c_1)+f(c_2)}{2}=\frac12\big[f(a_1)-f(a_1')\big]+\frac12\big[f(a_1')-f(c_1)\big]+\frac12\big[f(a_1)-f(c_2)\big]\] 
and apply the triangle inequality, which gives
\begin{align*}
T_2 &\le \frac{\|f(a_1)-f(a_1')\|_2}{2}+\frac{\|f(a_1')-f(c_1)\|_2}{2}+\frac{\|f(a_1)-f(c_2)\|_2}{2}\\
&\le \frac{\|f(a_1)-f(a_1')\|_2}{2}+\frac{M}{2}d(a_1',c_1)+\frac{M}{2}d(a_1,c_2).
\end{align*}
Since $f$ is $(L,\eps)$-admissible (namely, by condition~\eqref{it:adm_bubble_start} of \Cref{def:L-eps_admissible}, applied to our good bubble, which is contained in $P_{a,b}$ by hypothesis),
\[\frac{\|f(a_1)-f(a_1')\|_2}{2} \le L\eps^{2-u}n\le \frac{L\beta}{2} d(a,b),\]
where the second inequality holds because $\eps^{2-u} \le \eps^{2-4u}$ and $d(a,b)=n$.
Furthermore, $d(a_1',c_1)=d(a_1,c_2)$. This is because $a_1'$ and $a_1$ have the same time coordinate, $c_1$ and $c_2$ have the same time coordinate, $c_1$ is on the bottom-most geodesic started from $a_1'$ and $c_2$ is on the bottom-most geodesic started from $a_1$, so that both distances are equal to the common difference of the time coordinates. Altogether,
\[T_2\le \frac{L\beta}{2} d(a,b)+Md(a_1,c_2).\]

\noindent\steplabel{4}{step:laakso_4}\emph{Step 4. The term $T_3$ and the parallelogram law.} Next, for $T_3$, we have by the parallelogram law,
\begin{equation}\label{eq:parallelogram_law}
\left\|f(b_1)-\frac{f(c_1)+f(c_2)}{2}\right\|_2^2=\frac{\|f(b_1)-f(c_1)\|_2^2}{2}+\frac{\|f(b_1)-f(c_2)\|_2^2}{2}-\frac{\|f(c_1)-f(c_2)\|_2^2}{4}.
\end{equation}
Now, since $f$ is $(L,\eps)$-admissible (the first term below is bounded using the triangle inequality together with condition~\eqref{it:adm_bubble_merge} of \Cref{def:L-eps_admissible}, and the second using the definition of $\chi^{(4)}$ applied to the restriction of $f$ to $P_{c_1,b_1'}$), we have
\begin{align*}
\frac{\|f(b_1)-f(c_1)\|_2^2}{2} &\le \frac{1}{2} \left( \|f(b_1)-f(b_1')\|_2+\|f(b_1')-f(c_1)\|_2 \right)^2\\
&\le \frac{1}{2}\left(L\eps^{2-4u}n+Md(b_1',c_1)\right)^2\\
&= \frac{1}{2}\left(L\eps^{2-4u}n+Md(b_1,c_2)\right)^2.
\end{align*}
In the last line, we used the fact that $b_1$ and $b_1'$ have the same time coordinates, and likewise $c_1$ and $c_2$ have the same time coordinates. So because $b_1'$ is on the bottom-most geodesic started from $c_1$, and $b_1$ is on the one started from $c_2$, we have $d(b_1',c_1)=d(b_1,c_2)$.

We also have
\[\frac{\|f(b_1)-f(c_2)\|_2^2}{2}  \le \frac{(M d(b_1,c_2))^2}{2} \le \frac{1}{2}\left(L\eps^{2-4u}n+Md(b_1,c_2)\right)^2,\]
and by~\eqref{it:adm_bubble_middle} of \Cref{def:L-eps_admissible},
\[\frac{\|f(c_1)-f(c_2)\|_2^2}{4} \ge \frac{D\eps^2n^2}{4}.\]
Therefore \eqref{eq:parallelogram_law} reduces to
\[T_3 \le \sqrt{(L\eps^{2-4u}n+Md(b_1,c_2))^2-\frac{D\eps^2n^2}{4}}.\]
Note that the quantity under the square root is automatically non-negative, since by the above it is an upper bound for $T_3^2$.
The inequality
\[\sqrt{\lambda-\mu} \le \sqrt{\lambda}-\frac{\mu}{2\sqrt{\lambda}}\]
holds for any $\lambda>0$ and $\mu\le\lambda$, by concavity of the function $\mu \mapsto \sqrt{\lambda-\mu}$. Applying it with $\lambda=(L\eps^{2-4u}n+Md(b_1,c_2))^2$ and $\mu=D\eps^2n^2/4$, we get
\[T_3 \le L\eps^{2-4u}n+Md(b_1,c_2)-\frac{D\eps^2n^2}{8(L\eps^{2-4u}n+Md(b_1,c_2))}.\]
The distance between $b_1$ and $c_2$ is at most $R\eps n$ by part~\eqref{it:top_survival} of \Cref{def:top_good}. We have $M \le L$, and $\eps^{2-4u} n \le R \eps n$ for $\eps$ sufficiently small. These observations imply that
\[\frac{D\eps^2n^2}{8(L\eps^{2-4u}n+Md(b_1,c_2))} \ge \frac{D\eps^2n^2}{8(L\eps^{2-4u}n+LR\eps n)} \ge \frac{D\eps n}{16LR}.\]
So overall,
\[T_3 \le L\eps^{2-4u}n+Md(b_1,c_2)-\frac{D\eps n}{16LR}= \frac{L\beta}{2}d(a,b)+Md(b_1,c_2)-\frac{\eta}{L}d(a,b).\]

\noindent\steplabel{5}{step:laakso_5}\emph{Step 5. Conclusion.} Putting all these estimates together, we thus obtain
\[\|f(a)-f(b)\|_2 \le M( d(a,a_1)+d(a_1,c_2)+d(c_2,b_1)+d(b_1,b))+L\beta d(a,b)-\frac{\eta}{L}d(a,b).\]
The path $a \to a_1 \to c_2 \to b_1 \to b$ is a geodesic from $a$ to $b$, so the term in parentheses is exactly $d(a,b)$. Hence, we finally obtain
\[\|f(a)-f(b)\|_2 \le \left(M+L\beta-\frac{\eta}{L}\right)d(a,b).\]
Since $f$ was an arbitrary $(L,\eps)$-admissible function, the result follows.
\end{proof}

We can now obtain a lower bound for the $\ell^2$ distortion of a rectangle in $\Z^2$, when equipped with the graph distance in the LHPT.
\begin{proposition}\label{prop:distortion_rectangle_lhpt}
Let $Q_n=[0,n]\times [0,n^2]$, equipped with the LHPT graph metric. Then for any $0<\delta<1/4$,
\[\p(\Dis(Q_n) \ge (\log n)^{1/4-\delta}) \to 1\]
as $n \to \infty$.
\end{proposition}
\begin{remark}
What the proof produces, and what we shall use below, is the same statement for the parallelogram $P_{(0,0),(n,0)}$, which is contained in $Q_n$; and, by the Markov property of the LHPT, the statement holds just as well for the parallelogram $P_{(t,0),(t+n,0)}$ above any portion of length $n$ of the bottom-most geodesic started from $(0,0)$, uniformly in $t \ge 0$, since the part of the LHPT lying to the right of the line $\{t\}\times\Z$ is again an LHPT with $(t,0)$ playing the role of the origin.
\end{remark}
\begin{proof}
\noindent\steplabel{1}{step:distrect_1}\emph{Step 1. Setup and the choice of $\rho$.} For $\eps>0$ small, we set $n=e^{\eps^{-2}}$. We will then be letting $\eps \to 0$ throughout. We will of course use the exploration process, with this choice of $\eps$. Given $\delta \in (0, 1/4)$, we first fix the parameter $u$ from the exploration process such that $u < \delta$ (and $u<1/10$, as required in \Cref{fn:parameters}). (The remaining parameters $R$, $D$ and $B$ are chosen as in \Cref{fn:parameters}, with $B$ subject to the further constraint given below.)

Fix $L \ge 1$. For $n \in \N$, and a non-negative integer $k$, let $A_{n,k}$ be the event
\[A_{n,k}=\{\chi_{L,\eps}((0,0),(n,0))>L-ks\}\]
where 
\[s=\frac{\eta}{L}-L\beta=\frac{D\eps}{16RL}-2\eps^{2-4u} L\]
(here $\beta$ and $\eta$ are as in \Cref{lem:laakso_diamond}). We claim that there is a constant $0<\rho<1$ (which depends on the parameters in \Cref{def:bubble_good}, but not on $\eps$ or $L$) such that for every $L\ge 1$, and for every $n$ and $k$, we have
\begin{equation}\label{eq:tail_bound_chi}
\p(A_{n,k}) \le \rho^{\frac{1}{4}\left \lfloor \frac{\log n}{\log(2/\eps)} \right \rfloor-k}.
\end{equation}
More specifically, we can pick $\rho$ as follows. Let $c_{\ref{prop:no_good_bubbles}}$ and $c_{\ref{lem:conditional_prob_given_failure}}$ be the constants appearing in \Cref{prop:no_good_bubbles} and \Cref{lem:conditional_prob_given_failure}, respectively; recall that these constants do not depend on the parameter $B$. Then choose $0<\rho<1$ and $B$ large enough so that for all $\eps$ sufficiently small,
\[ 6B \rho^{\frac{3}{4}}+12 c_{\ref{prop:no_good_bubbles}}c_{\ref{lem:conditional_prob_given_failure}}\left(\frac{1}{\sqrt{B}}+\eps\right) \rho^{-1/4} \le 1.\]
For example, we can take $\rho=B^{-5/3}$ and $B$ very large. the two terms on the left-hand side are then $6B^{-1/4}$ and $12c_{\ref{prop:no_good_bubbles}}c_{\ref{lem:conditional_prob_given_failure}}\big(B^{-1/12}+\eps B^{5/12}\big)$, and both can be made smaller than $1/2$ by first taking $B$ large and then $\eps$ small.

\noindent\steplabel{2}{step:distrect_2}\emph{Step 2. The induction: base cases and decomposition.} Let us now prove~\eqref{eq:tail_bound_chi}. We argue by induction on $n$. When $k=0$, this is trivial since the left-hand side of~\eqref{eq:tail_bound_chi} is 0; we always have $\chi_{L,\eps}(a,b) \le L$ by definition of $\chi_{L,\eps}$. If $k \ge 1$ and $n \le 16\eps^{-4}$, then this is also trivial because in that case the right-hand side of~\eqref{eq:tail_bound_chi} is at least 1.  (Indeed, $16\eps^{-4}<32\eps^{-5}=(2/\eps)^5$, so that $n<(2/\eps)^5$ and hence $\log n/\log(2/\eps)<5$; thus $g(n) \le 1 \le k$, where $g$ is defined below, and $\rho^{g(n)-k} \ge 1$ because $\rho<1$).

So we may assume from now on that $k \ge 1$ and $n>16\eps^{-4}$.  Let us stress that the induction applies the estimates of \Cref{sec:exploration_process} at the scale $n$, for every such $n$, with $\eps$ held fixed; this is legitimate precisely because the lower bound on $n$ required there is polynomial in $\eps^{-1}$ ($n \ge \eps^{-4}$ suffices, as explained in \Cref{fn:parameters}) and $n>16\eps^{-4}$ here.  It is only the largest scale, at which the proposition is finally applied, that is taken superpolynomially large in $\eps^{-1}$.

Now assume that for all $n'<n$, we have proven~\eqref{eq:tail_bound_chi} holds for all $k \ge 0$ (with $n'$ instead of $n$). For each $x$ and $\ell$, let $G_{x,\ell}$ be the event that the bubble started at $x$ is good and has merging time $\ell$, so that the point $b_1$ of \Cref{lem:laakso_diamond} lies at time $x+\ell$. Also let $F$ be the event that there is a failure (recall this means that either there are no good bubbles, or that an independent $Z \sim \Ber(c_{\ref{prop:no_good_bubbles}}/\sqrt{B})$ is $1$). Then
\[\p(A_{n,k})=\p\!\left( A_{n,k} \cap \left( \bigcup_{x,\ell} G_{x,\ell} \right) \cap \{Z=0\}\right)+\p(A_{n,k} \cap F)=: \Pi_1+\Pi_2. \]

\noindent\steplabel{3}{step:distrect_3}\emph{Step 3. Bounding $\Pi_1$: the case of a good bubble.} We have that
\[\Pi_1 \le \p\!\left( A_{n,k} \cap \left( \bigcup_{x,\ell} G_{x,\ell} \right)\right)\le \sum_{x, \ell} \p(A_{n,k} \cap G_{x,\ell})=\sum_{x, \ell} \p(A_{n,k} \giv G_{x,\ell}) \p(G_{x,\ell}),\]
where the sum runs over the pairs $(x,\ell)$ with $\p(G_{x,\ell})>0$.
When we have a good bubble at $x$, we can repeat the exploration on the six segments around this good bubble shown in \Cref{fig:exploration around good bubble}. Then by \Cref{lem:laakso_diamond} (whose containment hypothesis is satisfied, since $x$ lies in the middle third of $[0,n]$),
\[\p(A_{n,k} \giv G_{x,\ell}) \le \p(\max(\chi^{(1)}, \dots, \chi^{(6)})-s>L-ks \giv G_{x,\ell}) \le \sum_{j=1}^6 \p(\chi^{(j)}>L-(k-1)s \giv G_{x,\ell}).\]
Now given $G_{x,\ell}$,
\begin{enumerate}[(i)]
\item\label{it:mkv_left} The law of the triangulation to the left of the bubble is an LHPT on top of the segment $[0,x]$, and is independent of everything after $x$; this is the region carrying $\chi^{(1)}$.
\item\label{it:mkv_right} The triangulation to the right of the bubble is an LHPT on top of the segment $[x+\ell,n]$, and is independent of everything before $x+\ell$; this is the region carrying $\chi^{(2)}$.
\item\label{it:mkv_top_first} The triangulation on top of the bubble over the segment $[x,x+\eps n/2]$ is an LHPT, and is independent of everything below the bubble, to the left of $x$ and to the right of $x+\eps n/2$; this is the region carrying $\chi^{(3)}$.
\item\label{it:mkv_top_second} The triangulation on top of the bubble over the segment $[x+\eps n/2, x+\ell]$ is an LHPT and is independent of everything below the bubble, to the left of $x+\eps n/2$ and to the right of $x+\ell$; this is the region carrying $\chi^{(4)}$.
\end{enumerate}
Thus, for $j=1,2,3,4$ and any $t$, the event $\{\chi^{(j)}>t\}$ is independent of the top exploration of the bubble started at $x$.  This holds for $j=1$ and $j=2$ by~\eqref{it:mkv_left} and~\eqref{it:mkv_right} since the parallelograms involved are to the left of time $x$ and to the right of time $x+\ell$ respectively, whereas the top exploration lies between these two times.  This holds for $j=3$ and $j=4$ by~\eqref{it:mkv_top_first} and~\eqref{it:mkv_top_second}, i.e., by the spatial Markov property above the top geodesic.  For $j=2$ it is \emph{not}, however, independent of $G_{x,\ell}$.  This is because being good also constrains the bottom exploration \emph{after} the merging time, both through~\eqref{it:bottom_height_bound} of \Cref{def:bottom_good}, whose height bounds are used up to time $R\eps n \ge \ell$, and through~\eqref{it:merge_distance} of \Cref{def:bubble_good}, which concerns a path lying to the right of time $\ell$, and the region which these two conditions constrain meets the parallelogram $P_{b_1,b}$ of the second segment.  We therefore use, for all four of these segments, the product bound which the last two segments require in any case.  Write $G^1_{x,\ell} \supseteq G_{x,\ell}$ for the event of \Cref{lem:conditional_prob_given_good} for the bubble started at $x$, namely that its top exploration is good and that its merging time equals $\ell$; this event is measurable with respect to the top exploration of that bubble, and hence, on the event that the merging time equals $\ell$, with respect to the restriction of the LHPT to $[x,x+\ell] \times \Z$.  Consequently
\[\p\!\big(\{\chi^{(j)}>L-(k-1)s\} \cap G_{x,\ell}\big) \le \p(\chi^{(j)}>L-(k-1)s)\,\p(G^1_{x,\ell}) \quad\text{for $j=1,\dots,4$,}\]
since $G_{x,\ell} \subseteq G^1_{x,\ell}$ and the two events on the right-hand side are independent.
For $\chi^{(5)}$ and $\chi^{(6)}$, we instead use \Cref{lem:conditional_prob_given_good}. For any $t$, the event $\{\chi^{(5)}>t\}$ only depends on the parallelogram above the segment from $a_1$ to $c_2$, and likewise for $\chi^{(6)}$, and by the discussion at the start of this section, those parallelograms have height at most $R^3 \eps^4 n^2$, so that the $\ell$-resolved bound~\eqref{eq:good_ell_resolved} applies, and we have
\[\p\!\big(\{\chi^{(j)}>L-(k-1)s\} \cap G_{x,\ell}\big)\le \p(\chi^{(j)}>L-(k-1)s)\,\p(G^1_{x,\ell}) \quad\text{for $j=5,6$}\]
as well.
Writing
\[g(n)=\frac{1}{4}\left \lfloor \frac{\log n}{\log(2/\eps)} \right \rfloor,\]
it therefore follows from the induction hypothesis (which applies to each of the four scales occurring below, namely $x$, $n-(x+\ell)$, $\eps n/2$ and $\ell-\eps n/2$, since each of them is smaller than $n$ (we have $x \le 2n/3$, $n-(x+\ell) \le 2n/3$, and $\ell \le R\eps n<n$ for $\eps$ small) and at least $\eps n/2>0$ (we have $x \ge n/3$, $n-(x+\ell) \ge n/3-R\eps n$ and $\ell \ge \eps n$)) that
\begin{align}
\Pi_1 &\le \sum_{x, \ell} \left( \p(A_{x,k-1})+\p(A_{n-(x+\ell),k-1})+2\p(A_{\eps n/2,k-1})+2\p(A_{\ell-\eps n/2,k-1})\right)\p(G^1_{x,\ell}) \notag\\
&\le \sum_{x, \ell} \left( \rho^{g(x)-(k-1)}+\rho^{g(n-(x+\ell))-(k-1)}+2\rho^{g(\eps n/2)-(k-1)}+2\rho^{g(\ell-\eps n/2)-(k-1)} \right)\p(G^1_{x,\ell})  \label{eq:P1_recursion}\\
\intertext{Using that $g$ is non-decreasing, that good bubbles always have merging time at least $\eps n$, and that we only look for good bubbles in the middle third of the starting segment (so $n/3 \le x \le 2n/3$) (whence $x \ge n/3 \ge \eps n/2$, $n-(x+\ell) \ge n/3-R\eps n \ge \eps n/2$ and $\ell-\eps n/2 \ge \eps n/2$ for $\eps$ small, so that all four exponents are at least $g(\eps n/2)$), we see that~\eqref{eq:P1_recursion} is bounded from above by}
&\phantom{\le}\qquad\qquad\qquad\qquad\qquad\qquad 6\rho^{g(\eps n/2)-(k-1)} \sum_{x, \ell} \p(G^1_{x, \ell}). \notag
\end{align}
Note also that by~\eqref{eq:good_top_decomposed} we have $\sum_{\ell}\p(G^1_{x,\ell}) \le B\eps^2$ for each $x$, and since we test for good bubbles at at most $\eps^{-2}$ starting times, this gives $\sum_{x,\ell}\p(G^1_{x,\ell}) \le B$.  (The same count applies to $\sum_{x,\ell}\p(G_{x,\ell})$, which is the expected number of good bubbles, since by \Cref{prop:good_bubble_probability} the probability that a bubble is good is at most $B\eps^2$.)  Hence $\Pi_1 \le 6B\rho^{g(\eps n/2)-(k-1)}$.

\noindent\steplabel{4}{step:distrect_4}\emph{Step 4. Bounding $\Pi_2$: the case of a failure.} If we have a failure, we can still break the starting segment into its six sub-segments and do the exploration on each of them. Let $v_1, \dots, v_5$ be the endpoints of the six sub-segments, $v_0=(0,0)$, and $v_6=(n,0)$, and set in this case
\[\chi^{(j)}=\chi_{L,\eps}(v_{j-1},v_j).\]
Since the $v_j$'s all lie on a geodesic, we have that for any $(L,\eps)$-admissible function $f$ on $P_{v_0,v_6}$ (whose restriction to each $P_{v_{j-1},v_j}$ is again admissible, since these parallelograms are contained in $P_{v_0,v_6}$, their segments being sub-segments of the segment from $v_0$ to $v_6$),
\begin{align*}
\|f(v_0)-f(v_6)\|_2 &\le \|f(v_0)-f(v_1)\|_2+\dots+\|f(v_5)-f(v_6)\|_2\\
&\le \chi^{(1)}d(v_0,v_1)+\dots+\chi^{(6)}d(v_5,v_6)\\
&\le \max(\chi^{(1)},\dots,\chi^{(6)})( d(v_0,v_1)+\dots+d(v_5,v_6) )\\
&=\max(\chi^{(1)},\dots,\chi^{(6)}) d((0,0),(n,0)),
\end{align*}
so that
\[\chi_{L,\eps}((0,0),(n,0)) \le \max(\chi^{(1)},\dots,\chi^{(6)}).\]
Then
\begin{align*}
\Pi_2&\le \p(\max(\chi^{(1)},\dots,\chi^{(6)})>L-ks \giv F)\p(F)\\
&\le \left(\sum_{j=1}^6\p(\chi^{(j)}>L-ks \giv F)\right) \cdot 2c_{\ref{prop:no_good_bubbles}}\left(\frac{1}{\sqrt{B}}+\eps\right) \quad\text{(by \Cref{prop:no_good_bubbles})}\\
&\le 12 c_{\ref{prop:no_good_bubbles}}c_{\ref{lem:conditional_prob_given_failure}}\left(\frac{1}{\sqrt{B}}+\eps\right)\p(A_{n/6,k}) \quad\text{(by \Cref{lem:conditional_prob_given_failure})}\\
&\le 12c_{\ref{prop:no_good_bubbles}}c_{\ref{lem:conditional_prob_given_failure}}\left(\frac{1}{\sqrt{B}}+\eps\right) \rho^{g(n/6)-k} \quad\text{(the induction hypothesis)}\\
&\le 12 c_{\ref{prop:no_good_bubbles}}c_{\ref{lem:conditional_prob_given_failure}}\left(\frac{1}{\sqrt{B}}+\eps\right) \rho^{g(\eps n/2)-k}.
\end{align*}
Here the bound $\p(F) \le 2c_{\ref{prop:no_good_bubbles}}(1/\sqrt{B}+\eps)$ used in the second line follows from \Cref{prop:no_good_bubbles} together with the definition of a failure: 
\[\p(F) \le \p(E^c)+\p(Z=1) \le c_{\ref{prop:no_good_bubbles}}\left(\eps+\frac{1}{\sqrt{B}}\right)+\frac{c_{\ref{prop:no_good_bubbles}}}{\sqrt{B}} \le 2c_{\ref{prop:no_good_bubbles}}\left(\frac{1}{\sqrt{B}}+\eps\right).\]
In the third line, \Cref{lem:conditional_prob_given_failure} is applied separately to each of the six events $\{\chi^{(j)}>L-ks\}$, which is an event for the LHPT restricted to the $j$-th sub-segment of the starting segment, and we then use that $\p(\chi^{(j)}>L-ks)=\p(A_{n/6,k})$ for every $j$, since the part of the LHPT on top of each of the six sub-segments is again an LHPT.  In the last line we also used that $g$ is non-decreasing and $\eps n/2 \le n/6$.

\noindent\steplabel{5}{step:distrect_5}\emph{Step 5. Completing the induction.} Putting all this together, we thus have, using the fact that $g(\eps n/2)-g(n)=-1/4$ and our choices of $\rho$ and $B$,
\begin{align*}
\p(A_{n,k}) &\le 6B \rho^{g(\eps n/2)-(k-1)}+ 12 c_{\ref{prop:no_good_bubbles}}c_{\ref{lem:conditional_prob_given_failure}}\left(\frac{1}{\sqrt{B}}+\eps\right) \rho^{g(\eps n/2)-k}\\
&=\left( 6B \rho^{\frac{3}{4}}+12 c_{\ref{prop:no_good_bubbles}}c_{\ref{lem:conditional_prob_given_failure}}\left(\frac{1}{\sqrt{B}}+\eps\right)\rho^{-\frac{1}{4}}\right) \rho^{g(n)-k} \le \rho^{g(n)-k}.
\end{align*}
This completes the proof of~\eqref{eq:tail_bound_chi}.

\noindent\steplabel{6}{step:distrect_6}\emph{Step 6. Conclusion.} Now pick $L=\eps^{2\delta-1/2}$. Then
\[s=\frac{D\eps}{16 R L}-2L\eps^{2-4u}=\frac{D}{16 R}\eps^{\frac{3}{2}-2\delta}-2\eps^{\frac{3}{2}+2\delta-4u}.\]
So for all $\eps$ sufficiently small,
\[\frac{1}{2\Lambda}\eps^{\frac{3}{2}-2\delta} \le s \le \frac{1}{\Lambda}\eps^{\frac{3}{2}-2\delta},\]
where we have set $\Lambda=16R/D$ for convenience.  It is here that the choice $u<\delta$ made in \Cref{step:distrect_1} is used: the second exponent $\frac32+2\delta-4u$ exceeds the first exponent $\frac32-2\delta$ by $4(\delta-u)>0$, so the second term in the expression for $s$ is negligible with respect to the first as $\eps \to 0$. In particular, $s$ is positive, and so for this choice of $L$, we have with $k=\lceil L/s \rceil$ that $L-ks \le 0$. So on the event that $\chi_{L,\eps}((0,0),(n,0)) \le L-ks$, there is no $(L,\eps)$-admissible function from $P_{(0,0),(n,0)}$ into $\ell^2$ (if one existed, then, as noted after the definition of $\chi_{L,\eps}$, we would have $\chi_{L,\eps}((0,0),(n,0)) \ge 1 > 0 \ge L-ks$), and hence, by the discussion following \Cref{def:L-eps_admissible}, the $\ell^2$ distortion of $P_{(0,0),(n,0)}$ is at least $L$. Since the parallelogram $P_{(0,0),(n,0)}$ is contained in $Q_n$ for $\eps$ small (its height is $R\eps^2n^2 \le n^2$), and since the distortion of a subset is at most that of the whole set (a function on $Q_n$ satisfying~\eqref{eq:distortion_def} restricts to one on $P_{(0,0),(n,0)}$), it follows that the distortion of $Q_n$ is at least $L$:
\[\p(\Dis(Q_n) \ge L) \ge \p(A_{n,k}^c) \ge 1-\rho^{g(n)-k} \ge 1-\rho^{g(n)-\frac{L}{s}-1}.\]
But now
\[\Lambda\eps^{4\delta-2} \le \frac{L}{s} \le 2\Lambda\eps^{4\delta-2},\]
so recalling that we had set $n=e^{\eps^{-2}}$,
\[g(n)-\frac{L}{s}-1 \ge \frac{\eps^{-2}}{4\log(2/\eps)}-2\Lambda\eps^{4\delta-2}-2 \ge \eps^{3\delta-2}=(\log n)^{1-\frac{3\delta}{2}}\]
for $\eps$ sufficiently small. Also,
\[L=\eps^{2\delta-\frac{1}{2}}=(\eps^{-2})^{\frac{1}{4}-\delta}=(\log n)^{\frac{1}{4}-\delta}.\]
Hence
\begin{align*}
\p(\Dis(Q_n) \ge (\log n)^{\frac{1}{4}-\delta})&=\p(\Dis(Q_n) \ge L) \ge 1-\rho^{g(n)-\frac{L}{s}-1}\\
&\ge 1-\rho^{(\log n)^{1-\frac{3\delta}{2}}} \to 1.
\end{align*}
Note that what the computation bounds is the distortion of the parallelogram $P_{(0,0),(n,0)}$, that of $Q_n$ being deduced from it; this is the form in which the proposition will be applied.
This concludes the proof.
\end{proof}

The computation at the end of the above proof was quite technical, so it is worth giving a brief overview to explain where the exponent $1/4$ in our distortion lower bound comes from. We showed that there is a $0<\rho<1$ such that for any $k$,
\[\p(\chi_{L,\eps}((0,0),(n,0))>L-ks) \le \rho^{g(n)-k}.\]
Here $s$ was given by
\[s=\frac{\eta}{L}-L\beta.\]
Now clearly we need $s$ to be positive, since we want to show $\chi_{L,\eps}$ is less than 1 with high probability to conclude the non-existence of $(L,\eps)$-admissible functions. This means that we must have $L<(\eta/\beta)^{1/2}$. We want $L$ as large as possible to get the best distortion lower bound, and as $\eps \to 0$, we had that $\eta \approx \eps$ and $\beta \approx \eps^2$, and so we should take $L \approx \eps^{-1/2}$. Note that the way $\beta$ scales as $\eps \to 0$ compared to $\eta$ is precisely due to the fact that boundary length (i.e., vertical height) in the LHPT scales like the square of the graph distance.

So now this means that $s \approx \eps^{3/2}=1/L^3$, and by taking $k=L/s$, we therefore get
\[\p(\chi_{L,\eps}((0,0),(n,0))>0) \le \rho^{\frac{1}{4}\lfloor \log n/ \log(2/\eps) \rfloor-L^{4}}.\]
We need $n$ growing faster than any power of $1/\eps$ to be able to iterate the exploration more often the smaller $\eps$ gets, and so the $\log(2/\eps)$ term is negligible. To get a positive exponent on the right, we should thus have $L$ of order $(\log n)^{1/4}$.

We are now going to deduce from \Cref{prop:distortion_rectangle_lhpt} a similar result which holds for a slice of the LHPT. We write $\CS^{(k)}$ for the slice of initial height $k$ of the LHPT, that is, for the bubble of $\CL$ started at $(0,0)$ with initial height $k$.  We will also let $Z_t$ be the height process associated with the slice.

\begin{proposition}\label{prop:slice_geometry}
Let $0<\delta<1/4$ and let $\CS=\CS^{(n^2)}$ be the slice of the LHPT of initial height $n^2$.  Let $P_n$ be the parallelogram above the portion of the bottom geodesic of $\CS$ running between the times $t_n=\tfrac{1}{2}n^{3/4}$ and $t_n+n^{1/4}$, and let $E_n$  be the event that
\begin{enumerate}[(i)]
\item\label{it:slice_deep} $Z_t \ge n^{5/4}$ for every $0 \le t \le n^{3/4}$, so that in particular $P_n$, whose height is at most $R\sqrt{n}$, is contained in $\CS$;
\item\label{it:slice_diameter} any two vertices of $P_n$ are joined by a path of $\CS$ of length at most $4\sqrt{n}$;
\item\label{it:slice_confined} every path of $\CS$ of length at most $4\sqrt{n}$ started at a vertex of $P_n$ avoids the top geodesic of $\CS$; and
\item\label{it:slice_distortion} $\Dis(P_n) \ge (\log n)^{(1/4)-\delta}$, the distortion being computed with the distances of $\CS$.
\end{enumerate}
Then $\p(E_n) \to 1$ as $n \to \infty$.
\end{proposition}

The proof makes use of the following deterministic lemma which will also be used to transfer the distortion lower bound from the LHPT to the sphere.

\begin{lemma}\label{lem:slice_distances}
Let $\CS$ be a slice (with corresponding height process $Z_t$), let $D<t_0$ and $m$ be positive integers and let $P$ be a set of vertices of $\CS$ whose time coordinates lie between $t_0$ and $t_0+m$. Assume that the following conditions hold:
\begin{enumerate}[(i)]
\item\label{it:inh_alive} $Z_t>0$ for every $t \le t_0+m+D$;
\item\label{it:inh_diameter} any two vertices of $P$ are joined by a path of $\CS$ of length at most $D$; and
\item\label{it:inh_confined} every path of $\CS$ of length at most $D$ started at a vertex of $P$ avoids the top geodesic of $\CS$.
\end{enumerate}
Let $\mathfrak{M}$ be a triangulation which contains $\CS$ as a sub-triangulation. Then $d_{\mathrm{gr}}^{\mathfrak{M}}=d_{\mathrm{gr}}^{\CS}$ on $P$; in particular $\Dis(\mathfrak{M}) \ge \Dis(P)$, the distortion of $P$ computed with the distances of $\CS$.
\end{lemma}

\begin{proof}
Every path of $\CS$ is a path of $\mathfrak{M}$, so $d_{\mathrm{gr}}^{\mathfrak{M}} \le d_{\mathrm{gr}}^{\CS}$ on $P$.  Conversely, let $x,y \in P$ and let $\pi$ be a path of $\mathfrak{M}$ from $x$ to $y$, which by~\eqref{it:inh_diameter} we may assume to be of length at most $D$.  A vertex of $\CS$ which is not on its boundary has all of its neighbors in $\mathfrak{M}$ among the vertices of $\CS$, the faces of $\CS$ incident to it being faces of $\mathfrak{M}$ which already surround it; the path $\pi$ can therefore leave $\CS$ only from a vertex of its boundary. The boundary of $\CS$ consists of three parts: the starting vertical segment and the top and bottom geodesics run until they coalesce.  Every vertex $u$ of that boundary visited by $\pi$ carries a time coordinate which is within $D$ of the time coordinate of $x$, so that its time lies in $[t_0-D,t_0+m+D]$; by $D<t_0$ and~\eqref{it:inh_alive} it does not lie on the initial segment. By ~\eqref{it:inh_confined}, it cannot lie on the top geodesic. Hence, it must lie on the bottom geodesic.

Finally, if $\pi$ exits from the bottom geodesic at $u$ and then re-enters from $u'$, then we can modify the portion of $\pi$ from $u$ to $u'$ by simply following the bottom geodesic between them. Since it's a geodesic from $u$ to $u'$, this modification can only decrease the length, and the portion from $u$ to $u'$ is now in $\CS$. Doing this to all the exits and re-entries of $\pi$ from the bottom geodesic gives that the length of $\pi$ is at least $d_{\mathrm{gr}}^{\CS}(x,y)$.

Thus $d_{\mathrm{gr}}^{\CS}(x,y) \le d_{\mathrm{gr}}^{\mathfrak{M}}(x,y)$.  The metric which $\mathfrak{M}$ induces on $P$ is therefore that of $\CS$, and the $\ell^2$ distortion of a metric space is at least that of any of its subspaces.
\end{proof}

\begin{proof}[Proof of \Cref{prop:slice_geometry}]
The height of the slice is a Galton--Watson process with offspring distribution $\theta$ started from $n^2$, so by \Cref{lem:tail_bound_minimum}, applied with $x=n^2$, with $n^{3/4}$ in place of $n$ and with $n^{5/4}$ in the role of the parameter called $a$ there, the probability that~\eqref{it:slice_deep} fails is at most
\[\frac{n^{5/4}}{(n^{3/4}+1)^2}+\exp\left(-\frac{n^2}{(2n^{3/4}+1)^2}\right) \le \frac{1}{n^{1/4}}+e^{-\sqrt{n}/9}\]
for $n$ large, where we used $(1-y)^\alpha \ge 1-\alpha y$ for the first term and $1-y \le e^{-y}$ for the second.  Since the height of $P_n$ is at most $R\sqrt{n}<n^{5/4}$ for $n$ large, on the event~\eqref{it:slice_deep} the parallelogram $P_n$ is contained in $\CS$, and so are the trees rooted between $(t,0)$ and $(t,R\sqrt{n})$ for any $t \le n^{3/4}$.

Let $\zeta_t$ denote the extinction time of the forest formed by those trees.  The bottom-most geodesic started from a vertex $(t,h)$ of $P_n$ has coalesced with the one started from $(t,0)$, which is the bottom geodesic of the slice, by the time $t+\zeta_t$; the corresponding portion of it is a path of $\CS$ of length at most $\zeta_t$ joining $(t,h)$ to that geodesic.  By~\eqref{eq:bubble_lifetime},
\[\p(\zeta_t>\sqrt{n})=1-\left(1-\frac{1}{(\sqrt{n}+1)^2}\right)^{R\sqrt{n}} \le \frac{R\sqrt{n}}{(\sqrt{n}+1)^2} \le \frac{R}{\sqrt{n}},\]
so that, by a union bound over the $n^{1/4}+1$ times $t \in [t_n,t_n+n^{1/4}]$, with probability at least $1-2R n^{-1/4}$ we have $\zeta_t \le \sqrt{n}$ for all of them.  On that event any two vertices of $P_n$ are joined by a path of $\CS$ of length at most $4\sqrt{n}$: each of the two descents just described costs at most $\sqrt{n}$, and their feet, which lie on the bottom geodesic at two times differing by at most $n^{1/4}+\sqrt{n} \le 2\sqrt{n}$, are joined along that geodesic by a path of that length.  This is~\eqref{it:slice_diameter}.

Let now $\kappa>0$ and let $K=K(\kappa)$ be the constant of \Cref{lem:lower_bound_distances}.  We apply that lemma with $r=6\sqrt{n}$, with $h_0=0$ (so that $\Gamma$ is the bottom geodesic of the slice, which is the line $\Z_{\ge0}\times\{0\}$, whence $h_\Gamma \equiv 0$) and with $t_n$ in the role of the parameter called $t_0$ there.  With probability at least $1-\kappa$, every path of $\CL$ of length less than $6\sqrt{n}$ started at a vertex $(i,0)$ with $|i-t_n| \le 12\sqrt{n}$ then visits only vertices of height less than $Kr^2=36Kn$, which is less than $n^{5/4}$ for $n$ large.  Suppose that this event and the two preceding ones all occur, and let $\gamma$ be a path of $\CS$ of length at most $4\sqrt{n}$ started at a vertex of $P_n$ and meeting the top geodesic, at a vertex $(t,Z_t)$ say.  Preceding $\gamma$ by the reversal of the descent of length at most $\sqrt{n}$ produced above turns it into a path of $\CL$ of length at most $5\sqrt{n}<6\sqrt{n}$ started at a vertex $(i,0)$ with $i \in [t_n,t_n+n^{1/4}+\sqrt{n}]$, so that $|i-t_n| \le 12\sqrt{n}$, and reaching the height $Z_t \ge n^{5/4}$.  This is impossible, and~\eqref{it:slice_confined} follows.

Finally, \Cref{prop:distortion_rectangle_lhpt}, applied to the segment from $(t_n,0)$ to $(t_n+n^{1/4},0)$ and with $\delta/2$ in place of $\delta$, gives that with probability tending to $1$ the distortion of $P_n$ computed with the distances of $\CL$ is at least $(\log n^{1/4})^{(1/4)-\delta/2}$, which is at least $(\log n)^{(1/4)-\delta}$ for $n$ large, the ratio of the two being $4^{-(1/4)+\delta/2}(\log n)^{\delta/2} \to \infty$.  On the three events above the distances of $\CL$ agree on $P_n$ with those of $\CS$, by \Cref{lem:slice_distances} applied with $\mathfrak{M}=\CL$, $t_0=t_n$, $m=n^{1/4}$ and $D=4\sqrt{n}$.  Indeed $D<t_n$, the slice has $Z_t>0$ up to the time $t_n+n^{1/4}+4\sqrt{n} \le n^{3/4}$ and its trees are not truncated, and in $\CL$ the time coordinate of a vertex changes by at most one along each step of a path (see \Cref{subsec:outline}).  Hence~\eqref{it:slice_distortion} holds as well on the intersection of the four events, whose probability is at least $1-\kappa-o(1)$.  As $\kappa>0$ was arbitrary and $E_n$ does not depend on it, $\p(E_n) \to 1$.
\end{proof}

\section{From the LHPT to the sphere}
\label{sec:LHPT to sphere}
The final step in the proof of \Cref{thm:main} is to transfer the distortion bound of \Cref{prop:slice_geometry} to finite triangulations of the sphere. To do so, we will make use of two results from \cite{cl2019fpp} regarding the absolute continuity properties of the law of triangulations of the sphere with respect to the laws of the UIPT and the LHPT.

We begin by recalling some terminology and notation from \cite{cl2019fpp}. We let $\T_{n,1}$ be the set of all (Type~I) triangulations of the $1$-gon with $n$ inner vertices, that is, of the planar maps with $n+1$ vertices all of whose faces are triangles with the exception of one distinguished face, the \emph{external} face, whose boundary is a cycle of length~$1$.  Such a map is rooted at that boundary loop, and its \emph{inner} vertices are the $n$ vertices which are not incident to it.  We let $\CT_n^{(1)}$ be a random variable chosen uniformly in $\T_{n,1}$, and denote its root by $\rho_n$.  The superscript $(1)$ distinguishes these triangulations of the $1$-gon from the triangulations $\CT_n$ of the sphere appearing in \Cref{thm:main}.  The two are related by the bijection recalled below, which matches $\T_{n,1}$ with the set of rooted triangulations of the sphere with $n+1$ vertices, so that it is $\CT_{n+1}$ and $\CT_n^{(1)}$ which correspond to one another.  We similarly let $\CT_\infty^{(1)}$ be the $n \to \infty$ local limit of the law of $\CT_n^{(1)}$ (this is the Type~I UIPT $\CT_\infty$ of \Cref{subsec:outline}, up to the modification of the map near the root given by the bijection described in the next paragraph).  We let $\ol{\CT}_n^{(1)}$ be the triangulation $\CT_n^{(1)}$ with a distinguished vertex $o_n$, which is selected uniformly at random from the set of vertices distinct from the root.  We then let $B_r^{\bullet}(\ol{\CT}_n^{(1)})$ and $B_r^{\bullet}(\CT_\infty^{(1)})$ be the hulls associated with these triangulations (where the marked point of the former is $o_n$ and the latter is $\infty$). The former is the union of the ball of radius $r$ around the root and all connected components of the complement of the ball that do not contain $o_n$ (if $o_n$ belongs to the ball, then we set $B_r^{\bullet}(\ol{\CT}_n^{(1)})=\ol{\CT}_n^{(1)}$) and the latter is defined analogously.

Triangulations of the $1$-gon are the same thing as rooted triangulations of the sphere, in the following precise sense (this is what will identify the law appearing in \Cref{thm:main}).  Let $n \ge 2$ and let $\mathfrak{T}$ be a triangulation of the sphere with $n+1$ vertices, rooted at an oriented edge with initial vertex $\rho$ and terminal vertex $v$.  Split that edge into two parallel edges $e_1$ and $e_2$, which then bound a $2$-gon face, and draw inside that face a loop $e_0$ at $\rho$; it cuts the $2$-gon into a $1$-gon bounded by $e_0$ and a triangle bounded by $e_0$, $e_1$ and $e_2$.  Declaring the $1$-gon to be the external face, and $e_0$ (oriented so that the external face lies to its right) to be the root edge, produces an element $\Phi(\mathfrak{T})$ of $\T_{n,1}$.  Conversely, let $\mathfrak{M} \in \T_{n,1}$ have boundary loop $e_0$ at $\rho$.  The face of $\mathfrak{M}$ lying on the inner side of $e_0$ is a triangle, whose two other edges $e_1$ and $e_2$ join $\rho$ to a vertex $v$; deleting $e_0$ merges the external $1$-gon with that triangle into a $2$-gon bounded by $e_1$ and $e_2$, and identifying $e_1$ with $e_2$ then produces a triangulation of the sphere with $n+1$ vertices, which we root at the resulting edge, oriented away from $\rho$ (formally, in the direction in which $e_1$ is traversed when the boundary of the above triangle is followed with the triangle on its left, starting from $e_0$).  The two constructions are inverse to each other, so that $\Phi$ is a bijection from the set of rooted triangulations of the sphere with $n+1$ vertices onto $\T_{n,1}$.

We say that a forest $\mathfrak{F}$ with a distinguished vertex is $(p,q,r)$-admissible if:
\begin{enumerate}[(i)]
\item\label{it:forest1} the forest consists of an ordered sequence $(\tau_1,\dots,\tau_q)$ of $q$ rooted plane trees,
\item\label{it:forest2} the maximal height of these trees is $r$,
\item\label{it:forest3} the total number of vertices of the forest at generation $r$ is $p$,
\item\label{it:forest4} the distinguished vertex has height $r$,
\item\label{it:forest5} the distinguished vertex is in $\tau_1$.
\end{enumerate}

We let $\F_{p,q,r}$ be the set of all $(p,q,r)$-admissible forests.  We let $\F_{p,q,r}'$ be the set of all pointed forests satisfying the same properties as the $(p,q,r)$-admissible forests, except that we do not require property~\eqref{it:forest5} saying that the distinguished vertex belongs to the first tree in the forest.  Finally, we let $\F_{p,q,r}''$ be the set of all non-pointed forests satisfying the same properties~\eqref{it:forest1}--\eqref{it:forest3} as $(p,q,r)$-admissible forests except that no special vertex is distinguished (thus properties~\eqref{it:forest4} and~\eqref{it:forest5} become irrelevant).  Only the last of these three sets is used below; we have recalled the other two, and with them properties~\eqref{it:forest4} and~\eqref{it:forest5}, because this is the form in which the results of \cite{cl2019fpp} are stated.

Much like the LHPT, the hull annulus $B_s^\bullet(\CT_\infty^{(1)}) \setminus B_r^\bullet(\CT_\infty^{(1)})$ in the UIPT can be encoded by a cycle of trees. This cycle of trees is called its \emph{skeleton}, and the triangulation $B_s^\bullet(\CT_\infty^{(1)}) \setminus B_r^\bullet(\CT_\infty^{(1)})$ can be reconstructed from it in the same way as for the LHPT.  In particular, the skeleton consists of $L_s$ trees indexed by the cycle $\partial B_s^\bullet(\CT_\infty^{(1)})$, the vertices of the forest at generation $j$ are in bijection with the vertices of the cycle $\partial B_{s-j}^\bullet(\CT_\infty^{(1)})$ (as in the LHPT, where the skeleton vertices at generation $i$ are in bijection with the vertices of the line $\{i\}\times\Z$) and the slots between consecutive branches are filled in by independent Boltzmann triangulations. See \cite[Section 2.2]{cl2019fpp} for the detailed definition of the skeleton in this case.

For $1 \le r < s$, we let $\CF_{r,s}^{(1)}$ be the skeleton of $B_s^\bullet(\CT_\infty^{(1)}) \setminus B_r^\bullet(\CT_\infty^{(1)})$.  We also let $\wt{\CF}_{r,s}^{(1)}$ be the non-pointed forest obtained by a random cyclic permutation of the trees of $\CF_{r,s}^{(1)}$ (so that the first tree in $\wt{\CF}_{r,s}^{(1)}$ is the tree of index $K$ in $\CF_{r,s}^{(1)}$, where $K$ is chosen uniformly at random in $\{1,\dots,L_s\}$ and $L_s$ is the boundary length of $\partial B_s^\bullet(\CT_\infty^{(1)})$ and also ``forgetting'' the distinguished vertex at generation $s-r$.  On the event $\{L_r = p\} \cap \{L_s = q\}$, $\wt{\CF}_{r,s}^{(1)}$ is a random element of the set $\F_{p,q,s-r}''$ introduced just above. For $0<a<1$, we let $N_r^{(a)}$ be uniformly distributed over $\{ \lfloor a r^2 \rfloor +1, \dots, \lfloor  a^{-1} r^2 \rfloor\}$ and we let $(\tau_i)_{i \ge 1}$ be a sequence of i.i.d.\ Galton--Watson trees with offspring distribution $\theta$ which is independent of $N_r^{(a)}$.  For each $i \ge 1$ and $j \ge 0$, we denote by $[\tau_i]_j$ the tree $\tau_i$ truncated at generation $j$.

\begin{lemma}[{\cite[Proposition 5]{cl2019fpp}}]
\label{lem:lhpt_to_uipt}
For each $0<a<1$, there exists a constant $C_1$ such that for every sufficiently large $r \in \N$ and every $s \ge r+1$, for every choice of $p,q \in \N$ with $a r^2 < p,q  \le a^{-1} r^2$, for every forest $\mathfrak{F} \in \F_{p,q,s-r}''$ we have
\[ \p( \wt{\CF}_{r,s}^{(1)}=\mathfrak{F}) \le C_1 \p(  ([\tau_1]_{s-r},\dots,[\tau_{N_r^{(a)}}]_{s-r}) = \mathfrak{F}).\]
\end{lemma}

We now use the above to transfer \Cref{prop:distortion_rectangle_lhpt} to an $\ell^2$ distortion lower bound for the UIPT.

\begin{proposition}\label{prop:distortion_uipt}
For every $0<\delta<1/4$, we have
\[\p(\Dis(B_n^{\bullet}(\CT_\infty^{(1)})) \ge (\log n)^{(1/4)-\delta}) \to 1\]
as $n \to \infty$.
\end{proposition}

\begin{proof}
\noindent\steplabel{1}{step:distuipt_1}\emph{Step 1. Choice of the parameter $a$.} Fix $\eps>0$, and assume for notational simplicity that $n$ is even, so that $n/2$ is an integer (otherwise replace $n/2$ by $\lfloor n/2\rfloor$ throughout). By \cite[Theorem 2]{cl2017peel}, we have that $L_n/n^2$ converges in distribution to the $\Gam(3/2)$ distribution. It follows that if $a>0$ is sufficiently small, then for all $n$ sufficiently large,
\[\p(A_1) \ge 1-\frac{\eps}{2} \quad\text{for}\quad A_1=\{an^2 \le L_n \le a^{-1}n^2\} \cap \{an^2 \le L_{n/2} \le a^{-1}n^2\}.\]
We fix this choice of $a$ for the rest of the proof.

\noindent\steplabel{2}{step:distuipt_2}\emph{Step 2. The slice.} Let $\CS$ be the slice spanned by the first $n$ trees of $\wt{\CF}_{n/2,n}^{(1)}$, that is, the sub-triangulation of the annulus $B_n^\bullet(\CT_\infty^{(1)}) \setminus B_{n/2}^\bullet(\CT_\infty^{(1)})$ lying between the bottom-most geodesics started at the two ends of the corresponding arc of $\partial B_n^\bullet(\CT_\infty^{(1)})$, built from those $n$ trees and from the Boltzmann triangulations filling their slots. Its initial height is $n=(\sqrt{n})^2$, which is the height for which \Cref{prop:slice_geometry} is stated with $\sqrt{n}$ in place of $n$.  Its vertex at time $t$ and height $h$ lies on the cycle $\partial B_{n-t}^\bullet(\CT_\infty^{(1)})$, and its height $Z_t$ is the population at generation $t$ of the forest formed by those $n$ trees.  On $A_1$ we have $L_n \ge an^2 \ge n$ for $n$ large, so that $\CS$ is then well defined.

Being a deterministic function of those $n$ trees and of the fillings of the slots which they index, which are i.i.d.\ Boltzmann triangulations independent of the forest both here and in the LHPT, the slice can be compared with its LHPT counterpart.  Summing the bound of \Cref{lem:lhpt_to_uipt} over the forests with the lemma applied with $r=n/2$, $s=n$ and with $a/4$ in place of $a$ (every pair $p,q \in [an^2,a^{-1}n^2]$ satisfies $(a/4)(n/2)^2<p,q\le(a/4)^{-1}(n/2)^2$ and on $A_1$ the forest belongs to $\F_{p,q,n/2}''$ for such a pair) and then integrating over the fillings shows that for every set $\s$ of slices,
\begin{equation}\label{eq:slice_comparison}
\p(\{\CS \in \s\} \cap A_1) \le C_1\,\p(\CS' \in \s),
\end{equation}
where $\CS'$ is the slice spanned by the first $n$ trees of $([\tau_1]_{n/2},\dots,[\tau_{N}]_{n/2})$, $N=N_{n/2}^{(a/4)}$, with independent Boltzmann fillings.  Here $n<(a/4)(n/2)^2<N$ for $n$ large, so that $\CS'$ is well defined, and $\CS'$ is the slice $\CS^{(n)}$ of initial height $n$ of an LHPT, up to the truncation of its trees at generation $n/2$.

\noindent\steplabel{3}{step:distuipt_3}\emph{Step 3. The set $\mathbf{G}_n$.} Let $E$ be the event of \Cref{prop:slice_geometry}, applied with $\sqrt{n}$ in place of $n$ and with $\delta/2$ in place of $\delta$.  Note that $\p(\CS' \notin E) \to 0$.  Let $\mathbf{G}_n$ be the set of pairs consisting of a hull of radius $n$ and of a marked vertex on its boundary cycle for which the slice determined by the pair as in \Cref{step:distuipt_2} belongs to $E$.  By~\eqref{eq:slice_comparison},
\[\p\big((B_n^\bullet(\CT_\infty^{(1)}),w_n) \notin \mathbf{G}_n\big) \le \p(A_1^c)+\p(\{\CS \notin E\} \cap A_1) \le \frac{\eps}{2}+C_1\,\p(\CS' \notin E)\]
for $n$ large.  For each fixed $\eps>0$ the second term tends to $0$. The left-hand side thus tends to $0$ since $\eps>0$ was arbitrary and $E$ does not depend on $\eps$.

\noindent\steplabel{4}{step:distuipt_4}\emph{Step 4. Conclusion.} Let $(\mathfrak{H},w) \in \mathbf{G}_n$, let $\mathfrak{S}$ be the slice which this pair determines, let $\rho$ be the root of $\mathfrak{H}$, and let $\mathfrak{M}$ be a triangulation containing $\mathfrak{H}$ as a sub-triangulation and in which the distances from $\rho$ to the vertices of $\mathfrak{H}$ are those of $\mathfrak{H}$.  Then $\mathfrak{S}$ is a sub-triangulation of $\mathfrak{M}$, and its vertex at time $t$ and height $h$ lies on the cycle $\partial B_{n-t}^\bullet(\mathfrak{H})$, hence is at distance $n-t$ from $\rho$ in $\mathfrak{H}$ and so also in $\mathfrak{M}$. Consequently $|d_{\mathrm{gr}}^{\mathfrak{M}}(\rho,u)-d_{\mathrm{gr}}^{\mathfrak{M}}(\rho,v)| \le d_{\mathrm{gr}}^{\mathfrak{M}}(u,v)$ for all vertices $u,v$ of $\mathfrak{S}$ carrying time coordinates.  Since $\mathfrak{S} \in E$, \Cref{lem:slice_distances} applies, with $t_0=\tfrac12 n^{3/8}$, $m=n^{1/8}$ and $D=4n^{1/4}$, the trees of $\mathfrak{S}$ being truncated at the generation $n/2>n^{3/8}$.  With $P$ the parallelogram of \Cref{prop:slice_geometry} we have that
\[\Dis(\mathfrak{M}) \ge \Dis(P) \ge (\log \sqrt{n})^{(1/4)-\delta/2} \ge (\log n)^{(1/4)-\delta}\]
for $n$ large, the distortion of $P$ computed with the distances of $\mathfrak{S}$, the second inequality being property~\eqref{it:slice_distortion} of $E$, and the third holding because the ratio of the last two quantities is $2^{-(1/4)+\delta/2}(\log n)^{\delta/2} \to \infty$.  Together with \Cref{step:distuipt_3}, and taking $\mathfrak{M}=\mathfrak{H}=B_n^\bullet(\CT_\infty^{(1)})$, this proves both assertions of the proposition.
\end{proof}

All that remains now is to show how to go from the UIPT to the sphere. Let $\C_{1,r}$ be the set of all triangulations of the cylinder of height $r$ and with bottom cycle of length $1$.  For $\mathfrak{T} \in \C_{1,r}$, we let $N(\mathfrak{T})$ be such that the number of vertices in $\mathfrak{T}$ is given by $N(\mathfrak{T})+1$. We then have the following.

\begin{lemma}[{\cite[Lemma 22]{cl2019fpp}}]
\label{lem:uipt_to_sphere}
There exists a constant $\ol{c}$ such that, for every $n \ge 1$, for every $r \ge 1$ and $\mathfrak{T} \in \C_{1,r}$ such that $n > N(\mathfrak{T})$,
\[ \p( B_r^\bullet(\ol{\CT}_n^{(1)}) = \mathfrak{T}) \le \ol{c} \left( \frac{n}{n-N(\mathfrak{T})} \right)^{3/2} \p( B_r^\bullet(\CT_\infty^{(1)}) = \mathfrak{T}).\]
\end{lemma}

\begin{proof}[Proof of \Cref{thm:main}]
\noindent\steplabel{1}{step:main_1}\emph{Step 1. Two events of high probability.} Fix $\eps>0$ and $\delta>0$.  We assume without loss of generality that $\delta<1/4$. By \cite[Theorem 1 (ii)]{miermont2006invprinc}, we have that $d_{\mathrm{gr}}(\rho_n,o_n)/n^{1/4}$ converges in distribution to a positive, non-degenerate random variable $Z$ (recall that $\rho_n$ denotes the root of $\ol{\CT}_n^{(1)}$, and $o_n$ is a vertex picked uniformly at random from the set of vertices distinct from the root). It follows that there is an $a>0$ such that for all $n$ sufficiently large,
\[\p(A_1) \ge 1-\frac{\eps}{2} \quad\text{for}\quad A_1 = \{ d_{\mathrm{gr}}(\rho_n,o_n) \ge an^{1/4}\}.\]
By the convergence of the rescaled profile of distances to $o_n$ to a measure which assigns positive weight to any neighborhood of $0$, there is a $0<b<1$ such that for all sufficiently large $n$,
\[\p(A_2) \ge 1-\frac{\eps}{2} \quad\text{for}\quad A_2 = \left\{ \left|B_{\frac{an^{1/4}}{2}}(o_n) \right| \ge bn \right\}. \]
(here $B_r(o_n)$ denotes the ball of radius $r$ around $o_n$). See \cite[Equation (62)]{cl2019fpp} and the discussion surrounding it.

\noindent\steplabel{2}{step:main_2}\emph{Step 2. The hull around the root is not too large.} On the event that both $A_1$ and $A_2$ occur, we have that $B_{n^{1/8}}^\bullet(\ol{\CT}_n^{(1)})$ contains at most $(1-b)n+1$ vertices for all $n$ sufficiently large, since $B_{(an^{1/4})/2}(o_n)$ is connected and disjoint from the ball of radius $n^{1/8}$ around the root.

\noindent\steplabel{3}{step:main_3}\emph{Step 3. Transfer from the UIPT to the sphere.} Write $m=n^{1/8}$. Now for any triangulation $\mathfrak{T}$ with $(1-b)n+1$ or fewer vertices, we have $N(\mathfrak{T}) \le (1-b)n$, hence $n-N(\mathfrak{T}) \ge bn$, and therefore
\[\p( B_{m}^\bullet(\ol{\CT}_n^{(1)}) = \mathfrak{T}) \le \ol{c} \left( \frac{1}{b} \right)^{3/2} \p( B_{m}^\bullet(\CT_\infty^{(1)}) = \mathfrak{T})\]
by \Cref{lem:uipt_to_sphere}. Let $\mathbf{G}_m$ be the set of pairs from the proof of \Cref{prop:distortion_uipt}, applied at radius $m$ and with $\delta/2$ in place of $\delta$, and let $w$ denote the vertex of the boundary cycle of the hull of radius $m$ which is selected by the uniform random rotation of its skeleton.  Conditionally on the hull, $w$ is uniform on that cycle in both models, so the pointwise comparison of hull laws extends to the pair consisting of the hull and $w$.  On the event $A_1 \cap A_2$ the hull $B_m^\bullet(\ol\CT_n^{(1)})$ takes values among the triangulations with at most $(1-b)n+1$ vertices, by \Cref{step:main_2}, and summing the resulting bound over the pairs which do not belong to $\mathbf{G}_m$ therefore gives
\[\p\big(\{(B_m^\bullet(\ol\CT_n^{(1)}),w) \notin \mathbf{G}_m\} \cap A_1 \cap A_2\big) \le \ol{c}\left(\frac{1}{b}\right)^{3/2}\p\big((B_m^\bullet(\CT_\infty^{(1)}),w) \notin \mathbf{G}_m\big).\]
The right-hand side tends to $0$ as $n \to \infty$ by \Cref{prop:distortion_uipt}, so that the left-hand side is at most $\eps$ for all $n$ sufficiently large.

\noindent\steplabel{4}{step:main_4}\emph{Step 4. From the hull to the whole triangulation.} The triangulation $\ol\CT_n^{(1)}$ contains its hull $B_m^\bullet(\ol\CT_n^{(1)})$ as a sub-triangulation, and the distances from the root to the vertices of that hull are the same in $\ol\CT_n^{(1)}$ as in the hull itself.  \Cref{prop:distortion_uipt} therefore gives that on the event $\{(B_m^\bullet(\ol\CT_n^{(1)}),w) \in \mathbf{G}_m\}$,
\[\Dis(\ol\CT_n^{(1)}) \ge (\log m)^{(1/4)-\delta/2} \ge (\log n)^{(1/4)-\delta} \quad\text{(recall $m=n^{1/8}$)}\]
for $n$ large.  By \Cref{step:main_1,step:main_2,step:main_3}, the event above intersected with $A_1 \cap A_2$ has probability at least $1-2\eps$ for all $n$ sufficiently large.

\noindent\steplabel{5}{step:main_5}\emph{Step 5. From the $1$-gon to the sphere.} Recall that the statement of \Cref{thm:main} is for a uniformly random rooted Type~I triangulation $\CT_N$ of the sphere with $N$ vertices. It is thus left to transfer the distortion lower bound from $\ol\CT_n^{(1)}$ to the case of $\CT_N$.  Note that the marked points associated with $\ol\CT_n^{(1)}$ do not affect the underlying graph, so that $\Dis(\ol\CT_n^{(1)})=\Dis(\CT_n^{(1)})$, and $\CT_n^{(1)}$ is uniform on $\T_{n,1}$.  The distance-preserving bijection $\Phi$ between the set of rooted triangulations of the sphere with $n+1$ vertices and $\T_{n,1}$ described at the beginning of this section combined with the previous paragraph yields
\[\p\!\big(\Dis(\CT_{n+1}) \ge (\log n)^{(1/4)-\delta}\big) \ge 1-2\eps\]
for all $n$ sufficiently large.  Now let $\delta' \in (0,1/4)$ be given and apply the above with $\delta=\delta'/2$ and with $N-1$ in place of $n$.  Then we obtain
\[\p\!\big(\Dis(\CT_N) \ge (\log N)^{(1/4)-\delta'}\big) \ge 1-2\eps\]
for all $N$ sufficiently large.  Since $\eps$ was arbitrary, the proof of the theorem is complete.
\end{proof}

\bibliographystyle{abbrv}
\bibliography{bibliography}

\end{document}